\documentclass[12pt]{amsart}

\usepackage{amsmath,amssymb,amsxtra,anysize,graphicx,xcolor,tikz,tikz-cd,combelow,pdflscape, mathtools}

\newcommand{\C}{\mathbb{C}}
\newcommand{\Q}{\mathbb{Q}}
\newcommand{\U}{\mathcal{U}}
\newcommand{\Z}{\mathbb{Z}}
\newcommand{\sq}{\square}

\newcommand{\fg}{\mathfrak{g}}
\newcommand{\fgl}{\mathfrak{gl}}
\newcommand{\fsl}{\mathfrak{sl}}

\newcommand{\N}{\mathbb{N}}
\newcommand{\n}{\mathbf{n}}
\newcommand{\m}{\mathbf{m}}

\newcommand{\CZ}{\mathcal{Z}}
\newcommand{\CR}{\mathcal{R}}
\newcommand{\CP}{\mathcal{P}}

\newtheorem{theorem}{Theorem}[section]
\newtheorem{remark}[theorem]{Remark}
\newtheorem{problem}[theorem]{Problem}
\newtheorem{proposition}[theorem]{Proposition}%[chapter]
\newtheorem{corollary}[theorem]{Corollary}%[chapter]
\newtheorem{lemma}[theorem]{Lemma}%[chapter]

\newtheorem{example}[theorem]{Example}

\DeclareMathOperator{\cont}{cont}

\newcommand{\xx}{\mathbf{x}}

\newcommand{\hc}{\mathit{hc}}
\newcommand{\ch}{\mathit{ch}}
\newcommand{\Tr}{\mathrm{Tr}}

\newcommand{\Id}{\mathrm{Id}}

\newcommand{\Rep}{\mathit{Rep}}
\newcommand{\SF}{\mathit{Sym}}

\newcommand{\K}{\mathcal{K}}

\newcommand{\ba}{\mathbf{a}}

\newcommand{\e}{\mathbf{\varepsilon}}

\title[Cyclotomic expansions via interpolation Macdonald polynomials]{Cyclotomic expansions for $\fgl_N$ link invariants \\[2mm] via interpolation Macdonald polynomials}
\author{Anna Beliakova}
\address{Institut f\"ur Mathematik\\
Universit\"at Z\"urich\\
Winterthurerstrasse 190\\
CH-8057 Z\"urich\\ Switzerland}
\email{anna@math.uzh.ch}
\author{Eugene Gorsky}
\address{Department of Matematics\\ University of California, Davis \\ One Shields Avenue, Davis CA 95616, USA}
\email{egorskiy@math.ucdavis.edu}

\begin{document}
\maketitle

%\tableofcontents

\begin{abstract}
In this paper we construct a new basis
for the cyclotomic completion of the center of
the  quantum $\fgl_N$ in terms of the interpolation Macdonald polynomials. Then we use
a result of Okounkov to provide a dual  basis with respect to the quantum Killing form (or Hopf pairing).
The main applications  are:
1)  cyclotomic 
expansions for the  $\fgl_N$ Reshetikhin--Turaev link invariants 
and the universal $\fgl_N$ knot invariant;
2) an explicit construction of the unified $\fgl_N$
invariants for integral homology 3-spheres using universal Kirby colors.
%obtained by knot surgeries. 
These results generalize those  of Habiro for $\fsl_2$. In addition,
we give a simple proof of the fact that the universal $\fgl_N$  invariant of any evenly framed link 
and the universal
$\fsl_N$ invariant  of any $0$-framed
algebraically split link are $\Gamma$-invariant, where $\Gamma=Y/2Y$ with 
 the root lattice $Y$. 
 %and Habiro--Le for simple Lie algebras.
\end{abstract}

\section{Introduction}
In a series of papers \cite{BBlL, BBL,H} Habiro, the first author et al. defined  {\it unified} invariants of homology 3-spheres that belong to the Habiro ring and dominate Witten--Reshetikhin--Turaev (WRT) invariants. 
Unified invariants provide an important tool to study
structural properties of the WRT invariants.
In \cite{BCL, BL} they were used to prove 
integrality of the  $\fsl_2$ WRT invariants for all 3-manifolds at all roots of unity.
\vspace{2mm}

The theory of unified invariants for $\fsl_2$ is based on  {\it
cyclotomic expansions}
for the colored Jones polynomial and for the universal knot invariant
 constructed as follows.
 Given a framed oriented link
$L$ in the 3-sphere, we open its components to obtain a bottom tangle $T$, 
 presented by  a diagram $D$ (see Figure 1).  For
 a ribbon Hopf algebra $U_q\fg$, the {\it universal} link invariant $J_L(\fg;q)$  is obtained by 
spliting $D$ intro elementary pieces: crossings, caps and cups and then by 
 associating to them
 $R^{\pm 1}$-matrices, and  pivotal elements,
 respectively. 
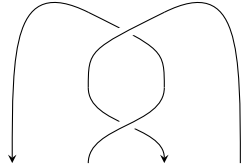
\begin{figure}[ht!]
\begin{tikzpicture}
\draw (1,1).. controls (1,1.5) and (1,1.5) .. (0,2);
\draw [line width=5,white] (0,1).. controls (0,1.5) and (0,1.5).. (1,2);
\draw (0,1).. controls (0,1.5) and (0,1.5).. (1,2);
\draw[stealth-] (1,0).. controls (1,0.5) and (0,0.5) .. (0,1);
\draw [line width=5,white] (0,0).. controls (0,0.5) and (1,0.5)..(1,1);
\draw (0,0).. controls (0,0.5) and (1,0.5)..(1,1);
\draw [stealth-] (-1,0) .. controls (-1,1.5) and (-1,2.5) .. (0,2);
\draw (2,0) .. controls (2,1.5) and (2,2.5) .. (1,2);
\end{tikzpicture}
\caption{An example of the clasp bottom tangle}
\label{fig:clasp}
\end{figure}

For a knot $K$, 
$J_K(\fg;q)$  belongs to (some completion of) the center $\CZ(U_q\fg)$.
In the easiest case $\fg=\fsl_2$, the center is generated by the Casimir $C$.
For a $0$-framed knot $K$, Habiro showed that
there are  coefficients $a_m(K)\in \Z[q^{\pm 1}]$
%the Habiro series  \cite{H} have the form
such that
\begin{align} 
\label{cyclotomic}
  J_K( \fsl_2; q) = 
  \sum^\infty_{m=0} a_m(K)\, \sigma_m \;  \quad \text{with} \;\quad \sigma_m=\prod_{i=1}^m\left(C^2 -(q^i+q^{-i}+2) \right)
 .
\end{align}
% is the Casimir of the quantum group $U_q(\fsl_2)$.
 Replacing $C^2$  in \eqref{cyclotomic}  by 
 its value $q^{n}+q^{-n}+2$
 on  the $n$-dimensional irreducible representation $V_{n-1}$,
 we get
the $n$-colored Jones polynomial of $K$ (normalized to 1 for the unknot)
 \begin{equation}\label{Jones}
  J_K(V_{n-1},q) = \sum_{m=0}^\infty (-1)^m q^{- \frac{m(m+1)}{2}} a_m(K) \,
  (q^{1+n};q)_m ( q^{1-n}; q)_m
\end{equation}
where $(a;q)_m=(1-a)(1-aq)\dots(1-aq^{m-1})$.
 Equation \eqref{Jones} is known as a cyclotomic expansion of the colored Jones polynomial.
 Thus, Habiro's series \eqref{cyclotomic}
 dominates all colored Jones polynomials of $K$.
   To prove the fact that $J_K( \fsl_2;q)$ belongs to the 
 even part of  $\CZ(U_q\fsl_2)$, generated by 
$C^2$, Habiro used the whole power of the theory of bottom tangles developed in \cite{Hbottom}.
\vspace{2mm}
   
%   certain cyclotomic completion
%$$ {\lim\limits_{\overleftarrow{\hspace{2mm}n%\hspace{2mm}}}}\;   
% \frac{\Z[q^{\pm 1}][C^2]}{(\sigma_n)}$$
%  To prove the fact Habiro uses 

In this paper we  give a simple proof for the ``evenness" of the universal invariant of algebraically split links for
all quantum groups of type $A$.
%For clarity of exposition we focus on $U_q(\fgl_N)$.  
Recall that $U_q\fg$ has a natural action of a finite group 
$\Gamma=Y/2Y$ where $Y$ is the root lattice of $\fg$. 
 For $\fg=\fgl_N$,  $\Gamma=\Z_2^N$ and for
 $\fg=\fsl_N$,  $\Gamma=\Z_2^{N-1}$. 

\begin{theorem}
\label{thm:intro even}
 The universal $\fgl_N$  invariant of any evenly framed link is $\Gamma$-invariant. 
 The universal $\fsl_N$  invariant of any  0-framed algebraically split link is $\Gamma$-invariant. 
\end{theorem}
 
%{ \color{red}   I do not know how to compare endomorphisms acting on different spaces, so our proof 
%need to be improved}
%Note that for even $N$ we can drop the condition on framing (compare Theorem \ref{thm: zeta RT}).

The quantum group $U_q\fgl_N$ admits a finite dimensional irreducible representation
$V(\lambda)$   with highest weight $v^{\lambda}$
for any partition $\lambda=(\lambda_1\geq \dots\geq \lambda_N)$ with $N$ parts and $v^2=q$.
To  prove  Theorem \ref{thm:intro even} we extend the  Reshetikhin-Turaev invariants to tangles colored with representations $L(\zeta)\otimes V(\lambda)$
where   $L(\zeta)$ is a one-dimensional representation of $U_q\fgl_N$ for $\zeta\in \Gamma$. 
Then the claim follows from the comparison of  the $\fgl_N$ Reshetikhin-Turaev link invariants  colored with 
 $L(\zeta)\otimes V(\lambda)$ and  $V(\lambda)$. 
\vspace{2mm}

The next main result of the paper establishes an explicit basis in the $\Gamma$-invariant part of the center $\CZ$ of  $U_q\fgl_N$. 
It generalizes Habiro's basis
$\{\sigma_m \,|\, m\in \mathbb N\}$
for the even part of  $\CZ(U_q\fsl_2)$. 

\begin{theorem}
\label{thm: intro basis} 
There exists a family of central elements $\sigma_{\lambda}\in \CZ$ labeled by partitions $\lambda$ with at most $N$ parts  with the following properties:
\begin{itemize}
\item[(a)] $\sigma_{\lambda}$ is $\Gamma$-invariant and annihilates  $L(\zeta)\otimes V(\mu)$ for all $\zeta\in \Gamma$ and partitions $\mu$ with at most $N$ parts not containing $\lambda$; %$\lambda\nsubseteq\mu$.
\item[(b)] $\sigma_{\lambda}$ does not annihilate $V(\lambda)$ and acts on it by an explicit scalar (see Theorem \ref{thm:basis}).
\end{itemize}
\end{theorem}
 
The proof uses the theory of interpolation Macdonald polynomials developed in \cite{Knop,KS,Ok,Ok2,Ok3,Ols,Sahi}. This theory allows one to reconstruct a symmetric function $f(x_1,\ldots,x_N)$ from its values at special points $x_i=q^{-\mu_i-N+i}$ where $\mu$ is an arbitrary partition with at most 
$N$ parts. The connection between the center of $U_q\fgl_N$ and symmetric functions goes through the quantum Harish-Chandra isomorphism, and we interpret $f(q^{-\mu_1-N+1}, \dots, q^{-\mu_N})$ as the scalar by which the element of the center $f$ acts on the irreducible representation $V({\mu})$. Interpolation Macdonald polynomials  then correspond to a natural  basis in the center of $U_q\fgl_N$.

The polynomials $\sigma_{\lambda}$ yield a basis in the $\Gamma$-invariant parts of both the center $\CZ$ and its completion (a function in the completion is determined by its values on all finite-dimensional representations).  We use a formula of Okounkov \cite{Ok} to give explicit expansion of a given central element $z$ in the basis $\sigma_{\lambda}$ in terms of the scalars by which $z$ acts on all finite-dimensional representations $V({\lambda})$.  
This leads
to an  expansion of the universal knot invariant in the basis $\sigma_{\lambda}$, where the coefficients are related to Reshetikhin-Turaev invariants of the same knot colored by $V({\mu})$ via an explicit triangular matrix $(d_{\lambda, \mu})$
which does not depend on the knot.
%To do so, we construct  a second basis
%$P_\mu$, such that
%  where 

\begin{theorem}\label{thm:main}
For any   evenly framed knot $K$,
there exist Laurent polynomials
$a_\lambda(K)\in \Z[q,q^{-1}]$  such that
 the universal  invariant of $K$ has the following expansion:
\begin{equation}
\label{eq:cycl}
J_K(\fgl_N;q)=\sum_{\lambda}
a_\lambda(K)\, \sigma_{\lambda}\ .
\end{equation}
Moreover, the
coefficients $a_\lambda(K)$ can be computed 
in terms of the Reshetikhin-Turaev invariants  as follows:
$$a_\lambda(K) =
%%%\langle P_\lambda, J_K(\fgl_N;q)\rangle
%=
\sum_{\mu\subset \lambda}
{d_{\lambda, \mu}(q^{-1})}\; J_K(V(\mu),q)$$
%\end{corollary}
where the coefficients $d_{\lambda, \mu}(q)$ are defined in Theorem \ref{thm: dlm}. \end{theorem}

We prove Theorem \ref{thm:main} as Proposition \ref{pro:main}.
We would like to emphasize that the fact that $a_\lambda(K)$ are Laurent polynomials in $q$ is highly nontrivial. Indeed, 
we have computed the tables of coefficients $d_{\lambda, \mu}(q)$  for $\fgl_2$ in Section \ref{sec: interpolation tables}
and these are complicated rational functions, so a priori $a_\lambda(K)$ are  rational functions as well. Theorem \ref{thm:main} thus encodes certain divisibility properties for the linear combinations of colored invariants of $K$. We refer to Section \ref{sec: figure 8} for the explicit computation of the coefficients $a_{\lambda}(K)$ for the figure eight knot.

%This is a generalization to $\fgl_N$ of the Habiro famous cyclotomic expansion of the colored Jones polynomial given in \eqref{cyclotomic}.
We  call \eqref{eq:cycl} a {\it cyclotomic expansion}
of the universal $\fgl_N$ knot invariant. 
The name {\it cyclotomic} is justified by the fact that 
\eqref{eq:cycl} has well-defined evaluations at any root of unity by Lemma \ref{lem: divisible} below.
Note that 
for $N=2$ and a $0$-framed knot, our expansion does not 
coincide with that of Habiro, simply because 
if an element $z\in U_q\fgl_2$ is
central and $\Gamma$-invariant, it
does not imply $z$ has a decomposition in even powers of the Casimir.
Therefore, our cyclotomic expansion is rather a generalization of $F_\infty$ in \cite{Sonny} or \cite[eq.(3.14)]{BH}, both having interesting application in the theory of non semisimple invariants of links and 3-manifolds.

%{\color{red} Would be great to compute this expansion 
%at least for trefoil and N=2, if we want to go for a good journal}

\vspace{2mm}

% Theorem \ref{thm:main} plays a crucial role in our
%construction of the unified
%invariants.
%for integral homology 3-spheres obtained by 
% knot surgeries.
%Let $r\in \CZ$ be the ribbon element.
For our next application, assume $M$ is an integral homology 3-sphere obtained by $\e$-surgery
on an $\ell$-component algebraically split $0$-framed link $L$ with $\e\in \{\pm 1\}^\ell$.
 Following Habiro--Le, we define 
an $\fgl_N$ unified invariant $I(M)$ 
 as  
$$I(M)= 
\langle \, r^{\otimes \e},J_L(\fgl_N;q)\, \rangle  \ $$
where $r$ is the $\fgl_N$
ribbon element and $\langle \cdot,\cdot\rangle$ is the Hopf pairing. 
%Now recall that the Habiro ring is defined as a 
%For simple Lie algebras Habiro--Le 
In the case of $\fsl_N$ Habiro--Le proved \cite{HL} that 
the unified invariant belongs to a 
cyclotomic completition of the polynomial ring
$$ \widehat{\Z[q]}:=
{\lim\limits_{\overleftarrow{\hspace{2mm}n\hspace{2mm}}}}\;   
\frac{\Z[q]}{((q;q)_n)} \ $$
known as {\it Habiro ring}.
Using interpolation, we are able to express $I(M)$
in terms of  special linear combinations of Reshetikhin--Turaev invariants of $L$, called {\it Kirby colors}.
%extend this result to $\fgl_N$. 
For this we diagonalize the
Hopf pairing, i.e. find a basis
$P_\mu$ that is orthonormal to $\sigma_{\lambda}$ and
 orthogonal to $V(\lambda)$
with respect to the Hopf pairing.
% $\langle\cdot, \cdot\rangle$.
%$\langle P'_\mu,V(\lambda)\rangle=
%c_\lambda\delta_{\mu, \lambda}$ for
%some $c_\lambda \in \Z[v^{\pm 1}]$. 
%An $\fsl_2$ analog of the basis $P_\mu$ was constructed   by Habiro in \cite{H}.
This allows us to give  explicit formulas for  the universal
Kirby colors $\omega_\pm$ (see \eqref{eq:Kirby})
in the basis  $P_\mu$
%({\color{red}
%could you please check that we really need P' here, Lemma 9.1, P' are not 
%orthogonal to $V(\lambda)$
and to prove the 
following result.
\begin{theorem}\label{thm:main2}
The unified invariant 
$$I(M)=J_L(\omega_{\epsilon_1}, \dots, \omega_{\epsilon_\ell}) \;\in \;\widehat{\Z[q]} $$
 belongs to the Habiro ring
 and dominates $\fgl_N$ WRT invariants
of $M_\pm$ at all roots of unity.  Moreover, $I(M)$ is equal
to the $\fsl_N$
Habiro--Le invariant of $M_\pm$. 
\end{theorem}

To prove  
%Theorem \ref{thm:main2} 
%we  show
that $I(M)$ is equal to the $\fsl_N$ Habiro--Le invariant we
show   the equality of the universal $\fgl_N$ and $\fsl_N$ invariants for $0$-framed algebraically split links, and the 
fact that the $\fgl_N$ and $\fsl_N$ twist forms $x \mapsto \langle r^{\pm 1}, x \rangle$ on them
coincide. 
 %by comparing their evaluations 
%at roots of unity. For $I^{\text{HL}}(M)$ they are known
%to coincide with the WRT invariants of $M$.
%
It follows that $I(M)$ belongs to the Habiro ring.
Then we establish invariance of Kirby colors $\omega_\pm$
under Hoste moves (a version of Fenn--Rourke moves
between algebraically split links)
in Lemma \ref{lem:Kirby}, and finally deduce   the
equality $I(M)=J_L(\omega_{\epsilon_1}, \dots, \omega_{\epsilon_\ell})$.

The main advantage of Theorem \ref{thm:main2}
compared to Habiro--Le approach is the interpretation of $I(M)$ as the Reshetikhin--Turaev  invariant 
of $L$ colored by
 $\omega_\e$. This  leads to various
 striking divisibility results and allows us to extend our cyclotomic expansion to links.

\begin{corollary} \label{cor:int_link1}
Given an $\ell$ component algebraically split $0$-framed link $L$, then
for all but finitely many
partitions $\lambda_i$ with $1\leq i\leq \ell$, 
there exist positive integers $n=n(\lambda_i, N)$, such that 
 $$
 J_L(P'_{\lambda_1}, \dots , P'_{\lambda_\ell})
\in (q;q)_n\,  \Z[q,q^{-1}]\  $$
where $P'_\lambda= v^{|\lambda|}\, \dim_q V(\lambda) \, P_\lambda$ is a scalar multiple of $P_\lambda$. 
\end{corollary}
This is a generalization of  the famous integrability theorem in \cite[Thm. 8.2]{H}. 
The authors do not know
any direct proof of Corollary \ref{cor:int_link1}
 without using the 
theory of unified invariants.
Based on Corollary \ref{cor:int_link} we obtain a 
cyclotomic expansion for the Reshetikhin-Turaev invariants of $L$:
\begin{equation}
J_L(\lambda_1, \dots, \lambda_\ell)= v^{\sum_i |\lambda_i|}\sum_{\mu_i\subset \lambda_i}
\prod^l_{j=1} c_{\lambda_j,\mu_j}(q^{-1})\,
J_L(P'_{\mu_1}, \dots, P'_{\mu_\ell})
\end{equation}
where the matrix $
\left[c_{\lambda,\mu}(q)\right]_{\lambda,\mu}:=\left[F_{\lambda}(q^{-\mu_i-N+i})\right]_{\lambda,\mu}
$ is the inverse of $\left[d_{\lambda,\mu}(q)\right]_{\lambda,\mu}$. 
%by Theorem \ref{thm: dlm} below.
This generalizes equation $(8.2)$ in \cite{H}.

In addition,
in the case of knot surgeries 
we give a direct proof
of the fact that
$$I(M_\pm)=J_L(\omega_\pm) \;\in \;\widehat{\Z[v]}$$
by using our cyclotomic expansion and the interpolation theory.

\vspace{2mm}

Finally, we would like to comment on potential ideas for categorification of these results. The ring of symmetric polynomials in $N$ variables is naturally categorified by the category of annular $\fgl_N$-webs, with morphisms given by annular foams \cite{BPW,QR,RW,GW,GHW}. By the work of the second author and Wedrich \cite{GW}, one can interpret it as a symmetric monoidal Karoubian category generated by one object $E$ corresponding to a single essential circle. The symmetric polynomials are then categorified by the Schur functors of $E$.

We expect the categorified interpolation polynomials to correspond to interpolation Macdonald polynomials where $q$ plays the role of quantum grading and $t$ of  the homological grading (after some change of variables). We recall the general definitions and properties of these polynomials from \cite{Ok} in  Appendix. The key obstacle for categorification of interpolation polynomials is that they are not homogeneous. Therefore one needs to enrich the category and allow additional morphisms between $E$ and identity.  

On the other hand, the conjectures of the second author, Negu\cb{t} and Rasmussen (\cite{GNR}, see \cite{GH,GHW} for further discussions) relate a version of the annular category to the derived category of the Hilbert scheme of points on the plane. The interpolation Macdonald polynomials appear in that context as well \cite{CNO}.

\vspace{2mm}

The paper is organized as follows. After recalling the 
definitions, we compare the Reshetikhin--Turaev invariants of tangles colored by $V(\lambda)$
and  $L(\zeta)\otimes V(\lambda)$
in Section \ref{sec:repribbon}.
In the next two sections we summarize  known results about
the center of $U_q\fgl_N$, define its completion
 and prove Theorem \ref{thm:intro even} in Section \ref{sec:HU}.
The remaining results
%n we summarize some facts about interpolation and use 
are proven in Sections \ref{sec:results}, \ref{sec:IM} assuming some facts about interpolation. 

%{\color{red}expand this description about the last sections}
In
the last sections we develop the theory of the interpolation
Macdonald polynomials, starting from the one variable case. We define multi-variable interpolation polynomials, state and prove their properties in Section \ref{sec: def polynomials}. Next, we solve the interpolation problem in two ways, one using the approach of Okounkov 
 (Theorem \ref{thm: dlm}), and another using Hopf pairing (see \eqref{eq: dlm from hopf}). We study divisibility of $F_{\lambda}(q^{a_1},\ldots,q^{a_n})$
 by quantum factorials
  %and their relation to the Habiro ring 
 in Section \ref{sec: divisibility}
 (see Lemma \ref{lem: divisible}). Section \ref{sec: stability} is focused on various stability properties of the interpolation polynomials such as adding a column to a partition $\lambda$ (Proposition \ref{prop: adding a column}) and changing $N$ for a fixed Young diagram $\lambda$. In particular, in Proposition \ref{prop: HOMFLY interpolation} we describe a HOMFLY-PT analogue of the interpolation polynomials depending on an additional parameter $A=q^N$. We provide lots of examples and tables of interpolation polynomials, especially for $\fgl_2$.
In Appendix A  we collect some additional known facts about the interpolation Macdonald polynomials and the Habiro ring. 
%and the interpolation.

\section*{Acknowledgments}

The authors would like to thank Pavel Etingof, Kazuo Habiro, Thang Le, Matt Hogancamp, Andrei Okounkov and Paul Wedrich for the useful discussions. We thank Satoshi Nawata for his comments on the first version of the paper and for bringing the reference \cite{KNTZ} to our attention.
Our work  was partially supported by the NSF grants DMS-1700814 (E.G.), DMS-1760329 (E.G.) and the NCCR SwissMAP (A.B.).

\section{Notations and conventions}

\subsection{$q$-binomial formulas}

Throughout the paper we will use the following notations for the $q$-series. The $q$-Pochhammer symbols are defined as
$$
(a;q)_{m}=\prod_{i=0}^{m-1}(1-aq^{i}),\ (a;q)_{\infty}=\prod_{i=0}^{\infty}(1-aq^{i}),\ m\ge 0.
$$
It is easy to see that
$$
(a;q)_{m+k}=(a;q)_{m}(aq^{m};q)_{k},\ (a;q)_{m}=\frac{(a;q)_{\infty}}{(aq^{m};q)_{\infty}}.
$$
We will use two normalizations for $q$-binomial coefficients defined as follows:
$$\{a\}_q=1-q^a,\
[a]_{q}=\frac{\{a\}_q}{\{1\}_q},\ [a]_{q}!=[1]_{q}\cdots [a]_{q},\ \binom{a}{b}_q=\frac{[a]_{q}!}{[b]_{q}![a-b]_{q}!}\ .
$$ 
Note that
$$
[a]_{q}=\frac{(q;q)_{a}}{(1-q)^{a}},\ \binom{a}{b}_q=\frac{(q;q)_{a}}{(q;q)_{b}(q;q)_{a-b}}.
$$
Finally, the $q$-binomial formula gives
$$
(a;q)_{m}=\sum_{j=0}^{m}(-1)^{j}q^{\frac{j(j-1)}{2}}\binom{m}{j}_{q}a^{j}.
$$
Let us also define symmetric $q$-numbers. For this we chose  $v$ such that
$v^2=q$ and set
$$ \{a\}=v^a-v^{-a},\quad [a]:=\frac{\{a\}}{\{1\}}, \quad
\left[\begin{array}{cc}\! a\!\\\!b\!\end{array}\right]=\frac{\{a\}!}{\{b\}!\{a-b\}!}\ .$$

We will use all these formulas throughout the paper without a reference.

\subsection{Partitions}

We will work with partitions $\lambda=(\lambda_1\ge \lambda_2\ge \ldots \lambda_N)$ which we will identify with the corresponding Young diagrams in French notation, where the rows have length $\lambda_i$. Transpose diagram to $\lambda$ is denoted by $\lambda'$, and $|\lambda|=\sum \lambda_i$.
Given a box in a Young diagram, we define its arm, co-arm, leg and co-leg as in Figure \ref{fig: arm}.

\begin{figure}[ht!]
\begin{tikzpicture}[scale=0.5]
\draw (0,0)--(10,0)--(10,2)--(7,2)--(7,5)--(4,5)--(4,7)--(0,7)--(0,0);
\draw (2,3)--(3,3)--(3,4)--(2,4)--(2,3);
\draw [<->] (3,3.5)--(7,3.5);
\draw [<->] (2,3.5)--(0,3.5);
\draw [<->] (2.5,0)--(2.5,3);
\draw [<->] (2.5,4)--(2.5,7);
\draw (3,1.5) node {$l'$};
\draw (3,5.5) node {$l$};
\draw (1,3) node {$a'$};
\draw (5,3) node {$a$};
\end{tikzpicture}
\caption{Arm, co-arm, leg and co-leg}
\label{fig: arm}
\end{figure}
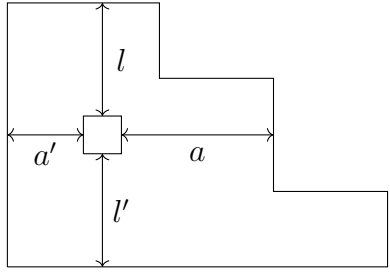

We define the hook length as $h(\sq)=a(\sq)+l(\sq)+1$, and the content $c(\sq)=a'-l'$.

Let 
$$
n(\lambda)=\sum (i-1)\lambda_i=\sum_{\sq} l'(\sq)=\sum_{\sq} l(\sq),
$$
then 
$$
n(\lambda')=\sum \frac{\lambda_i(\lambda_i-1)}{2}=\sum_{\sq} a'(\sq)=\sum_{\sq} a(\sq).
$$
The content of $\lambda$ is defined as
$$
c(\lambda)=\sum_{\sq} c(\sq)=n(\lambda')-n(\lambda).
$$
Let $\bar{\lambda}_i=\lambda_i+N-i$ for $1\leq i\leq N$,
then we have the following identity
\begin{equation}\label{prem:id}
\prod_{\Box\in \lambda}(1-t^{h(\Box)})=\frac{\prod_{i\geq 1}\prod^{\bar{\lambda}_i}_{j=1} \left(1-t^j\right)}{\prod_{i<j}\left(1-t^{\bar{\lambda}_i-\bar{\lambda}_j}\right)}
\end{equation}
 and we define
\begin{align}
\label{eq:def D}
D_N(\lambda)&=\sum_{i=1}^{N} \frac{(\bar{\lambda}_i)(\bar{\lambda}_i-1)}{2}=
\sum_{i} \frac{\lambda_i(\lambda_i-1)}{2}+\sum_{i} (N-i)\lambda_i+\sum_{i=1}^{N}\binom{N-i}{2}\\&= n(\lambda')+(N-1)|\lambda|-n(\lambda)+\binom{N}{3}=c(\lambda)+(N-1)|\lambda|+\binom{N}{3}.
\end{align}

\section{Quantum groups}

\subsection{Quantum $\fgl_N$}

The quantum group $\U=U_q\fgl_N$ is a 
$\C(v)$-algebra generated by $E_1,\ldots,E_{N-1}$, $F_1,\ldots,F_{N-1}$, $K^{\pm1}_1,\ldots,K_N^{\pm 1}$ 
satisfying the following relations:
\begin{equation}\label{eq:def1}
K_iE_i=vE_iK_i,\ K_iF_i=v^{-1}F_iK_i,\ K_{i+1}E_{i}=v^{-1}E_iK_{i+1},\ K_{i+1}F_i=vF_iK_{i+1}
\end{equation}
\begin{equation}\label{eq:def2}
[E_i,F_j]=\delta_{ij}\frac{K_iK_{i+1}^{-1}-K_{i+1}K_i^{-1}}{v-v^{-1}},\ [K_i,K_j]=0,
\end{equation}
\begin{equation}\label{eq:def3}
E_i^2E_j - [2]E_iE_jE_i + E_jE_i^2=0\ \text{if}\ |i - j| = 1\ \text{and}\ [E_i, E_j ]=0 \
 \text{otherwise}\end{equation}
 
\noindent
and analogously for $F_i$, where $v^2=q$.
To simplify the notation we set $\K_i:=K_iK^{-1}_{i+1}$.
Then the Hopf algebra structure on $\U$
(i.e. coproduct, antipode and counit) can be defined as follows:
$$
\Delta(E_i) = E_i\otimes 1 + \K_{i}\otimes E_i,\ \Delta(F_i) = 1\otimes F_i + F_i\otimes \K_i^{-1},\ \Delta(K_i^{\pm 1})=K_i^{\pm 1}\otimes K_i^{\pm 1},
$$
$$
S(K_i^{\pm 1})=K_i^{\mp 1},\ S(E_i) = -E_i\K_i^{-1},\ S(F_i) = -\K_iF_i
$$
$$
\varepsilon(K_i^{\pm 1}) = 1, \varepsilon(E_i) = \varepsilon(F_i) = 0.
$$

\vspace{0.8mm}

Usually $\U$ is considered as a subalgebra of $\U_h$ that is an $h$-adically complete $\mathbb{C}[[h]]$-algebra
topologically generated by $E_i$, $F_i$ and $H_j$
for $1\leq i\leq N-1$ and $1\leq j\leq N$
with
$$v=\exp{h/2},\quad  K_i=v^{H_i}=\exp{hH_i/2}$$
satisfying \eqref{eq:def2}, \eqref{eq:def3} and
$$H_iE_i-E_iH_i= E_i, \;\; H_iF_i-F_iH_i= -F_i,\;\;
 H_{i+1}E_i-E_iH_{i+1}=- E_i, \;\;
 H_{i+1}F_i-F_iH_{i+1}= F_i$$
replacing \eqref{eq:def1}.
Rewriting the defining relations in terms of  the  generators
$$e_i={E_i}(v-v^{-1}),\quad 
F_i^{(n)}=\frac{F_i^n}{[n]!}\quad \text{and}\quad
K_j\quad 
\text{for}\quad 1\leq i\leq N-1, \quad 1\leq j\leq N$$
we obtain an integral version $\U_\Z$ 
as a Hopf algebra over $\Z[v,v^{-1}]\subset \C(v)\subset
\C[[h]]$.

%{\color{red}last sentence added}
%Given a sequence of non-negative integers $n=(n_1, \dots n_{N-1})\in {\mathbf N}^{N-1}$,
%let $\U^e_n=\U_\Z e^n\U_\Z$ be the two sided ideal in $\U_\Z$ generated by $e^n:=\prod_i e^{n_i}_i$.
%Let us consider the completion
%$$\check\U
%=
%{\lim\limits_{\overleftarrow{\hspace{2mm}n\hspace{2mm}}}}\;   
%\frac{\U_\Z}{\U^e_n}$$

%By the very definition, the universal link invariant belongs
%to this completion.

%{\color{blue}
%We need to introduce $U_h$ to speak about $R$-matrix,
%please feel free to delete my comments after reading}

The quantum group $\fgl_N$ has a fundamental representation $\C^N$ with basis $v_1,\ldots,v_N$ such that 
$$
K_iv_j=v^{\delta_{ij}}v_j,\ %E_iv_{i+1}=v_i,\ F_iv_{i}=v_{i+1},\ 
E_iv_j=\begin{cases}
v_i & \text{if}\ j=i+1\\
0 & \text{otherwise}\\
\end{cases},
F_iv_j=\begin{cases}
v_{i+1} & \text{if}\ j=i\\
0 & \text{otherwise}.\\
\end{cases}
$$
%{\color{red}these action does not satisfy the commutator
%relations, and the numbering of vectors has nothing to do with the weights, isn't it?}
It generates a braided monoidal category with simple objects $V(\lambda)$, where $\lambda$ is a partition with at most $N$ parts. These are highest weight modules where $K_i$ act on the highest weight vector by $v^{\lambda_i}$. The fundamental representation corresponds to $\lambda=(1)$.
The representations $V(\lambda)$ have integral basis where $\U_{\Z}$ acts by  $\Z[v,v^{-1}]$-valued matrices.

\subsection{Ribbon structure}\label{sec:ribbon}
The Hopf algebra $\U_h$ admits a ribbon Hopf algebra structure (see e.g. \cite[Cor. 8.3.16]{CP}).
The universal $R$-matrix has 
the form
$\mathcal R=D \Theta$ where the diagonal part $D$
and the quasi-$R$-matrix are defined as follows
$$D=v^{\sum^N_{i=1} H_i\otimes H_i} \quad
\text{and} \quad \Theta=\sum_{\n\in \mathbb{N}^{N-1}} F_\n \otimes e_\n
$$
%{\color{red}{not sure about the sign in D}}
where 
for any sequence of non-negative integers
$\n=(n_1, \dots ,n_{N-1})$, the elements $e_\n$ and 
$F_\n$ are defined by equations (66) and (67) in
\cite{HL} and form  topological bases of the positive and negative parts in the triangular decomposition of
 $\U_\Z$.
%= \prod^{N-1}_{i=1}  v^{\frac{n_{i}(n_i-1)}{2}}
%F^{(n_i)}_i \quad\text{and}\quad
%e_\n=\prod^{N-1}_{i=1}  (v-v^{-1})^{n_i} E^{n_i}\ .
The inverse matrix $\CR^{-1}=\iota(\Theta) D^{-1}$ is obtained by 
applying the involution $\iota: v\to v^{-1}$.

%\F'_\n:=\bar i(F_\n) \quad\text{and}\quad
%\E'_\n:=\bar i(E_\n)$$
%we obtain 
The ribbon element and its inverse have the form
\begin{equation}
\label{eq:ribbon}
r=\sum_{\n} F_\n\; \K_{\n} \; r_0 \; e_\n
\quad\text{and}\quad
r^{-1}= \sum_{\n} \iota( F_\n)\; \K_{-\n} \;r^{-1}_0 \;\iota (e_{\n})
\end{equation}
where $r_0=K_{-2\rho} v^{-\sum_{i} H^2_i}$
and   
$K_{-2\rho}=\prod^N_{i=1} K_i^{2i-N-1}$ is the pivotal  element. 
%{\color{red} please check } 
Here for any sequence of integers $\n\in \Z^{N-1}$ we set $\K_\n=\prod_i \K^{n_i}_i$, and
 denote by $$
\rho=\left(\frac{N-1}{2},\frac{N-3}{2},\ldots,\frac{1-N}{2}\right)=\frac{1-N}{2}(1,\ldots,1)+(N-1,N-2,\ldots,0)
$$
the half sum of all positive roots. Using the central element $K=\prod^N_{i=1} K_i$, we 
can write the previous definitions as follows:
$$r^{-1}_0=K^N \;\prod^N_{i=1} K^{-2i}_i \; v^{\sum_i H_i(H_i+1)},
\quad K_{-2\rho}=K^{-N-1} \prod^N_{i=1} K^{2i}_i \ .
$$
The central element $r^{-1}$   acts on $V(\lambda)$ by the 
multiplication with $$\theta_{V(\lambda)}=v^{ (\lambda, \lambda+2\rho)}=
v^{N|\lambda|} q^{c(\lambda)}
$$
where $(\lambda, \mu)=\sum^N_{i=1} \lambda_i \mu_i$, $c(\lambda)$ is the content of $\lambda$ and $v^2=q$.

\subsection{Even part of  $\U$}

The algebra $\U$ has a natural grading by $\Gamma=\Z_2^N=\{\pm 1\}^N$ where $\zeta=(\zeta_1,\ldots,\zeta_N)\in \Gamma$ acts on $K_i$ by $\zeta_i$,  on $E_i$ by 1 and on $F_i$ by $\zeta_i\zeta_{i+1}$. It is easy to see that the defining relations are preserved under this action. Following \cite{HL}, we call an element of %$U_q(\fgl_N)$ 
$\U_N$ {\it even} or
 $\Gamma$-invariant
if it is preserved under the action of $\Gamma$.

Let us denote by $\U^{\text{ev}}_\Z$ a
 $\Z[q,q^{-1}]$-subalgebra of $\U_\Z$
 generated by $e_i$, $F^{(n)}_i\K_i$ and $K^2_j$
 for $1\leq i\leq N-1$ and $1\leq j\leq N$.
 It is easy to check that $\U^{\text{ev}}_\Z$ is $\Gamma$-invariant.

The action of $\Gamma$ descends on the category $\Rep(\U)$ of all finite-dimensional representations. Given  $\zeta=(\zeta_1,\ldots,\zeta_N)\in \Gamma$, we can define a one-dimensional representation $L(\zeta)$ where $E_i$ and $F_i$ act by zero, and $K_i$ act by $\zeta_i$. We can also define representation $V(\lambda)\otimes L(\zeta)$ where $K_i$ act on the highest weight vector by $\zeta_iv^{\lambda_i}$.

\begin{lemma}
\label{lem: twisted action}
 The action of $\U$ on $V(\lambda)\otimes L(\zeta)$ agrees with the $\Gamma$-twisted action of $\U$ on $V(\lambda)$. 
\end{lemma}

\begin{proof}
Indeed,   $\Delta(F_i)=1\otimes F_i+F_i\otimes \K_i^{-1}$, so $F_i$ acts on $V(\lambda)\otimes L(\zeta)$ via 
$F_i\otimes \K_i^{-1}=F_i\zeta_i\zeta_{i+1}$. Similarly, $E_i$ acts on $V(\lambda)\otimes L(\zeta)$ via $E_i\otimes 1=E_i$ and
$K_i$ acts via $K_i\otimes K_i=K_i\zeta_i.$
\end{proof}

\subsection{The subalgebra  $U_q\fsl_{N}$}
We define $U_q\fsl_{N}$ as a subalgebra of $\U$ generated by $E_i, F_i$ and $\K^{\pm 1}_i:=K^{\pm 1}_iK^{\mp 1}_{i+1}$ for $1\leq i\leq N-1$. The Hopf algebra $U_q\fsl_{N}$ also admits an integral version
$\U_{\Z}\fsl_{N}$ generated by 
$$e_i, \;\;
 F^{(n)}_i \quad
\text{and}\quad 
\mathcal{K}^{\pm 1}_i$$
 over $\Z[q,q^{-1}]$. 
 %Note that the generators
% $\K_i$ are group-like, i.e. $\Delta(\K_i)=\K_i\otimes \K_i$.
The braiding $\CR=D'\Theta$ with $\Theta$
as for $\fgl_N$, but different diagonal part 
$$D'=v^{\sum^{N-1}_{i=1} \frac{{\mathcal H}_i\otimes {\mathcal H}_i}{2}} \quad
\text{where}\quad {\mathcal H}_i=H_i-H_{i+1}\ .
$$
The ribbon element is defined by \eqref{eq:ribbon}
with $r_0=K_{-2\rho}\prod^{N-1}_{i=1}v^{{-\mathcal H}_i^2/2}$.
The pivotal element  $K_{-2\rho}$ does not change.
%{\color{blue} please check, HL gives $K_{-2\rho}$ for the pivotal
%in eq(73) but for sl2 it is $K^{-1}$}
Note that the $\Gamma$-invariant part of $U_q\fsl_N$ 
generated by $e_i$, $F^{(n)}_i\K_i$ and $\K^2_j$
 for $1\leq i, j\leq N-1$ has a smaller Cartan part
than its $\fgl_N$ analogue.
\begin{example}\nonumber{\rm
For $N=2$  the product $K_1K_2$ is central. By denoting $\mathcal K=K_1K_2^{-1}, E=E_1,F=F_1$ we get the standard presentation for $U_q(\fsl_2)$:
$$
\K E=v^2E \K ,\; \K F=v^{-2}F \K ,\; [E,F]=\frac{\K-\K^{-1}}{v-v^{-1}}.
$$}
\end{example}

\subsection{Universal invariant}
Lawrence, Reshetikhin, Ohtsuki and Kauffman constructed
quantum group valued {\it universal} link invariants.
As it was already mentioned in the introduction,
the universal invariant of a link is defined 
by splitting a diagram of its bottom tangle into elementary pieces and by associating $R$-matrices
 and pivotal elements to them. For more details  and references
we recommend to consult \cite[Sec. 7.3]{Hbottom}. However, we admit here the convention from \cite[Sec. 2.7]{HL} and write the contributions from left to right along the orientation of each component.

%Note that the universal invariant of the $(+1)$
%framed unknot is 

\section{Ribbon structure on $\Rep(\U)$}
\label{sec:repribbon}

The aim of this section is to compare the 
Reshetikhin-Turaev invariants of a bottom tangle whose components 
are colored with $V(\lambda)$ and $V(\lambda)\otimes L(\zeta)$.
%, from which
This will be  later used to prove Theorem \ref{thm:intro even}.

Let us denote by $\CR_\Q$ the representation ring
of $\Rep(\U)$ over $\Q(v)$. Given an $l$ component link $L$,
Reshetikhin--Turaev functor associated with Lie algebra $\fg$ provides a $\Q(v)$-multilinear map
\begin{align*}
 J_L: \CR_\Q\times \dots\times \CR_\Q&\to \Q(v)\\
(\mu_1, \dots, \mu_l)&\mapsto \bigotimes_i \Tr^{V(\mu_i)}_q \left(J_L(\fg;q)\right)=: J_L(\fg;\mu_1,...,\mu_l)
\end{align*}
 normalized to 
$\prod_i\dim_q (V(\mu_i))$ for the $0$-framed
$(\mu_1, \dots, \mu_l)$-colored unlink. 
In cases when $\fg$ is fixed in the context, we will remove it from the notation for simplicity.

Note that
in the case of a knot, we have $J_K(\lambda)=\dim_q(V(\lambda))
J_K(V(\lambda),q)$ where the last invariant is the colored Jones polynomial used in Introduction and normalized to be 1 for the unknot.
%{\color{red}do you agree with this convention??
%I just removed $q$, the dimension is in $\Z[v^{\pm 1}]$,
%e.g. it is $v+v^{-1}$ for the unknot colored by $V_1$,
%however if we divide by dimension it will be a function of $q$.}

The universal $R$-matrix defines a braiding between the representations $V(\lambda)$. We can extend this braiding to $\Rep(\U)$ as follows. Clearly, $L(\zeta)\otimes L(\zeta')\simeq L(\zeta\zeta')$ and we define the braiding between $L(\zeta)$ and $L(\zeta')$ to be trivial. Let $V$ be a finite-dimensional representation of $\U$ where the eigenvalues of $K_i$ are integral powers of $v$. Given $\zeta\in \Gamma$ we consider a $\C$-linear map $T_V(\zeta):V\to V$ which acts by $\prod \zeta_i^{a_i}$ on the weight subspace of $V$ where $K_i$ acts as $v^{a_i}$.

\begin{lemma}
The maps 
$$
c_{\zeta,V}:=\mathrm{swap}\circ (\Id\otimes T_V(\zeta)):L(\zeta)\otimes V\to V\otimes L(\zeta)
$$
with inverses
$$
c_{V, \zeta}:=\mathrm{swap}\circ ( T_V(\zeta)\otimes\Id): V\otimes L(\zeta)\to L(\zeta)\otimes V
$$
define a braiding on  $\Rep(\U)$.
\end{lemma}

\begin{proof}
First, let us check that $\text{swap}\circ (\Id\otimes T_V(\zeta))$ intertwines the actions of $\U$ on both sides. Indeed, let $v\in V$ be a vector with weight $(v^{a_1},\ldots,v^{a_N})$, then $E_iv$ has weight $(v^{a_1},\ldots,v^{a_i+1},v^{a_{i+1}-1},\ldots,v^{a_N})$
 while 
$F_iv$ has weight $(v^{a_1},\ldots,v^{a_i-1},v^{a_{i+1}+1},\ldots,v^{a_N})$.

Let $\bullet$ denote the basis vector in $L(\zeta)$, then
\begin{multline*}
c_{\zeta,V} E_i(\bullet \otimes v)= c_{\zeta,V} \left(\zeta_i\zeta_{i+1}\bullet \otimes E_i(v)\right)=\zeta_1^{a_1}\cdots \zeta_i^{a_i}\zeta_{i+1}^{a_{i+1}}\cdots \zeta^{a_N}_N   E_i(v)\otimes \bullet,\\
c_{\zeta,V} F_i(\bullet \otimes v)=  c_{\zeta,V}\left(\bullet \otimes F_i(v)\right)=\zeta_1^{a_1}\cdots \zeta_i^{a_i-1}\zeta_{i+1}^{a_{i+1}+1}\cdots \zeta^{a_N}_N F_i(v)\otimes \bullet,\\
c_{\zeta,V} K_i(\bullet \otimes v)= c_{\zeta,V}\left( \zeta_i\bullet \otimes K_i(v)\right)=\zeta_1^{a_1}\cdots \zeta_i^{a_i+1} \cdots \zeta^{a_N}_N  K_i(v)\otimes \bullet,\\
\end{multline*}
while 
\begin{multline*}
E_i c_{\zeta,V}(\bullet\otimes v)= E_i(\zeta_1^{a_1}\cdots \zeta_N^{a_N} v\otimes \bullet)=\zeta_1^{a_1}\cdots \zeta^{a_N}_N  E_i(v)\otimes \bullet,\\
 F_ic_{\zeta,V}(\bullet \otimes v)=F_i(\zeta_1^{a_1}\cdots \zeta_N^{a_N} v\otimes \bullet)=\zeta_1^{a_1}\cdots \zeta_i^{a_i-1}\zeta_{i+1}^{a_{i+1}+1}\cdots \zeta^{a_N}_N   F_i(v)\otimes \bullet,\\
K_i c_{\zeta,V}(\bullet \otimes v)= K_i(\zeta_1^{a_1}\cdots \zeta_N^{a_N} v\otimes \bullet)=\zeta_1^{a_1}\cdots \zeta_i^{a_i+1} \cdots \zeta^{a_N}_N K_i(v)\otimes \bullet.\\
\end{multline*}

Next, we observe that $T_V(\zeta)T_{V}(\zeta')=T_{V}(\zeta\zeta')$ and $T_{U\otimes V}(\zeta)=T_U(\zeta)\otimes T_V(\zeta)$, so $c_{\zeta,V}$ indeed defines a braiding.  Even more concretely, we get the braiding as the composition
\begin{multline}
\label{eq: extended braiding}
c_{L(\zeta)\otimes V,L(\zeta')\otimes U}:L(\zeta)\otimes V\otimes L(\zeta')\otimes U\xrightarrow{c_{V, \zeta'}}L(\zeta)\otimes L(\zeta')\otimes V\otimes U = L(\zeta')\otimes L(\zeta)\otimes V\otimes U \xrightarrow{c_{V,U}} \\
 L(\zeta')\otimes L(\zeta)\otimes U\otimes V\xrightarrow{c_{\zeta,U}} L(\zeta')\otimes U\otimes L(\zeta)\otimes V.
\end{multline}
\end{proof}

The representations $L(\zeta)$ are self-dual, and it is easy to see that the braiding $c_{\zeta,V}$ is compatible with changing $V$ to $V^*$. Therefore,
$\Rep(\U)$ with objects
  $L(\zeta)\otimes V$ form a pivotal braided monoidal category. %{\color{red} pivotal instead of rigid?}

 The quantum dimension of $L(\zeta)$ equals 
to the trace of the action of the pivotal element, which is $(\prod_i \zeta_i)^{N+1}$. 
%Indeed, $L(\zeta)\otimes L(\zeta)$ is naturally isomorphic to the trivial representation $L(1)$. Therefore the composition $L(1)\to L(\zeta)\otimes L(\zeta)\to L(1)$ is the identity map.
The twist coefficient $\theta_{L(\zeta)}$ is defined as the action of the ribbon element on $ L(\zeta)$, and is given by 
$(\prod_i \zeta_i)^{N}$.

\begin{lemma}
$\Rep(\U)$  is a ribbon  category with twist $\theta_{L(\zeta)\otimes V}=\theta_{L(\zeta)} \theta_V$.
\end{lemma}
\begin{proof}
%{\color{blue} I would replace the proof below by the following:
By  definition
$ \theta_{L(\zeta)\otimes V}=
c_{\zeta,V}\theta_{L(\zeta)} \theta_V c_{V,\zeta}=
\theta_{L(\zeta)} \theta_V$.
%From \eqref{eq: extended braiding} we get the double braiding
%$$ c_{L(\zeta')\otimes U,L(\zeta)\otimes V}\;
%c_{L(\zeta)\otimes V,L(\zeta')\otimes U}=c_{\zeta',V}c_{U,V}c_{U,\zeta}c_{\zeta,U}c_{V, U}c_{V,\zeta'}=c_{\zeta',V}c_{U,V}c_{V,U}c_{V,\zeta'}.
%$$
%Since the ribbon elements  $\theta_U$ and $\theta_{V\otimes U}$ and $\theta_{L(\zeta)}$ commute with $c_{\zeta,U}$, we get
%\begin{align*}
%(\theta_{L(\zeta)}\theta_V\otimes \theta_{L(\zeta')}\theta_U)c_{\zeta',V}c_{U,V}c_{V,U}c_{V,\zeta'}
%&=c_{\zeta',V}(\theta_{L(\zeta)}\theta_V\otimes \theta_{L(\zeta')}\theta_U)c_{U,V}c_{V,U}c_{V, \zeta'}\\  &=
%c_{\zeta',V}\theta_{(L(\zeta) \otimes V)\otimes 
%(L(\zeta')\otimes U)}c_{V,\zeta'}\\ &= \theta_{(L(\zeta) \otimes V)\otimes 
%(L(\zeta')\otimes U)}.
%\end{align*}
\end{proof}

%\subsection{Action of $\Z_2^{N-1}$ for $U_q\fsl_N$}
%\textcolor{red}{Generalize this section to $\fsl_N$?}
\subsection{Braiding in $\Rep(U_q\fsl_N)$}
In this section, we study the action of $\Gamma$ and the corresponding braiding for $U_{q}\fsl_N$, starting from $N=2$.
Similarly to the previous section,  $U_q\fsl_2$ has a one dimensional representation $L(-1)$ where $E$ and $F$ act by 0 and $\K$ acts by $-1$. 
The action of  $U_q\fsl_2$  on $L(-1)\otimes V$ is equivalent to $\Z_2$-twisted action on $V$ where $\Z_2$ scales $E$ by 1 and $F,\K$ by $-1$. 
 
One can attempt to define a  braiding for $U_q\fsl_2$. Since $E$ and $F$ shift the weights by $2$, it is easy to see that the analogue of $T_V$ should act by $(\sqrt{-1})^{a}$ on a subspace with weight $v^a$, and it does not square to identity. Nevertheless, it squares to $\pm \text{id}$ on each irreducible representation. This means that braiding relations on $\Rep(U_q\fsl_2)$ hold up to sign.

To pin down this sign, we define the {\em sign automorphism} $\Sigma_{V}$  which acts by $(-1)^{a}$  on a subspace with weight $v^a$. Since $E, F$ shift the weight by $\pm v^2$, $\Sigma_{V}$ commutes with the action of $U_q\fsl_2$ on $V$. The operator $\Sigma_{V}$ acts on the irreducible representation $V(n)$ by a scalar $(-1)^{n}$. Also, it is easy to see that $\Sigma_{V\oplus W}=\Sigma_{V}\oplus \Sigma_{W}$ and $\Sigma_{V\otimes W}=\Sigma_{V}\otimes \Sigma_{W}$.  

\begin{lemma}
\label{lem: sign braiding for sl2}
The operators $T_{V}$ and $\Sigma_{V}$ 
satisfy the following properties:
\begin{itemize}
\item[(a)] We have
$$
T_{V}^2=\Sigma_{V},\ c_{L(-1),V}=c^{-1}_{L(-1),V}(1\otimes \Sigma_{V})=(\Sigma_V\otimes 1)c^{-1}_{L(-1),V}
$$
\item[(b)] Let $c_{V,W}:V\otimes W\to W\otimes V$ be the braiding, then 
$$
c_{V,W}(\Sigma_{V}\otimes 1)=(1\otimes \Sigma_{V})c_{V,W},\ c_{V,W}(1\otimes \Sigma_{W})=(\Sigma_{W}\otimes 1)c_{V,W}
$$
\item[(c)] We have $c_{L(-1),V\otimes W}=c_{L(-1),V}\circ c_{L(-1),W}.$
\item[(d)] The braiding with $L(-1)$ satisfies Yang-Baxter equation, that is, the following diagram commutes:

\begin{tikzcd}
L(-1)\otimes V\otimes W \arrow{r}{c_{L(-1),V}} \arrow{d}{c_{V,W}}& V\otimes L(-1)\otimes W \arrow{r}{c_{L(-1),W}} & V\otimes W\otimes L(-1) \arrow{d}{c_{V,W}}\\
L(-1)\otimes W\otimes V \arrow{r}{c_{L(-1),W}} &  W\otimes L(-1) \otimes V  \arrow{r}{c_{L(-1),V}} & W\otimes V\otimes L(-1)\\
\end{tikzcd}
\end{itemize}
\end{lemma}

\begin{proof}
Part (a) is clear. To prove (b), observe that the action of $U_q\fsl_2\otimes U_q\fsl_2$ on $V\otimes W$ commutes with both $\Sigma_{V}\otimes 1$ and $1\otimes \Sigma_{V}$, and the $R$-matrix is an element of the completion of  $U_q\fsl_2\otimes U_q\fsl_2$. 

Given a pair of vectors $u\in V,w\in W$ such that $Ku=v^{i}u$ and $Kw=v^{j}w$, we get $K(u\otimes w)=v^{i+j}u\otimes w$, so 
$T_{V\otimes W}=T_{V}\otimes T_{W}$. Since $c_{L(-1),V}=\mathrm{swap}\circ (\Id\otimes T_V)$, we get the desired relation. 
Finally, (d) follows from (c).
\end{proof}

We can generalize the above results to representations of $U_q\fsl_N$ as follows. For $\zeta\in \Z_2^{N-1}$ there is a one-dimensional representation $L(\zeta)$ of $U_q(\fsl_N)$ where $E_i,F_i$ act by 0 and $\K_i=K_iK_{i+1}^{-1}$ act by $\zeta_i$ ($1\le i\le N-1$). Given a representation $V$ where all weights of $\K_i$ are integral powers of $v$, we can define an operator $T_{\zeta,V}:V\to V$ which acts by $\zeta^{A^{-1}\ba}$ on a subspace where $\K_i$ acts by $v^{a_i}$. Here 
$A$ 
is the Cartan matrix for $\fsl_N$ 
given by 
\begin{equation}\label{Cartan_matrix}
A=\left(\begin{matrix}
2 & -1 & 0 &\ldots & 0\\
-1 & 2 & -1 & \ldots & 0\\
0 & -1 & 2 & \ldots & 0\\
\vdots & \vdots & \vdots & \ddots & \vdots\\
0 & 0 & 0 &\ldots & 2\\
\end{matrix}\right)
\end{equation} 
%in \eqref{Cartan_matrix}
and $\ba=(a_1,\ldots,a_{N-1})$. Note that $\det(A)=N$, so $A^{-1}$ has rational entries with denominator $N$ and one needs to choose an $N$-th root of $(-1)$ to define $\zeta^{A^{-1}\ba}$. Define $\Sigma_{\zeta,V}=T^2_{\zeta,V}$.

\begin{lemma}
\label{lem: sign braiding for slN}
The operators $T_{\zeta,V}$ and $\Sigma_{\zeta,V}$ satisfy the following properties:
\begin{itemize}
\item[(a)] $T_{\zeta,V}E_i=\zeta_iE_iT_{\zeta,V}, T_{\zeta,V}F_i=\zeta_iF_iT_{\zeta,V}$
\item[(b)] $\Sigma_{\zeta,V}$ commutes with the action of $U_q\fsl_N$ on $V$
\item[(c)] The map  $c_{L(\zeta),V}=\mathrm{swap}\circ (\Id\otimes T_{\zeta,V}):L(\zeta)\otimes V\to V\otimes L(\zeta)$
is a morphism of $U_q\fsl_N$-representations
\item[(d)] The maps $T_{\zeta,V}$ and $\Sigma_{\zeta,V}$ satisfy all equations in Lemma \ref{lem: sign braiding for sl2}
with $L(-1)$ changed to $L(\zeta)$. 
\end{itemize}
\end{lemma}

\begin{proof}
(a) The operator $F_i$  changes the weight $\ba=(a_1,\ldots,a_{N-1})$ by $Ae_i$, so if $\K_i v=v^{a_i}v$ then
$$
T_{\zeta,V}F_i(v)=\zeta^{A^{-1}(\ba+Ae_i)}F_iv=\zeta^{A^{-1}\ba+e_i}F_iv=\zeta_iF_iT_{\zeta,V}(v).
$$
The proof for $E_i$ is similar. Part (b) immediately follows from (a).

For (c), we observe that the action of $E_i$ on $L(\zeta)\otimes V$ is the same as the action on $V$, while the actions of $F_i,\K_i$ are twisted by $\zeta_i$. On the other hand, the  action of $F_i$ on $V\otimes L(\zeta)$ is the same as the action on $V$, while the actions of $E_i,\K_i$ are twisted by $\zeta_i$. Therefore by (a) the operator $c_{L(\zeta),V}$ intertwines the actions of $U_q\fsl_N$ on 
$L(\zeta)\otimes V$ and $V\otimes L(\zeta)$. 

Finally, the proof of the rest of Lemma \ref{lem: sign braiding for sl2} extends to $U_q\fsl_N$ verbatim. 
\end{proof}

\begin{remark}{\rm
The above construction of $T_{\zeta,V}$ and $\Sigma_{\zeta,V}$ can be extended to an arbitrary semisimple Lie algebra with Cartan matrix $A$. The action of $\Sigma_{\zeta,V}$ can be interpreted in terms of projection of the weight lattice to its quotient by the root lattice.}
\end{remark}
 
%{\color{red}move proofs of main results in one section?}
%\subsection{Proof of Theorem \ref{thm:intro even}}
We  draw a tangle colored by a representation $V=V(\lambda)$ using solid lines, and a tangle colored by $L(\zeta)$ by dotted lines. If a component is colored by $L(\zeta)\otimes V$, we draw a dotted line on the left of a solid line and parallel to it. The crossings between solid and dotted lines correspond to $c^\pm_{L(\zeta),V}$ depicted in Figure \ref{braid}. Note that unlike $\fgl_N$ case, $c_{L(\zeta),V}$ does not square to identity and we have to distinguish under- and over-crossings between solid and dotted lines. This allows us to define Reshetikhin-Turaev invariants for framed tangles colored by representations of $U_q\fsl_N$ of the form $L(\zeta)\otimes V(\lambda)$.

Using the notations as in Figure \ref{braid},
we can visualize the statement of Lemma \ref{lem: sign braiding for slN} in Figure \ref{fig: diagrams for sigma}.
\begin{figure}[ht!]
\begin{center}
\begin{tikzpicture}
\draw (1,0)--(0,1);
\draw[line width=3,white] (0,0)--(1,1);
\draw[dotted] (0,0)--(1,1);
\draw (1.5,0.5);

\draw[dotted] (3,0)--(4,1);
\draw[line width=3,white] (4,0)--(3,1);
\draw (4,0)--(3,1);
\draw (4.5,0.5);

\draw (6,0.5) circle (0.2);
\draw (6,0.5) node {\scriptsize $\Sigma$};
\draw (6,0)--(6,0.3);
\draw (6,0.7)--(6,1);
\draw (7,0.5);
\end{tikzpicture}
\end{center}
\caption{The operators $c_{L(\zeta),V}$, $c^{-1}_{L(\zeta),V}$ and $\Sigma_{\zeta}$.}\label{braid}
\end{figure}
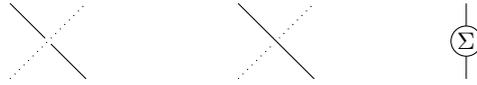

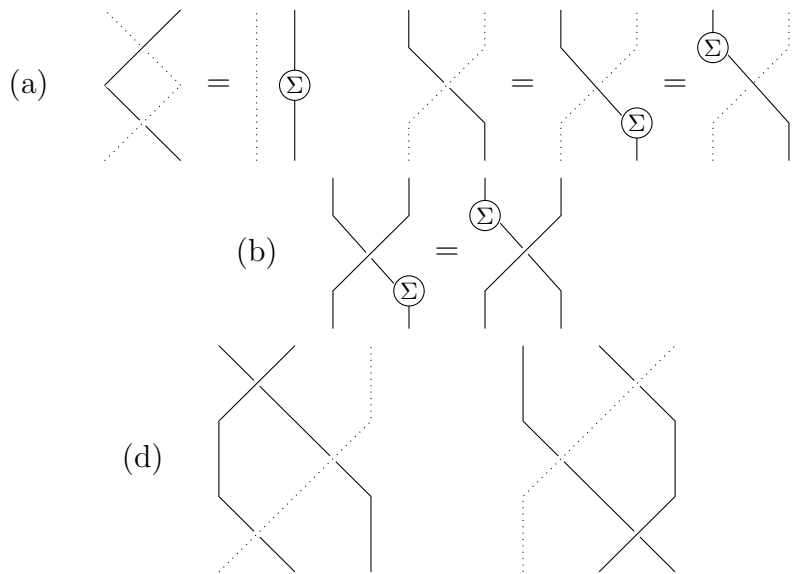
\begin{figure}
\begin{center}
\begin{tikzpicture}
\draw (-1,1) node {(a)};
\draw (1,0)--(0,1);
\draw[line width=3,white] (0,0)--(1,1);
\draw[dotted] (0,0)--(1,1);
\draw [dotted] (1,1)--(0,2);
\draw[line width=3,white] (0,1)--(1,2);
\draw (0,1)--(1,2);
\draw (1.5,1) node {$=$};
\draw [dotted] (2,0)--(2,2);
\draw (2.5,1) circle (0.2);
\draw (2.5,0)--(2.5,0.8);
\draw (2.5,1.2)--(2.5,2);
\draw (2.5,1) node {\scriptsize $\Sigma$};

\draw (5,0)--(5,0.5)--(4,1.5)--(4,2);
\draw[line width=3,white] (4,0)--(4,0.5)--(5,1.5)--(5,2);
\draw[dotted] (4,0)--(4,0.5)--(5,1.5)--(5,2);
 
\draw (5.5,1) node {$=$};

\draw[dotted] (6,0)--(6,0.5)--(7,1.5)--(7,2);
\draw[line width=3,white] (7,0.5)--(6,1.5);
\draw (7,0.5) circle (0.2);
\draw (7,0)--(7,0.3);
\draw (7,0.5) node {\scriptsize $\Sigma$};
\draw (6.8,0.6)--(6,1.5)--(6,2);

\draw (7.5,1) node {$=$};

\draw[dotted] (8,0)--(8,0.5)--(9,1.5)--(9,2);
\draw[line width=3,white] (9,0.5)--(8,1.5);
\draw (8,1.5) circle (0.2);
\draw (8,2)--(8,1.7);
\draw (8,1.5) node {\scriptsize $\Sigma$};
\draw (8.2,1.4)--(9,0.5)--(9,0);

\end{tikzpicture}
\end{center}

\begin{center}
\begin{tikzpicture}
\draw (5,1) node {(b)};

\draw (7,0.5) circle (0.2);
\draw (7,0)--(7,0.3);
\draw (7,0.5) node {\scriptsize $\Sigma$};
\draw (6.8,0.6)--(6,1.5)--(6,2);
\draw  [line width=3,white]  (6,0)--(6,0.5)--(7,1.5)--(7,2);
\draw    (6,0)--(6,0.5)--(7,1.5)--(7,2);

\draw (7.5,1) node {$=$};

\draw (8,1.5) circle (0.2);
\draw (8,2)--(8,1.7);
\draw (8,1.5) node {\scriptsize $\Sigma$};
\draw (8.2,1.4)--(9,0.5)--(9,0);
\draw [line width=3,white]  (8,0)--(8,0.5)--(9,1.5)--(9,2);
\draw   (8,0)--(8,0.5)--(9,1.5)--(9,2);
 
\end{tikzpicture}
\end{center}

\begin{center}
\begin{tikzpicture}
\draw (-1,1.5) node {(d)};
\draw (1,0)--(0,1)--(0,2);
\draw (2,0)--(2,1)--(0,3);
\draw [line width=3,white] (0,0)--(2,2);
\draw [dotted] (0,0)--(2,2)--(2,3);
\draw [line width=3,white] (0,2)--(1,3);
\draw (0,2)--(1,3);

\draw (6,0)--(4,2)--(4,3);
\draw [line width=3,white] (5,0)--(6,1)--(6,2)--(5,3);
\draw (5,0)--(6,1)--(6,2)--(5,3);
\draw [line width=3,white] (4,0)--(4,1)--(6,3);
\draw [dotted] (4,0)--(4,1)--(6,3);
\end{tikzpicture}
\end{center}
\caption{Diagrammatics for Lemma \ref{lem: sign braiding for slN}}
\label{fig: diagrams for sigma}
\end{figure}

\begin{theorem}
\label{thm: zeta RT}
%(a) Let $T$ be an algebraically split $0$-framed bottom tangle. Then the Reshetikhin--Turaev invariants of $T$ colored with  the $U_q\fsl_N$ modules $V(\lambda)$ and  $L(\zeta)\otimes V(\lambda)$   coincide.
%\textcolor{red}{holds also for $\fsl_N$?}
(a) Let $L$ be an algebraically split $0$-framed link with $\ell$ components. Then for arbitrary partitions $\lambda_1,\ldots,\lambda_{\ell}$ and $\zeta_1,\ldots,\zeta_{\ell}\in \Gamma$ the following identity of Reshetikhin--Turaev invariants holds:
\begin{multline}
\label{eq: RT identity}
J_L\left({\fsl_N};V(\lambda_1)\otimes L(\zeta_1),\ldots,V(\lambda_{\ell})\otimes L(\zeta_{\ell})\right)=\\
J_L\left({\fsl_N};V(\lambda_1),\ldots,V(\lambda_{\ell})\right)\cdot \dim_{q}L(\zeta_1)\cdots \dim_{q}L(\zeta_{\ell}),
\end{multline} 
where $\dim_{q}L(\zeta_i)=\Tr^{L(\zeta_i)}_{q}(1)=\pm 1$.

(b) Let $L$ be an arbitrary link with 
evenly framed components, if $N$ is odd. Then \eqref{eq: RT identity} holds for $\fgl_N$ Reshetikhin--Turaev invariants.
%the Reshetikhin-Turaev invariants of $T$ corresponding to the  $\U_q\fgl_N$ representations $V({\lambda})$ and  $L(\zeta)\otimes V({\lambda})$  coincide.

%(c) Let $T$ be an algebraically split $0$-framed bottom tangle. Then the universal $\fsl_N$ invariant $J_T$ is $\Gamma_{\fsl_N}$-invariant. 

%(d) Let $T$ be an arbitrary bottom tangle with 
%evenly framed components, if $N$ is odd. Then the universal $\fgl_N$ invariant $J_T$ is $\Gamma_{\fgl_N}$-invariant.
\end{theorem}
%{\color{red} here is unclear what the coloring means,
%all components the same?? I do not know how to compare endomorphisms acting on different spaces}

\begin{proof}
(a) We use the results of Lemmas \ref{lem: sign braiding for sl2} and  \ref{lem: sign braiding for slN} and the above diagrammatic notation. By Lemma \ref{lem: sign braiding for sl2}(a), we can change crossings between dotted and solid lines at a cost of placing $\Sigma_{\zeta}$ on solid lines. By doing this iteratively, we can make all dotted lines to be above solid lines.
%so that the dotted components can be unlinked from each other. 
%and their contribution to the invariant depend only on their framings.
At this stage, each solid component of $L$ acquires several copies of 
$\Sigma_{\zeta}$ and $\Sigma_{\zeta}^{-1}$ at various places of the link diagram. The number of these copies (with signs) equals the linking number between this component and the dotted part which is even by our assumption. By Lemma  \ref{lem: sign braiding for sl2}(b) we can combine all these copies of $\Sigma_{\zeta}$  together and cancel out. 
%Finally, using Lemma \ref{lem: sign braiding for sl2}(d), we can separate the dotted and solid links.
Finally, using Lemma \ref{lem: sign braiding for sl2}(d), we can separate the dotted and solid links. By changing the crossings in the dotted link, we transform it to the $0$-framed unlink. Therefore the invariant of the solid link equals $J_L\left({\fsl_N};V(\lambda_1),\ldots,V(\lambda_{\ell})\right)$ while the invariant of the dotted link equals $\dim_{q}L(\zeta_1)\cdots \dim_{q}L(\zeta_{\ell})$.
%%Since  the contributions from the pivotal elements, 
%%crossings and framings  on the doted part depend on the linking matrix only, they
%%have to cancel as well.

The proof of (b) is similar, except that $\Sigma_{V}$ is trivial for all $V$. As before we can unknot dotted components.
Now  the ribbon element 
acts on $L(\zeta)$  by
$\theta_{L(\zeta)}=(\prod_i \zeta_i)^{N}$, 
 and  hence, any (even if $N$ is odd) number of them  acts by $1$.
 The result follows.

%To compute the contributions from
%self-crossings
%(or framings)
%$\theta_{L(\zeta)}=(\prod_i \zeta_i)^{N}$   
%and the pivotal elements on $L(\zeta_i)$ we argue as follows.
%Choose a braid diagram $D$ such that the closures of $D$ and $T$ coincide.
%Note that the power of the pivotal element in $T$ (that is, the number of minima in $T$) has the same parity as the number of crossings. Indeed, for a braid on $n$ strands which closes up to a link with $l$ components, the number of minima is $n-l$ (we close all but one strand for each component). On the other hand, the corresponding permutation in $S_n$ has $l$ cycles, hence its sign equals $(-1)^{n-l}$ and the number of crossings equals $n-l$ modulo 2.   Since, each self-crossing together with a pivotal element acts on $L(\zeta)$ in the same way as a ribbon element
%$\theta_{L(\zeta)}=(\prod_i \zeta_i)^{N}$, we have the result.

\end{proof}

\section{Center of $\U$}
Let $\CZ$ be the center of $\U_\Z$.
In this section we recall the main  facts known about $\CZ$.
 
\subsection{Harish-Chandra isomorphism}
\label{sec:HC}
%{\color{blue} Do we need $\CR_N, \SF_N$ and $\CZ_N$ instead?}
Let $(\U^0_\Z)^{S_{N}}:=\Z[v,v^{-1}][{K^{\pm 1}_{1}},\ldots,K_N^{\pm 1}]^{S_{N}}$
be the Cartan part of $\U_\Z$ invariant under the Weyl group action. After a multiplication by an appropriate power of the central element $K:=\prod^N_{i=1}K_i$, 
each element of $(\U^0)^{S_N}$
 can be viewed as a symmetric function in $N$ variables.
This allows  to identify $(\U^0_\Z)^{S_{N}}$ with
 the ring of symmetric functions divided by powers of
the elementary symmetric polynomial
$e_N=K$.
In the classical case, this ring can be identified with the center using the Harish-Chandra isomorphism. After quantization,  
the image of the Harish-Chandra homomorphism belongs to
$$\SF=\Z[v^{\pm 1}, e_N^{-1}][x_1,\dots,x_N]^{S_N}$$ 
where $x_i=K_i^2$ (compare e.g. \cite[Ch. 6]{Ja}).
%$\frac{x_{i+1}}{x_i}+\frac{x_i}{x_{i+1}}$.
In this section 
we will furthermore identify  $\SF$ with   the Grothendieck ring $\CR$  of 
$\Rep(\U)$ with coefficients $\Z[v^{\pm 1}]$.

First,  the character map $$\ch: \CR\to \SF$$ sends a representation $U$ to its character $\ch(U)$. Clearly, $\ch(U\oplus V)=\ch(U)+\ch(V)$ and $\ch(U\otimes V)=\ch(U)\ch(V)$, so $\ch$ is a ring homomorphism. The character of $V(\lambda)$ equals the Schur function $s_{\lambda}(x_1,\dots, x_N)$, while the character of $L(\zeta)$ equals $\zeta_1\cdots \zeta_N$.

The Harish-Chandra map $$\hc:\CZ\to \SF $$ 
is defined as follows. 
Let $\phi$ be a central element in $\U$, it acts in the Verma module $\Delta(\lambda)$ by some scalar $\phi|_{\Delta(\lambda)}$. We define $\hc(\phi)$ to be the polynomial in $\SF$ defined by the condition
$$
\hc(\phi)(q^{\rho+\lambda})=\phi|_{\Delta(\lambda)}\quad \text{for all}\quad \lambda
$$
where 
$
\rho=\left(\frac{N-1}{2},\frac{N-3}{2},\ldots,\frac{1-N}{2}\right)
$. 
Note that the product $\phi\phi'$ acts on $\Delta(\lambda)$ by the product of the corresponding scalars, so $\hc$ is also a ring homomorphism. It is known to be an isomorphism (see e.g. \cite[Ch. 6]{Ja}).

%The $\Gamma$-invariant part of the center is identified
%under the Harish-Chandra isomorphism with
%$$\SF^\Gamma=
%\Z[v^{\pm 1}, e_N^{-2}][x_1,\dots,x_N]^{S_N}\quad
%\text{ where} \quad x_i=y_i^2\ .$$
Finally, the map $\xi:\CR\to \CZ$ is defined by $\xi=\hc^{-1}\circ \ch$.
  It is a composition of two ring homomorphisms and hence a ring homomorphism too. 
Hence, we get the commutative diagram:
\begin{center}
\begin{tikzcd}
\CR \arrow{dr}{\ch} \arrow{rr}{\xi}& & \CZ \arrow{dl}{\hc}\\
 & \SF & 
\end{tikzcd}
\end{center}  
  In Lemma \ref{lem:Hopf}
we will show that $\xi$ actually coincides with the Drinfeld map.

%It is well-known that for generic $q$, $\xi$ is actually an
%isomorphism. {\color{blue} reference?}
\begin{example}
{\rm The central element $K=K_1\cdots K_N$ acts on $V(\lambda)$ by a scalar $v^{\sum \lambda_i}$. Since $\sum \rho_i=0$, we get $\hc(K_1\cdots K_N)=y_1\cdots y_N$. }
\end{example}

\begin{example}
{\rm The center of $U_q\fsl_2$ is generated by the Casimir element:
$$
C=(v-v^{-1})^2 FE+v\K+v^{-1}\K^{-1}
$$
It acts on a representation $V_m$ by $v^{m+1}+v^{-m-1}$, so $\hc(C)=y+y^{-1}$ (note that $v^{\rho}=v$ in this case). 
On the other hand, $\ch(V_1)=y+y^{-1}$, so $\xi(V_1)=C$,
where $V_1$ the 2-dimensional representation.

Similarly, we can consider the corresponding central element in $U_q\fgl_2$ defined by
$$
C_{\fgl_2}=(v-v^{-1})^2 FE+vK_1K_2^{-1}+v^{-1}K_1^{-1}K_2.
$$
It acts on  a representation $V(\lambda)$ by a scalar
$$
v^{1+\lambda_1-\lambda_2}+v^{-1-\lambda_1+\lambda_2}=\frac{y_1}{y_2}+\frac{y_2}{y_1},\ \quad y_1=v^{1/2+\lambda_1},y_2=v^{-1/2+\lambda_2},
$$
so $\hc(C_{\fgl_2})=\frac{y_1}{y_2}+\frac{y_2}{y_1}=\frac{y^2_1+y^2_2}{y_1y_2}=e^{-1}_2(y_1,y_2)(x_1+x_2)$.
}
\end{example}

% The Harish-Chandra isomorphism descends to $\U_{\Z}$, and identifies the center of 
%$\U_Z$ with the space of symmetric Laurent polynomials in $x_1,\ldots,x_n$ with coefficients in $\Z[v,v^{-1}]$. 

\subsection{Hopf pairing}
\label{sec: Hopf}

%We assume that the reader is familiar with the Reshetikhin-Turaev invariants of links \cite{} with components labeled by the representations of a Hopf algebra (in our case, $U_q(\fgl_N)$). 
The Hopf pairing $\langle U,V\rangle$ of two representations $U,V\in \CR$ is defined as the Reshetikhin--Turaev invariant of the Hopf link with components labeled by $U$ and $V$. This is a symmetric bilinear pairing on $\CR$.
The map $\xi$ is related to the Hopf pairing as follows:

\begin{lemma}\label{lem:Hopf}
The Hopf pairing on $\CR$  can be computed as
$$
\langle U,V\rangle=\Tr_q^{U}(\xi(V)).
$$
\end{lemma}
%%{\color{red} check normalization!}
%%%\begin{proof}
%%%Without loss of generality we can assume $U=V(\mu)$, $V=V(\lambda)$, then $\ch(V)=s_{\lambda}(x_1,\dots, x_N)$ and $\xi(V)=\hc^{-1}(s_{\lambda})$.
%%%The central element $\xi(V)$ acts on $U=V(\mu)$ by the scalar $s_{\lambda}(q^{\rho}\cdot q^{\mu})=s_{\lambda}(q^{\mu+\rho})$, therefore 
%%%$$
%%%\Tr_q^{U}(\xi(V))=s_{\lambda}(q^{\mu+\rho})\Tr_q^U(1)=s_{\lambda}(q^{\mu+\rho})s_{\mu}(q^{\rho}).
%%%$$
%%%On the other hand, it is well known \cite{HL} that the Hopf pairing is given by the same formula.
\begin{proof}
Consider the Drinfeld map $D$ \cite{Dri} which sends a representation $V$ to a central element corresponding to the universal invariant of the following tangle:
\begin{center}
%\begin{tikzpicture}
%\draw (0,0) circle (1);
%\draw [white,line width=5] (0,2)--(0,-0.9);
%\draw (0,2)--(0,-0.8);
%\draw (0,-1.1)--(0,-2);
%\draw (1.2,0) node {$V$};
%\draw (-2,0) node {$D(V)=$};
%\draw (-0.1,-0.2)--(0,0)--(0.1,-0.2);
%\draw (-1.1,0.2)--(-1,0)--(-0.9,0.2);
%\end{tikzpicture}
\begin{tikzpicture}[scale=0.8]
\draw (1,1).. controls (1,1.5) and (1,1.5) .. (0,2);
\draw [line width=5,white] (0,1).. controls (0,1.5) and (0,1.5).. (1,2);
\draw (0,1).. controls (0,1.5) and (0,1.5).. (1,2);
\draw[stealth-] (1,0).. controls (1,0.5) and (0,0.5) .. (0,1);
\draw [line width=5,white] (0,0).. controls (0,0.5) and (1,0.5)..(1,1);
\draw (0,0).. controls (0,0.5) and (1,0.5)..(1,1);
\draw [stealth-] (-1,0) .. controls (-1,1.5) and (-1,2.5) .. (0,2);
\draw (2,0) .. controls (2,1.5) and (2,2.5) .. (1,2);
\draw (1,0).. controls (1,-1) and (2,-1) .. (2,0);
\draw (-1,0)--(-1,-0.7);
\draw (0,0)--(0,-0.7);
\draw (-2.5,1) node {$D(V):=$};
\draw (2.5,1) node {$V$};
\end{tikzpicture}
\end{center}
By e.g. \cite[eq. (20)]{GZB} (see also \cite[Proposition 8.19]{HL} and references therein) the eigenvalue of $D(V)$ on the irreducible representation $V(\lambda)$ equals $\ch(q^{\lambda+\rho})$ where $\ch$ is the character of $V$. By the definition of the Harish-Chandra map, this means that $\hc(D(V))=\ch(V)$, and $$D(V)=\hc^{-1}(ch(V))=\xi(V), $$ so $\xi$ agrees with the Drinfeld map.
Now $ \langle U,V\rangle=\Tr_q^{U}(D(V))=\Tr_q^{U}(\xi(V))$ or more precisely,
$$\langle V(\lambda), V(\mu)\rangle=
s_\lambda(q^{\mu+\rho})s_\mu(q^\rho)\quad \text{where}
\quad \dim_q V(\mu)=s_\mu(q^\rho) .$$
\end{proof}

Using 
the  Drinfeld isomorphism $\xi$ we can extend the Hopf pairing  to the center  by setting
%$$\langle \xi(U),\xi(V)\rangle:=\langle U, V\rangle
%\quad \text{for any} \quad  U,V \in \CR
%$$
%Using the isomorphism $\xi:\CR\to \CZ$ one can extend
%Hopf pairing to the center.
%Let us extend the Hopf pairing to $\CZ$ as follows
$$\langle z_1,z_2\rangle:=\langle \xi^{-1}(z_1),\xi^{-1}(z_2)
\rangle\quad \text{for any}\quad z_1,z_2\in \CZ\ .$$
%where the right hand side was given in Section \ref{sec: Hopf}.

\section{Cyclotomic completion and the universal invariant}

The universal invariant of a link 
%(or bottom tangle)
%is defined in full details in \cite[Sec. 7.3]{Hbottom}.
 belongs a priori  to a (completed) tensor product
of copies of  $\U_h$, rather than $\U$ or $\U_\Z$, due to the
diagonal part of the $R$-matrix.
The aim of this section is to define a certain
completion of $\U_\Z$ and its tensor powers, such that
the  universal $\fgl_N$ invariant of evenly framed 
links belongs to it. Since the action of $\Gamma$ extends to the completion,
this will allow us to speak about $\Gamma$-invariance
of $J_L(\fgl_N;q)$.

%Below we will need an integral form $\U_{\Z}$ of the quantum group. It is generated over $\Z[q,q^{-1}]$ by  $K_i,e_i=E_i(q-q^{-1})$ and divided differences $F^{(n)}=F^n/(q;q)_{n}$. It is not hard to see (we refer to \cite{H??,HL} for all details) that $\U_{\Z}$ is closed under product, coproduct and antipode, hence it is a Hopf subalgebra of $\U$. It is also known \cite{} that the representations $V(\lambda)$ have integral basis where $\U_{\Z}$ acts by  $\Z[q,q^{-1}]$-valued matrices. 
%Given $\n \in \N^{N}$, 
%and $\mathbf d\in \Z^{N-1}$, 
%let us denote by 
%$$ f_\n (K^2):=
%\prod_i f_{n_i}(K^2_i) \quad\text{where}\quad
%\ .$$
\subsection{Cyclotomic completion of $\U_\Z$}
Given $n \in \N$,
we define a family of two-sided  ideals $\U^{(n)}_{\Z}$ as the minimal filtration such that $\U^{(n)}_{\Z}\U^{(m)}_{\Z}\subset \U^{(m+n)}_{\Z}$ 
and 
$$
(q;q)_{n},\;\; e^{n}_i,\;\; f_n(K^2_j)\in \U^{(n)}_{\Z}
$$
%for any $\n\in \N^{N-1}$ with $\sum_i n_i=n$. 
for any $1\leq i\leq N-1$ and $1\leq j\leq N$
where $f_n(x)=(x;q)_n $.
In other words,  $\U^{(n)}_{\Z}$ is the two-sided ideal
generated by the products 
\begin{equation}
\label{eq: filtration}
(q;q)_a \;e_{\m}\; f_{c_1}(K^2_1)\cdots f_{c_{N}}(K^2_{N}),\quad
\text{with}\quad a+\sum_i m_i+ \sum_i c_i=n \ .
\end{equation}

%{\color{red} use root vectors for $e_{\n}$}
%\textcolor{blue}{Maybe move this lemma to symmetric functions part and refer to it here?}

\begin{lemma}
\label{lem: coproduct fn}
We have 
$$
\Delta\left(f_n\left(K^2_i\right)\right)=\sum_{a=0}^{n}\binom{n}{a}_q f_a(K^2_i)\otimes K_i^{2(n-a)}f_{n-a}(K^2_i).
$$
\end{lemma}

%\begin{example}
%We have $f_2(K_i)=(1-K_i)(1-qK_i)$, so 
%$$
%\Delta f_2(K_i)=(1-K_i\otimes K_i)(1-qK_i\otimes K_i)=
%$$
%$$
%(1-K_i)(1-qK_i)\otimes K_i^2+(1+q)(1-K_i)\otimes K_i(1-K_i)+1\otimes (1-K_i)(1-qK_i).
%$$
%\end{example}

\begin{proof}
We prove Lemma by induction in $n$. For $n=0$ it is clear. The induction step follows from the identities
$$
f_{n+1}(K^2_i)=f_n(K_i)(1-q^nK^2_i) 
$$
and
$$
\Delta(1-q^nK^2_i)=1\otimes 1 -q^nK^2_i\otimes K^2_i=(1-q^aK^2_i)\otimes q^{n-a}K^2_i+1\otimes (1-q^{n-a}K^2_i).
$$
%so 
%$$
%\Delta(f_{n+1}(K_i))=\Delta(f_n(K_i))\Delta(1-q^nK_i)=
%$$
%$$
%\sum_{a=0}^{n}\binom{n}{a}_q f_a(K_i)(1-q^aK_i)\otimes K_i^{n-a}f_{n-a}(K_i)q^{n-a}K_i+
%$$
%$$
%\sum_{a=0}^{n}\binom{n}{a}_q f_a(K_i)\otimes K_i^{n-a}f_{n-a}(K_i)(1-q^{n-a}K_i)=
%$$
%$$
%\sum_{a=0}^{n+1}\binom{n+1}{a}_q f_a(K_i)\otimes K_i^{n+1-i}f_{n+1-a}(K_i).
%$$
\end{proof}

%Observe that Lemma \ref{lem: coproduct fn} holds after replacing
%$K_i$ with $K^2_i$.

\begin{proposition}
\label{prop: filtration}
a) $\U^{(n)}_{\Z}$ is the left ideal generated by \eqref{eq: filtration}.

b) $\U^{(n)}_{\Z}$ form a Hopf algebra filtration, that is $\Delta \,\U^{(n)}_{\Z}\subset\, \sum_{i+j=n}\,\U^{(i)}_{\Z}\otimes \U^{(j)}_{\Z}$.

c) Assume that $\lambda_i\le k$ for all $i$. Given arbitrary $m$, there exists $n=n(k,m)$ such that the elements of $\U^{(n)}_{\Z}$ act on the integral basis of $V(\lambda)$ by matrices divisible by $(q;q)_{m}$. 
\end{proposition}

\begin{proof}
a) Observe that by Lemma \ref{lem: binomial} we get $f_n(q^{s}K^2_i)\in \U^{(n)}_{\Z}$ for all integer $s$. Now the statement follows from the identities
$$
f_n(K^2_i)F_i^{(s)}=F_i^{(s)}f_n(q^{-s}K^2_i),\ f_n(K^2_{i+1})F_{i}^{(s)}=F_{i}^{(s)}f_n(q^{s}K^2_{i+1})
$$
and 
$$
f_n(K^2_i)e^s_i=e^s_if_n(q^{s}K^2_i),\ f_n(K^2_{i+1})e^s_{i}=e^s_{i}f_n(q^{-s}K^2_{i+1}).
$$
%$$f_n(K^2_j)F_i^{(s)}=f_n(q^{-a_{ij}s}K^2_j)F_i^{(s)}\quad \text{and}
%\quad f_n(K^2_j)e^s_i=e^s_if_n(q^{a_{ij}s}K^2_j).$$ 
%{\color{red}these equations are not true, 
%even after moving F's on the other side:
%set $i=j$, $a_{ii}=2$}

%Recall that
%$A=(a_{ij})$ is the Cartan matrix
% given in \eqref{Cartan_matrix}.

b) Follows from the identity 
$$\Delta(e_j^m)=\sum_{i=0}^{m}\binom{m}{i}_qe_j^{m-i}\K^i\otimes e_j^i$$ and 
Lemma \ref{lem: coproduct fn}.

c) By (a), it is sufficient to check the statement for $e_i^n$ and $f_n(K^2_i)$. If $\lambda_i\le k$ then for $n>k$ $e_i^n$ annihilates $V(\lambda)$, while 
$f_n(K^2_i)$ acts on a vector with weight $(v^{\lambda_1},\ldots,v^{\lambda_N})$ by $f_n(q^{\lambda_i})=(q^{\lambda_i};q)_n$ which is divisible by $(q;q)_{n}$.
\end{proof}
%{\color{red}I do not understand why do we need (c)}
 By Proposition \ref{prop: filtration}(b), the filtration 
$$\U^{}_{\Z}=\U^{(0)}_{\Z} \supset \U^{(1)}_{\Z}
\supset \dots \U^{(n)}_{\Z}\supset \dots$$
is a Hopf algebra filtration  of $\U_\Z$ with respect 
to a descending filtration of ideals $I_n=((q;q)_n)$ in {$\Z[v,v^{-1}]$} in the sense of \cite[Sec. 4]{Hcenter}.
%where $I_n$ is generated by $$.
Hence, the completion 
$$
\widehat{\U}:={\lim\limits_{\overleftarrow{\hspace{2mm}n\hspace{2mm}}}}\;  \; \frac{\U_{\Z}}{\U^{(n)}_{\Z}}
$$
{is a complete Hopf algebra over the  ring
 $$ \widehat{\Z[v]}:=
{\lim\limits_{\overleftarrow{\hspace{2mm}n\hspace{2mm}}}}\; \;  
\frac{\Z[v]}{((q;q)_n)} .$$}
%{\color{red} we can not remove $v$ here, since as long as we have $K_i\in \U_\Z$ the commutator relation involve $v$, therefore only the even part will have all coeffs in $Z[q,q^{-1}]$.}
We refer to \cite[Section 4]{Hcenter} for details.
Analogously, we define the $\Gamma$-invariant subalgebra
$$
\widehat{\U^{\text{ev}}}:={\lim\limits_{\overleftarrow{\hspace{2mm}n\hspace{2mm}}}}\;  \; \frac{\U^{\text{ev}}_{\Z}}{\U^{(n)}_{\Z}}
$$
as a complete Hopf algebra over the Habiro ring $\widehat{\Z[q]}$.
Let us now extend the completion to the tensor powers of $\U_\Z$.
For this we define the filtration for $\U_\Z^{\otimes l}$
for $l\geq 1$ as follows
$$
{\mathcal F}_n(\U_\Z^{\otimes l})=\sum^l_{i=1}
\U_\Z^{\otimes i-1}\otimes \U^{(n)}_\Z
\otimes\U_\Z^{\otimes l-i}
$$
and the completed tensor product
$\U_\Z^{\hat{\otimes}l}$ with respect to this filtration
will be the image of the homomorphism
$$
{\lim\limits_{\overleftarrow{\hspace{2mm}n\hspace{2mm}}}}\;  \; \frac{\U_\Z^{\otimes l}}{
{\mathcal F}_n(\U_\Z^{\otimes l})}  \;\;\to\;\;
\U_h^{{\otimes}k}
$$
where on the right hand side we use the $h$-adically completed tensor product.
%that we will omit in what follows.

\subsection{Hopf pairing and universal invariants}\label{sec:HU}
Let us denote by $c\in \U_h {\otimes}\, \U_h$  the double braiding or the universal invariant of the clasp tangle in Figure \ref{fig:clasp}, given by
$$c=(S\, \otimes\,\text{id})\,\CR_{21}\CR \ .$$
The main point about this element is that it is dual to
the Hopf pairing or the quantum Killing form (compare \cite[Sec. 4]{HL}).
%{\color{red} cite Rosso here \cite{Ro}?} 
Hence, 
after writing
$c=\sum_{i} c(i) \otimes  c'(i)
$ 
  the Hopf pairing is defined by setting
\begin{equation}\label{eq:Hopf_on_U}
\langle c(i), c'(j) \rangle :=\delta_{ij} 
\end{equation}

%where  $c(i)$ and $c'(i)$ are dual 
% $\Z[v^{\pm 1}]$-bases of $\widehat \U_\Z$ with respect to the Hopf pairing or the quantum Killing form.
 Restricting to the Cartan part this gives us
 (compare \cite[Lemma 3.12]{HL})
 \begin{align}
 D^{-2}=\prod^N_{i=1} q^{-H_i\otimes H_i}=
\prod^N_{i=1} \sum_{n_i} (-1)^{n_i}\frac{h^{n_i}}{n_i!}H_i^{n_i}\, \otimes \, H_i^{n_i}
\end{align}
and hence, $\langle H_i^n, H_j^m\rangle=\delta_{ij}\delta_{nm} (-1)^n
\frac{n!}{h^n}$. We deduce that
$\langle K_i^2, K_j^2\rangle=q^{-1}$ or, more generally,
$$\langle K_i^{2a}, K_j^{2b}\rangle=\delta_{ij}q^{-ab}$$
defines the Hopf pairing on the $\Gamma$-invariant part of the Cartan. In Section \ref{sec:inter} we construct
another basis for the Cartan given by
$\prod^N_{i=1}  f_{n_i}(K^2_i)$ such that
$ \langle f_n, f_m \rangle= \delta_{nm}(-1)^n q^{-n}(q;q)_n$. In this new basis, we can rewrite the Cartan part of the clasp element as follows:
\begin{equation}\label{eq:D2}
D^{-2}= \sum_{\n\in \mathbb{N}^{N}}\prod^N_{i=1} \frac{(-1)^{n_i}q^{n_i}}{(q;q)_{n_i}} f_{n_i}(K^2_i)\otimes f_{n_i}(K^2_i)
\end{equation}
For $\fsl_N$ similar computations will give
$$(D')^{-2}= \sum_{\n\in \mathbb{N}^{N-1}}\prod^{N-1}_{i=1} \frac{(-1)^{n_i}q^{n_i}}{(q;q)_{n_i}} f_{n_i}(\K_i)\otimes f_{n_i}(\K^2_i) \ $$
(compare Section B.1 in \cite{HL}).

Let us denote by
$$\text{Inv}\,(\U)=
\{u\in \U\, |\, x \vartriangleright u=\epsilon(x) u \quad\forall
x\in \U\}$$ 
 the invariant part of $\U$ under the adjoint action 
 $x \vartriangleright u:=x_{(1)}u S(x_{(2)})$ in Sweedler
 notation. 
The main advantage of  the usage of bottom tangles
in the definition of $J_L(\fgl_N;q)$
is that in this case $J_L(\fgl_N;q) \in \text{Inv}\,(\U)$
(compare \cite[Sec.4.3]{H}).
As a corollary, we get the following:

\begin{proposition}
\label{prop: JK in completion} Given an $l$-component 
evenly framed link $L$, the 
 universal invariant $J_L(\fgl_N;q)$ 
 is a well defined element of $\text{\rm Inv}\, \left(\widehat{\U}^{\hat{\otimes}l}\right)$.
\end{proposition}
%{\color{red}it should be 0-framed knot?}

\begin{proof}
By definition, $J_K$ is obtained by multiplying together
elementary pieces, such as
$F_\n$, $e_\n$,  $K^{\pm 1}_{2\rho}$,  $D^{\pm 1}$, 
and by
then  taking  a sum over all indices.
The linking between different components and framing will make appear powers of $D^{\pm 2}$ that we can decompose using 
the basis elements $f_n(K^2_i)$ of the completion
by \eqref{eq:D2}.
Note that we can collect all diagonal contributions 
of each component by using   formulas like
$$ D(E_i\otimes 1)D^{-1}=E_i \otimes \K_i  \quad
\text{and} \quad D(1 \otimes F_j)D^{-1}=\K^{-1}_j \otimes F_j \ 
 .$$
%\text{for any}\quad x,y\in \U$$
%where $|e_i|=\alpha_i$, $|F_i|=-\alpha_i$
% $|H_i|=0$ and $\alpha_i=e_i-e_{i+1}$ are the 
% $\fsl_N$ roots in the standard basis $e_i$ of $\mathbb R^{N}$.
%{\color{red}please check} 
Since framing is assumed to be even, we will
have an even number of $D$-parts. 
%and pivotal elements  (by the same argument as in Theorem \ref{thm: zeta RT}(b)). 
Hence using \eqref{eq:D2} and the explicit form of the quasi $R$-matrix $\Theta$, we get the claim.
%the contribution for each component
%  belongs to $ \U^{\text{ev}}$.
\end{proof}

\begin{remark}{\rm
For $\fsl_N$ we can build the same completion
after replacing $K_i$ with $\K_i$. Then the arguments in the proof of Proposition \ref{prop: JK in completion}
will show us that for any algebraically split link
the universal invariants belongs to this completion.
}
\end{remark}

\noindent
{\bf Proof of Theorem \ref{thm:intro even}.}
%To prove (c) and (d) recall that $J_T\in \U^{\otimes \ell}$, 
%{\color{red}this is not true, the diagonal part is in $U_h$, therefore I moved it originally after the discussion of completions}

Using Proposition \ref{prop: JK in completion} and remark
above, we can define the action of $\Gamma$ on each component of $J_L(\fg;q)$ separately. 
We will denote by $J_L^{\zeta_1,\ldots,\zeta_{\ell}}(\fg;q)$ the result of this action. 
%Let $L$ be the link obtained by closure of $T$. 
Then we have
$$
J_L\left(V(\lambda_1)\otimes L(\zeta_1),\ldots,V(\lambda_{\ell})\otimes L(\zeta_{\ell})\right)=
$$
$$
\bigotimes^l_{i=1}\Tr^{V(\lambda_i)\otimes L(\zeta_i)}_{q}\left(J_L(\fg;q)\right)=
\bigotimes^l_{i=1}
\Tr^{V(\lambda_i)}_{q}\left(J_T^{\zeta_1,\ldots,\zeta_{\ell}}(\fg;q)\right)
\cdot \dim_{q}L(\zeta_1)\cdots \dim_{q}L(\zeta_{\ell})
.
$$
The second equation follows from Lemma \ref{lem: twisted action}.
By Theorem \ref{thm: zeta RT} %(a), resp. (b) 
we conclude that 
$$
J_L^{\zeta_1,\ldots,\zeta_{\ell}}\left(\lambda_1, \ldots,\lambda_{\ell}\right)=J_L\left(\lambda_1, \ldots,\lambda_{\ell}\right)
$$
for all $\lambda_1,\ldots,\lambda_{\ell}$ under the assumptions of Theorem \ref{thm:intro even}, therefore $J_L(\fg;q)=J_L^{\zeta_1,\ldots,\zeta_{\ell}}(\fg;q)$ and hence, $J_L(\fg;q)$ is $\Gamma$-invariant under the same assumptions.

%Using Proposition \ref{prop: JK in completion} and remark
%above, we can define the action of $\Gamma$ on both invariants.
%Now Theorem \ref{thm:intro even} follows 
%from Theorem \ref{thm: zeta RT},
%if we 
%observe that the action  of the universal invariant 
%$J_L(\fgl_N;q)$  on $V(\lambda)$ is $J_L(V(\lambda),q)$,
%however the action of $\zeta \cdot J_L(\fgl_N;q)$ on
% $V(\lambda)$ is $J_L(L(\zeta)\otimes V(\lambda),q)$.
$\hfill\Box$
\vspace{2mm}

%{\color{red}the proof as it is written works only for a knot, $J_L$ in general defines an endomorphism of $V(\mu_1) \otimes \dots V_(\mu_l)$ or the same with $L$'s their traces may be different due to the non-trivial
%pivotal on $L$'s}

\begin{corollary}\label{cor:JKeven}
For any $\ell$-component evenly framed  link $L$,
 $J_L(\fgl_N;q)$ belongs to the $\Gamma$-invariant part of $\text{\rm Inv}\left(\widehat{\U}^{\hat{\otimes}\ell}\right)$. 
 Moreover, for every $0$-framed algebraically split link $L$,
 $$J_L(\fgl_N;q)=J_L(\fsl_N;q) \ .$$
\end{corollary}
\begin{proof}
{ The first statement is the direct consequence of Theorem 
\ref{thm:intro even}. The second one follows from the fact that 
the only difference in the definitions of both invariants
is in the diagonal part of the $R$-matrix, that does not contribute since the linking matrix vanishes and  the rules for moving of  $D$ and $D'$ along a component of the link coincide.}
%by replacing $D$ with $D'$ we 
\end{proof}

\subsection{Twist forms}\label{sec:tau}
Let us denote by $\widehat \CZ$ the center of $\widehat{\U}$.
In what follows, we will be particularly interested in the following twist forms
$${\mathcal T}_{\pm }: \widehat\CZ\to\widehat\CZ\quad\text{
given by} \quad
{\mathcal T}_{\pm }(z):=\langle r^{\pm 1}, z\rangle$$
the Hopf pairing with the ribbon element.
On the $\Gamma$-invariant Cartan part they are 
easy to compute, given the Hopf pairing between the generators $H_i$
 in Section \ref{sec:HU} . We have
\begin{equation}
\label{eq:tau}
{\mathcal T}_\pm ( K_{2\ba} )=\langle r^{\pm 1}_0, K_{2\ba} \rangle=
v^{\pm(\ba, 2\rho-\ba)}\in \Z[v,v^{- 1}]\ 
\end{equation}
for any $\ba \in \Z^N$.
%{\color{red}please check}
%The other form is given by  a similar formula.
Now equation \eqref{eq:Hopf_on_U} allows to extend the 
twists form to $\widehat{\U}^{\text{ev}}$ as follows:
$${\mathcal T}_\pm (F_\m\K_\m K_{2\ba}e_\n)=
\delta_{\m,\n}q^{(\rho, \sum_i n_i \alpha_i)}
v^{\pm(\ba, 2\rho-\ba)} \in \Z[v,v^{- 1}]\ $$
where $\alpha_i=e_i-e_{i+1}$ are the simple roots.
Observe that after restriction to $\U_q^{\text{ev}}\fsl_N$,
i.e. replacing $K_{2\ba}$ with $\K_{2\mathbf{b}}$ in the above formula, the result
%${\mathcal T}_\pm ( \K_{2\n} )$ 
{\it  belong to}
$\Z[q,q^{-1}]$ and coincide with
\cite[eq. (102)]{HL} for any $\mathbf{b} \in \Z^{N-1}$.
% since the scalar product we are using here is the standard one on $\mathbb R^N$, i.e.
%$(\m,\m)=\sum_i m^2_i$ need not to be even. 

\section{Habiro's basis for $\CZ(U_q\fsl_2)$}
In this section we summarize Habiro's results for
$\fsl_2$ in the way suitable for our generalization.

Habiro \cite{H}  defined a remarkable family of central elements in $\CZ(U_q\fsl_2)$:
\begin{equation}
\label{eq:sigma}
\sigma_m:=\prod^m_{i=1} \left(C^2-(v^i+v^{-i})^2\right)=\prod_{i=1}^{m}(C-v^{i}-v^{-i})(C+v^{i}+v^{-i}).
\end{equation} 
Since $C$ acts on the $(j+1)$-dimensional representation 
$V_j$ by a scalar $v^{j+1}+v^{-j-1}$, the polynomial $\sigma_m$ is completely characterized by the following properties:
\begin{center}
\begin{itemize}
\item[(a)] (Parity) $\sigma_m$ is $\Gamma=\Z_2$-invariant.
\item[(b)] (Vanishing) $\sigma_m$ annihilates the representations $V_j$ for $j<m$.
\item[(c)] (Normalization) $\sigma_m$  acts  on the representation $V_m$ by a scalar 
\begin{equation}
\label{eq:value}
\prod^m_{i=1} \left((v^{m+1}+v^{-m-1})^2-(v^i+v^{-i})^2\right)\ .
\end{equation}
\end{itemize}
\end{center}
Note that parity implies that $\sigma_m$ also annihilates the representations $L(-1)\otimes V_j$ for $j<m$. 
By using the Harish-Chandra isomorphism, we can alternatively consider the polynomials
$$
T_m(y):=\hc(\sigma_m):=\prod^m_{i=1} (yv^i-y^{-1}v^{-i})(yv^{-i}-y^{-1}v^i)=(-1)^m\prod^m_{i=1} q^{-i}(1-y^2 q^{i})(1-y^{-2}q^{i})
$$
%{\color{blue} $S_m$ needs to be in Rep, in order Hopf pairing to be defined}
which are characterized by the following properties:
\begin{center}
\begin{itemize}
\item[(a)] (Parity) $T_m$ is $\Z_2$-invariant, that is, $T_m(-y)=T_m(y)$
\item[(b)] (Vanishing) $T_m(\pm v^{j+1})=0$ for $j<m$
\item[(c)] (Normalization) $T_m(v^{m+1})$ is given in \eqref{eq:value}.
\end{itemize}
\end{center}

Habiro  proved that $\{\sigma_m\}_{m\geq 0}$ form a basis in (a certain completion of) the $\Gamma$-invariant part of the center. Hence,
 the elements $S_m=\xi^{-1}(\sigma_m)$, given by
$$ S_m:=\prod_{i=1}^{m}(V_1-v^{i}-v^{-i})(V_1+v^{i}+v^{-i})$$
form a basis of $\CR$.
We will show that $$P_n=\prod_{i=0}^{n-1} (V_1-v^{2i+1}-v^{-2i-1})\in \CR$$
is a dual basis to $\{S_m\}_{m\geq 0}$ with respect to the Hopf pairing.
The following is a slight reformulation of \cite[Prop.  6.3]{H}.

\begin{lemma}
We have $$\langle P_n,S_m\rangle =\frac{\{2n+1\}!}{\{1\}}\delta_{n,m}\ .$$
\end{lemma}

\begin{proof}
Clearly, one has
$$
\xi(P_n)=\prod_{i=0}^{n-1} (C-v^{2i+1}-v^{-2i-1})
$$
which annihilates $V_{2i}$ for $i<n$.
We have the following cases:

1) For $n<m$ we have $\langle P_n,S_m\rangle =\Tr_q^{P_n}(\sigma_m)$. Since $P_n$ is in span of $V_i$ for $i\le n$ and  $\sigma_m$ annihilates all these, we get $\langle P_n,S_m\rangle=0$.

2) For $m<n$ we have $\langle P_n,S_m\rangle =\Tr_q^{S_m}(\xi(P_n))$. Since $S_m$ is in span of $V_{2i}$ for $i\le n$ and 
$$\langle P_n, V_{2i}\rangle=\{i+n\}\dots \{i-n+1\}[2i+1]\ .$$
Hence $P_n$ annihilates all these, we get $\langle P_n,S_m\rangle=0$.

3) Finally, for $n=m$ we observe that $P_n$ has a unique copy of $V_{n}$ and  
$$
\langle P_n,S_n\rangle = \langle V_n,S_n\rangle =\Tr_q^{V_n}(\sigma_n)
$$
which is easy to compute.
\end{proof}

We can use the above results to compute the coefficients
in the decomposition of any central element into
$\{\sigma_m\}_{m\geq 0}$.
\begin{lemma}
Let $\phi$ be a $\Z_2$-invariant element in $\CZ(U_q\fsl_2)$ which acts on $V_j$ by a scalar $\phi_j$. Then
$$
\phi=\sum a_n \sigma_n,\; \text{where}\;\;  a_n=\sum_{i=0}^{n} (-1)^{n-i}\frac{\{2i+2\}\{i+1\}}{\{n+i+2\}!\{n-i\}!}\; \phi_i \ .
$$
\end{lemma}

\begin{proof}
We have (\cite[Lemma 6.1]{H})
$$
P_n=\sum_{i=0}^{n} (-1)^{n-i} \frac{[2i+2]}{[n+i+2]}\left[\begin{array}{cc}\! 2n+1\!\\\!n+1+i\!\end{array}\right] V_i.
$$
If $\phi=\sum a_m \sigma_m$ then
$$
a_n=\frac{\{1\}}{\{2n+1\}!}\Tr_q^{P_n}(\phi)=\frac{\{1\}}{\{2n+1\}!}\sum_{i=0}^{n} (-1)^{n-i} \frac{[2i+2]}{[n+i+2]} \left[\begin{array}{cc}\! 2n+1\!\\\!n+1+i\!\end{array}\right] 
\Tr_q^{V_i}(\phi)=
$$
$$
\sum_{i=0}^{n} (-1)^{n-i} \frac{\{2i+2\}\{1\}}{\{n+i+2\}!\{n-i\}!}\dim_q(V_i)\phi_i.
$$
Using $\dim_q(V_i)=[i+1]$ we obtain the result.
\end{proof}

Habiro proved that for any $0$-framed knot $K$, there exist $a_n(K) \in \Z[q,q^{-1}]$ such that
$$J_K(\fsl_2;q) =\sum_{n\geq 0} a_n (K)\, \sigma_n $$
known as a  {\it cyclotomic} expansion of the colored Jones polynomial of the knot $K$.

%{\color{red}add cyclotomic expansion for links  .....?}

\section{New basis for the center of $\widehat \U$}
\label{sec:results}
Recall that $\widehat \CZ$ is the center of the completion
$\widehat \U$.
In this section we construct the basis 
$\{\sigma_\lambda\}_\lambda$
of the $\Gamma$-invariant part of $\widehat\CZ$. %diagonalizing the Hopf pairing.  
Furthermore, we explicitly define its  dual $\{P_\lambda\}_\lambda$ with respect to the Hopf pairing.
 This  allows us to construct the cyclotomic expansion of $J_K(\fgl_N;q)$ for any $0$-framed knot $K$.

%We also prove Theorem \ref{thm:basis} for $\U$
%that is a reformulation of Theorem
%\ref{thm: intro basis} in the introduction. 

The proof uses the existence and properties of interpolation Macdonald polynomials \cite{Ok} which are summarized in the following theorem.

\begin{theorem}\label{thm:F's}
There is a family of symmetric polynomials $F_{\lambda}(x_1,\ldots,x_N;q)$ such that:
\begin{itemize}
\item[(a)] $F_{\lambda}$ is in the span of Schur functions $s_{\mu}$ for $\mu\leq \lambda$ with the leading term $$F_{\lambda}=(-1)^{|\lambda|+\binom{N}{2}}  q^{D_N(\lambda)}s_{\lambda}+\ldots \ .$$
\item[(b)] $F_{\lambda}(q^{-\mu_1-N+1},\ldots,q^{-\mu_N})=0$ unless $\mu$ contains $\lambda$. 
\item[(c)] $F_{\lambda}(q^{-\lambda_1-N+1},\ldots,q^{-\lambda_N})=(-1)^{\binom{N}{2}}q^{n(\lambda)+\binom{N}{3}}\prod_{\sq\in \lambda}(1-q^{-h(\sq)})$.
\item[(d)] Any function $F$ in the completion  can be written as 
\begin{equation}
\label{eq: interpolation}
F(x_1,\ldots,x_N)=\sum_{\lambda, \;\mu\subset \lambda} d_{\mu,\lambda}(q) F(q^{-\mu_1-N+1},\ldots,q^{-\mu_N}) F_{\lambda}(x_1,\ldots,x_N;q)
\end{equation}
where $d_{\lambda,\mu}$ are explicit coefficients prescribed by Theorem \ref{thm: dlm}.
\end{itemize}
\end{theorem}
%{\color{blue}  coeff in (c) should be modified}
We discuss the definition and give more details on interpolation Macdonald polynomials in Section \ref{sec:inter}.

\begin{theorem}\label{thm:basis}
There exists a family of central elements $\sigma_{\lambda}\in \CZ$  with the following properties:
\begin{itemize}
\item[(a)] $\sigma_{\lambda}$ is $\Gamma$-invariant and annihilates  $L(\zeta)\otimes V(\mu)$ for all  $\mu$ not containing $\lambda$ and $\zeta\in \Gamma$.
\item[(b)] $\hc(\sigma_{\lambda})$ is in the span of $s_{\mu}(x_1,\ldots,x_N)$ for $\mu\leq \lambda$, 
%and the top degree $\hc(\sigma_{\lambda})$ equals 
with the leading term $$\hc(\sigma_{\lambda})=(-1)^{|\lambda|+\binom{N}{2}}v^{(N-1)|\lambda|}q^{D_N(\lambda)}s_{\lambda}+\ldots \ .$$
\item[(c)] $\sigma_{\lambda}$ acts on $V(\lambda)$ by a scalar 
$$\sigma_{\lambda}|_{V(\lambda)}=(-1)^{\binom{N}{2}}q^{-n(\lambda)-\binom{N}{3}}\prod_{\sq\in \lambda}(1-q^{h(\sq)})\ .$$
\end{itemize}
\end{theorem}

\begin{proof}
Define $\sigma_{\lambda}=\hc^{-1}(g_{\lambda})$, where $g_{\lambda}(x_1,\dots, x_N)=F_{\lambda}(v^{N-1}x_1,\ldots,v^{N-1}x_N;q^{-1})$. Then $\sigma_{\lambda}$ is clearly $\Gamma$-invariant and
$$
\sigma_{\lambda}|_{L(\zeta)\otimes V({\mu})}=g_{\lambda}(\zeta_i\cdot v^{\mu_i+\rho_i})=
F_{\lambda}(q^{(\mu_1+N-1)},\ldots, q^{\mu_N};q^{-1}).
$$
Indeed, if $y_i=\zeta_i\cdot v^{\mu_i+\rho_i}=\zeta_iv^{(\mu_i-\frac{N-1}{2}+N-i)}$ then $v^{N-1}y_i^2=q^{(\mu_i+N-i)}$.

Now $F_{\lambda}(q^{(\mu_1+N-1)},\ldots, q^{\mu_N};q^{-1})$ vanishes unless $\mu$ contains $\lambda$, and has the nonzero value prescribed by the previous theorem for $\mu=\lambda$.
\end{proof}
Let us define $\CR_\Q:=\CR\otimes \Q(v)$ by extending the
coefficient ring $\Z[v^{\pm 1}]$ of $\CR$ to the rational functions in $v$.

\begin{theorem}\label{thm:pi's}
Define the following formal elements of $\CR_\Q$
$$
P_{\lambda}=\sum_{\mu\subset \lambda}
\frac{d_{\lambda, \mu}(q^{-1})}{\dim_q V(\mu)}\; V(\mu)\in \CR_\Q,
$$
then  one has
\begin{equation}
\langle P_{\lambda}, \sigma_{\nu}\rangle :=
\Tr^{P_{\lambda}}_q(\sigma_{\nu})=\delta_{\lambda,\nu}\ . 
\end{equation}
\end{theorem}

\vspace{2mm}
\begin{proof}
First, let us write the interpolation formula \eqref{eq: interpolation} for $F=F_{\nu}$:
$$
F_{\nu}(x_1,\ldots,x_N;q)=\sum_{\mu\subset \lambda} d_{\lambda,\mu}(q) F_{\nu}(q^{-\mu_1-N+1},\ldots,q^{-\mu_N};q) F_{\lambda}(x_1,\ldots,x_N;q),
$$
so 
$$
\sum_{\mu\subset \lambda} d_{\lambda,\mu}(q) F_{\nu}(q^{-\mu_1-N+1},\ldots,q^{-\mu_N};q)=\delta_{\lambda,\nu}.
$$
By changing $q$ to $q^{-1}$ we get
\begin{equation}
\label{eq:ortho}
\sum_{\mu\subset \lambda} d_{\lambda,\mu}(q^{-1}) F_{\nu}(q^{\mu_1+N-1},\ldots,q^{\mu_N};q^{-1})=\delta_{\lambda,\nu}.
\end{equation}
Now $\Tr^{V(\mu)}_q(\sigma_{\nu})=\dim_{q}(V(\mu))\; g_{\nu}(q^{\mu_1+N-1},\ldots,q^{\mu_N})$, hence
$$
\Tr^{P_{\lambda}}_q(\sigma_{\nu})=\sum_{\mu\subset \lambda}\frac{d_{\lambda,\mu}(q^{-1})}{\dim_q V({\mu})}\Tr^{V({\mu})}_q (\sigma_{\nu}) = \; \delta_{\lambda, \nu} \ .
$$
%It remains to notice that $\Tr^{V({\mu})}_q(\sigma_{\nu})=g_{\nu}(q^{\mu_1+N-1},\ldots,q^{\mu_N})\dim_q(V({\mu})).$
\end{proof}

Next, we would like to study the integrality properties of the universal knot invariant.

 \begin{lemma}
\label{lem: completion center}
(a) Let $\sigma \in\U^{\text{ev}}_{\Z}$. Then  $\sigma=(K_1\cdots K_N)^{-2s}\sum a_{\lambda}\sigma_{\lambda}$ with  $a_{\lambda}\in \Z[q,q^{-1}]$.

(b) Given $k$ and $m$, there exists $n=n(k,m)$ such that for all $\Gamma$-invariant central elements $\sigma$ in the ideal $\U^{(n)}_{\Z}$ the coefficients 
$a_{\lambda}$ are divisible by $(q;q)_{m}$  for $|\lambda|\le k$. 
\end{lemma}

\begin{proof}
(a) Recall that Harish-Chandra transform $\hc$ identifies the $\Gamma$-invariant part of the center of $\U_{\Z}$ with the space of symmetric functions in $x_1,\ldots,x_N$ with coefficients in $\Z[q,q^{-1}]$.  Since $F_{\lambda}$ is a polynomial with top degree part equal to the Schur polynomial (up to a monomial in $q$),  we can write $(x_1\cdots x_N)^{s}f(x_1,\ldots,x_N)=\sum_{\lambda}a_{\lambda}F_{\lambda}(x_1,\ldots,x_N;q^{-1})$ and the result follows.

(b) If $\sigma$ is in the ideal $\U^{(n)}_{\Z}$  for sufficiently large $n$, then by Proposition \ref{prop: filtration} its matrix elements in the integral basis of $V(\lambda)$ are divisible by $(q;q)_{m}$. By definition of Harish-Chandra transform, this implies that the values $f(q^{-\lambda_1-N+1},\ldots,q^{-\lambda_N})$ are divisible by $(q;q)_{m}$ and hence by the interpolation formula  \eqref{eq: interpolation} the coefficients $a_{\lambda}$ are divisible by $(q;q)_{m}$ as well.
\end{proof}

\begin{corollary}
The center of the completion $\widehat{\U}$ is isomorphic to the completion of the space of symmetric polynomials with coefficients in $\widehat{\Z[v]}$ with respect to the basis $F_{\lambda}$. 
\end{corollary}

\begin{proof}
By Lemma \ref{lem: completion center} any element of the center of $\widehat{\U}$ can be written as an infinite series $\sum a_{\lambda}F_{\lambda}$
with coefficients in $\widehat{\Z[v]}$, up to a factor $(x_1\cdots x_N)^{-s}$. By Corollary  \ref{cor: inverse product} the multiplication by $(x_1\cdots x_N)^{-s}$ preserves the space of such series.
\end{proof}

\begin{corollary}\label{cor:main}
Any $\sigma\in \widehat{\U^{\text{ev}}}$  can be written as an infinite sum  $\sigma=\sum a_{\lambda}\sigma_{\lambda}$ with coefficients $a_{\lambda}=\Tr^{P_{\mu}}_{q}(\sigma)\in \widehat{\Z[q]}$.
\end{corollary}

\begin{proposition}\label{pro:main}
The   universal knot invariant admits an expansion $$J_K(\fgl_N;q)=\sum_\lambda a_{\lambda}(K)\sigma_{\lambda}
\quad \text{with} \quad a_{\lambda}(K)= \sum_{\mu\subset \lambda}
{d_{\lambda, \mu}(q^{-1})}\, J_K(V(\mu),q)
\in \Z[q,q^{-1}]\ $$
called a {\it cyclotomic} expansion of the universal $\fgl_N$ knot
invariant.
\end{proposition}

 Proposition \ref{pro:main} 
implies Theorem \ref{thm:main} in Introduction.
 Note that the knot invariant $J_K(V(\mu),q)$
 is normalized to be 1 for the unknot.

\begin{proof}
By Corollary \ref{cor:JKeven}, $J_K(\fgl_N;q)$ is a central element in $\widehat{\U^{\text{ev}}}$, so it can be written as $\sigma=\sum a_{\lambda}\sigma_{\lambda}$ with coefficients $a_{\lambda}\in \widehat{\Z[q]}$. On the other hand, the value of $J_K$ on any representation $V_{\lambda}$ is in $\Z[q,q^{-1}]$, so by the interpolation formula \eqref{eq: interpolation} the coefficients $a_{\lambda}$ can be written as rational functions with numerators in $\Z[q,q^{-1}]$ and cyclotomic denominators. By Proposition \ref{prop: appendix} this implies that $a_{\lambda}\in \Z[q,q^{-1}]$.   The explicit formula for $a_\lambda$ is obtained
by taking Hopf pairing with $P_\mu$
and observing that $\Tr^{V({\mu})}_{q} \left(J_K(\fgl_N;q)\right)=
\dim_q (V(\mu) )J_K(V(\mu),q)$ according to our normalization.
\end{proof}

%Let us denote by  
%$$P'_{\lambda}=\sum_{\mu\subset \lambda}
%{d_{\lambda, \mu}(q^{-1})}\; V(\mu)$$
%The last result shows that $a_\lambda(K)=J_K (P_\lambda,q)\in \Z[q,q^{-1}]$, even through the coefficients
%$d_{\lambda, \mu}(q)$  are  rational  functions in $q$ (compare Example \ref{d's}).

%\begin{corollary}
%For any $\Gamma$-invariant element $\sigma\in \widehat\CZ$ one has
%$$
%\sigma=\sum_{\mu\subset \lambda}
%\Tr^{P_{\mu}}_{q}(\sigma)\sigma_{\lambda}.
%$$
%\end{corollary}

The last result shows that $a_\lambda(K)=\Tr_q^{P_\lambda}(J_K(\fgl_N;q))\in \Z[q^{\pm 1}]$, even through the coefficients
$d_{\lambda, \mu}(q)$  are  rational  functions in $q$ (compare Example \ref{d's}).

\section{Unified  invariants of integral homology 3-spheres}\label{sec:IM}
%In this section 
%we first recall the
% Habiro and Le construction \cite{HL} of the unified invariants
%of integral homology 3-spheres, domination the WRT ones.
This section is devoted to our main application of the previous results  --- a
construction of the
 unified invariants for integral homology 3-spheres.
We start with few auxiliary  results.
%3-manifolds obtained by $(\pm 1)$-surgeries on a knot. 
%have values in the Habiro ring.
%{\color{blue} To extend this to link surgeries we will need to extend the  Hopf pairing to $\U_N$}
%\subsection{Unified 3-manifold invariants}
%Let $r\in \CZ$ be the ribbon element.

Let us denote by  
$$P'_{\lambda}=v^{-|\lambda|}\dim_q V(\lambda)
\sum_{\mu\subset \lambda}
\frac{d_{\lambda, \mu}(q^{-1})}{\dim_q V(\mu)} \;  V(\mu)\in \CR_\Q$$
and  define
\begin{equation}
\label{eq:Kirby}
\omega_{\pm}=\sum_\lambda (-1)^{|\lambda|+\binom{N}{2}} 
q^{\mp c(\lambda)}
q^{w_{\pm}(\lambda)}P'_\lambda \in \widehat \CR_\Q
\quad\text{with}\quad 
\begin{array}{ll}w_{+}(\lambda)&=D_N(\lambda)\\
w_{-}(\lambda)&=D_N(\lambda)+
N|\lambda| \end{array}
 \end{equation}
where $c(\lambda)$ is the content of $\lambda$.
The next Lemma implies that $\omega_{\pm}$ is the universal Kirby color
for $(\pm 1)$-surgery. 

\begin{lemma}\label{lem:Kirby}
 For any $x\in \widehat\CR_\Q$, we have
\begin{equation}
\label{eq:H}
\langle  \omega_\pm, x\rangle=  J_{U_\mp}(x)=
\langle r^{\pm 1}, \xi(x)\rangle 
\end{equation}
where $ J_{U_\pm}(x)$ is the Reshetikhin--Turaev   invariant of the
$(\pm 1)$-framed unknot colored by $x$. 
%and normalized
%$\tilde J_U(V(\lambda))=\dim_q(V(\lambda))$.
\end{lemma}
\begin{proof}
%{\color{red} please check}
It is enough to check \eqref{eq:H} for the basis
elements 
 $x=V(\nu)$.
 %\frac{V(\nu)}{\dim_q(V(\nu))}$.
%=\frac{V(\nu)}{{\dim_q(V(\nu))}}$,
We compute
\begin{align*}
\langle P'_\lambda, V(\nu)\rangle& = v^{-|\lambda|}\dim_q V(\lambda)\sum_{\mu\subset \lambda}
\frac{{d_{\lambda, \mu}(q^{-1})}}{\dim_q V(\mu)}\; 
\langle V(\mu), V(\nu)\rangle=
v^{-|\lambda|} \dim_q V(\lambda) \sum_{\mu\subset \lambda}
{d_{\lambda, \mu}(q^{-1})} s_\nu(q^{\mu_i+N-i})\\
&=
\dim_q V(\lambda) \sum_{\mu\subset \lambda}
{d_{\lambda, \mu}(q^{-1})} C_{\nu}F_\nu(q^{\mu_i+N-i})
 = {C_\lambda}
\delta_{\lambda, \nu}  \dim_q V(\lambda)
\end{align*}
where we used Lemma \ref{lem:Hopf}, equation \eqref{eq:ortho} and the expansion
$s_\nu= (-1)^{|\lambda|+\binom{N}{2}} q^{-D_N(\lambda)}v^{(1-N)|\lambda|}F_\nu + \text{lower terms}$ and hence,
$$C_\lambda=(-1)^{|\lambda|+\binom{N}{2}} q^{-D_N(\lambda)}v^{-N|\lambda|}\ .$$
% $b=|\lambda|+\binom{N}{2}$.
Using this computation it is easy to check that
\begin{equation}\label{eq:Kirby1}
 \langle\omega_{\pm}, V(\nu)\rangle= v^{\mp N|\nu|}
 q^{\mp c(\nu)} \dim_q V(\nu)=
 v^{\mp (\nu, \nu+2\rho)} \dim_q V(\nu)=\Tr^{V(\nu)}_q(r^{\pm 1})
 \end{equation}
%{\color{red} should be $r^{\mp1}$ in the last equation}
is equal to  $ J_{U_\mp}(V(\nu))$.
%the Hopf pairing of $r^{\pm 1}$ with $s_\lambda$.
\end{proof}

From the following computation for $V'(\nu)=\frac{V(\nu)}{\dim_q(V(\nu))}$
%for any $V'_\nu= \frac{V(\nu)}{\dim_q(V(\nu))}$
$$\langle \omega_+\omega_-, V'(\nu)  \rangle=
\langle \omega_+, V'(\nu) \rangle
\langle \omega_-, V'(\nu)\rangle=
v^{ (\nu, \nu+2\rho)} v^{- (\nu, \nu+2\rho)}=
\langle 1,V'(\nu)\rangle$$
we see that $\omega_+$ and $\omega_-$ are inverse to each other in the algebra $\widehat \CR_\Q$ isomorphic to $\widehat \CZ_\Q$.

A direct consequence of the above Lemma and the fusion rules is the following result.

\begin{theorem}\label{thm:topinv}
Let $L\cup K=L_1\cup L_2 \dots \cup L_l\cup K$
be an $(l+1)$ component algebraically split $0$-framed link such that $K$ is the unknot. We denote by
$L_{(K,\pm 1)}$ the framed link in $S^3$ obtained from $L$ by $\pm 1$-surgery along $K$, then for any
$p_1, \dots, p_l\in  \CR$
$$ J_{L\cup K}(p_1, \dots, p_l, \omega^{\pm 1})=
J_{L_{(K,\pm 1)}}(p_1, \dots, p_l) .$$
\end{theorem}

\begin{proof}
The proof is given in \cite[Thm. 9.4]{H}.
%{\color{red} should we summarize the proof for completeness?
%Then we need Figure 9.1 of Habiro.}
\end{proof}

\subsection{Construction of the unified invariants}
Without loss of generality, we can assume that 
an integral homology 3-sphere $M$ is obtained by $\e$-surgery
on an $\ell$ component algebraically split link $L$, 
where $\e \in \{\pm 1\}^\ell$.
For $\fsl_N$ Habiro and Le defined a unified invariant
    of $M$ as follows 
$$I^{\text{HL}}(M):= \mathcal T'_{\e} (J_L(\fsl_N;q))\; \in
\;\widehat{\Z[q]}  \quad\text{
where}  \quad\mathcal T'_{\varepsilon}=
\bigotimes^\ell_{i=1} \mathcal T'_{\varepsilon_i} \ $$
is the $\fsl_N$ twist form.
They also proved that $I^{\text{HL}}(M)$  belongs to the Habiro ring
\cite{HL}.
%is acting on the $i$th tensor power of $
%J_L(\fgl_N;q)$.

For $\fgl_N$ we define the  unified invariant of $M$
similarly
$$I(M):= \mathcal T_{\e} (J_L(\fgl_N;q))\; \in
\;\widehat{\Z[q]}  \quad\text{
where}  \quad\mathcal T_{\varepsilon}=
\bigotimes^\ell_{i=1} \mathcal T_{\varepsilon_i} \ $$
by using the $\fgl_N$ twist forms.

%\mathcal T_{\pm} (J_K(\fgl_N;q))=\langle  r^{\e},J_L(\fgl_N;q)\rangle  \ $$
%where $r^\e=\otimes_i r^{\varepsilon_i}$ is the ribbon element.
% and $J_K(\fgl_N;q)$
%is the universal invariant of the $0$-framed knot $K$.
\begin{theorem}\label{thm:linksurgery}
For any integral homology $3$-sphere $M$,
$$I(M)=J_L(\omega_{\varepsilon_1}, \dots, \omega_{\varepsilon_l})\in  \;\widehat{\Z[q]}   $$
Moreover, its evaluation
at any root of unity coincides with the
$\fsl_N$  Witten--Reshetikhin--Turaev invariant of $M$.
%a topological invariant of $M$.
\end{theorem}
This implies Theorem \ref{thm:main2} from Introduction.

\begin{proof}
By Corollary \ref{cor:JKeven}, 
for any algebraically split $0$-framed link $L$ we have
\begin{equation*}\label{eq:equality}
J_L(\fgl_N;q)=J_L(\fsl_N;q) .
\end{equation*}
 Hence,    as explained in Section \ref{sec:tau},
 the $\fgl_N$
and $\fsl_N$ twist forms on $J_L$ do coincide.
%(see ). 
This implies $I(M)=I^{\text{HL}}(M)$.
Since the Habiro--Le invariant  
is known to belong
to the Habiro ring and to
 evaluate at a root of unity to the
Witten-Reshetikhin-Turaev (WRT) one,
it remain to show
$I(M)=J_L(\omega_{\epsilon_1}, \dots, \omega_{\epsilon_\ell})$.
%we deduce that the $\fgl_N$ and $\fsl_N$ WRT
 %invariants of $M$ coincide. 
We prove this claim in two steps.

Step 1: Assume $\ell=1$, then $J_L(\omega_\pm
)=I(M)$ by Lemma \ref{lem:Kirby}.

Step 2: For any $ x=x_1\otimes\dots\otimes x_\ell \in 
\text{Inv}(\widehat{\U}_\Z^{\hat{\otimes}\ell})$
we define $a_k$ for $k=0,1,\dots, \ell$ and $b_k=1,\dots, \ell$ as follows:
$$a_k=\prod^{k}_{i=1}\langle  r^{\epsilon_i}, x_i\rangle
\prod^{\ell}_{j=k+1}\langle \omega_{\epsilon_j},x_j\rangle,
\quad
b_k=x_k\prod^{k-1}_{i=1}\langle  r^{\epsilon_i}, x_i\rangle
\prod^{\ell}_{j=k+1}\langle \omega_{\epsilon_j},x_j\rangle .
$$ 
Then 
$$ a_{k-1}=\langle \omega_{\epsilon_k}, b_k\rangle
\quad \text{and}\quad
a_{k}=\langle r^{\epsilon_k}, b_k\rangle .$$
where we identify $\omega_\pm$ with their image under $\xi$ for simplicity.
Since $b_k\in \CZ_\Q$, we have $a_k= a_{k-1}$ by Step 1
for $k=1,2, \dots, \ell$. Hence, we have $a_0= a_\ell$
which is our claim.
%Therefore, $I(M_\pm)$ is equal
%to the $\fsl_N$ Habiro-Le invariant, for which 
%all statements 
%but $I(M_\pm)=J_K(\omega_\pm,q)$
%are known to be true. 
\end{proof}

%One of the main  applications of the cyclotomic expansion is the following result.

Theorem \ref{thm:linksurgery}
has striking consequences. Indeed, for any $\lambda\in \CR$ let us denote
$$\CP_\lambda=\text{\rm Span}_{\Z[q^{\pm 1}]}
\{P'_\mu \,|\, \lambda\subset \mu\},
\quad \CP=\CP_{\emptyset}\quad \text{and}\quad 
\widehat{\CP}:=
{\lim\limits_{\overleftarrow{\hspace{2mm}\lambda\hspace{2mm}}}}\; \;  
\frac{\CP}{\CP_\lambda} .$$
For any framed link $L$,
the Reshetikhin--Turaev functor provides a $\Q(v)$-multilinear map
$ J_L: \CR_\Q\times \dots\times \CR_\Q\to \Q(v)$.
% normalized to 
%$\prod_i\dim_q (V(\mu_i))$ for the $0$-framed
%$(\mu_1, \dots, \mu_l)$-colored unlink.
For any algebraically split $0$-framed link $L$,
 Theorem \ref{thm:linksurgery} implies that its
restriction to $\CP$ provides
a $\Z[q^{\pm 1}]$-multilinear map
$$ J_L: \CP\times \dots\times \CP\to \Z[q,q^{-1}]$$
inducing
$$ J_L: \widehat\CP\times \dots\times \widehat\CP\to \widehat{\Z[q]} \ .$$
This leads to a generalization of  the famous integrability theorem in \cite[Thm. 8.2]{H}.

\begin{corollary} \label{cor:int_link}
Given an $\ell$ component algebraically split $0$-framed link $L$, then
for all but finitely many
partitions $\lambda_i$ with $1\leq i\leq \ell$, 
there exist positive integers $n=n(\lambda_i, N)$, such that 
 $$
 J_L(P'_{\lambda_1}, \dots , P'_{\lambda_\ell})
\in (q;q)_n \Z[q,q^{-1}]\ . $$
\end{corollary}
%where the link invariant is normalized to be 1 for the unlink.
%The authors do not know

It would be interesting to have a
 direct proof of Corollary \ref{cor:int_link}
 without using the 
theory of unified invariants.

Based on Corollary \ref{cor:int_link} we can give a 
cyclotomic expansion of the Reshetikhin--Turaev invariant of $L$ as follows:
\begin{equation}
J_L(\lambda_1, \dots, \lambda_\ell)= v^{\sum_i |\lambda_i|}\sum_{\mu_i\subset \lambda_i}
\prod^\ell_{j=1} c_{\lambda_j,\mu_j}(q^{-1})\,
J_L(P'_{\mu_1}, \dots, P'_{\mu_\ell})
\end{equation}
where the matrix $
\left[c_{\lambda,\mu}(q)\right]_{\lambda,\mu}:=\left[F_{\lambda}(q^{-\mu_i-N+i})\right]_{\lambda,\mu}
$ is the inverse of $\left[d_{\lambda,\mu}(q)\right]_{\lambda,\mu}$ by Theorem \ref{thm: dlm} below.
This generalizes equation $(8.2)$ in \cite{H}.

%
%\begin{conjecture}

%\end{conjecture}

\subsection{Few direct arguments}%{\color{red}we can also move all that in the appendix}
 Our proof of the fact that  $I(M)$ belongs to the Habiro ring is based on the result
 that $I^{\text{HL}}(M)\in \widehat{\Z[q]}$
 proven in \cite{HL} on more than 100 pages. 
Given the complexity of their argument, we decided
to collect here different facts that can be shown
without reference to \cite{HL}.

%, then we can reprove a slightly weaker fact directly using the interpolation theory developed in this paper. 

\begin{theorem} Assume $M_\pm$ is obtained by $(\pm 1)$-surgery on the knot $K$, then 
%the unified invariant
$$I(M_\pm)=J_K(\omega_\pm) \;\in \;\widehat{\Z[v]} $$
belongs to  the Habiro ring.
\end{theorem}

\begin{proof}
By Theorem \ref{thm:main} we know 
$$ J_K(\fgl_N;q )=\sum_\mu a_\mu(K) \sigma_\mu
\quad\text{ with} \quad a_\mu(K) \in \Z[q,q^{-1}]
$$
The fact that $I(M_\pm)$ belongs to the Habiro ring
 easily follows from the  claim that $\mathcal T_{\pm}(\sigma_\mu)$ is divisible
by $(v;v)_m$ for some $m$ depending on $\mu$ and $N$.
Let us prove this claim. 
By \eqref{eq:tau} 
 the Hopf pairing with $r^{\pm 1}$ replaces an element
$x^{k_i}_i$ 
%in  $\CZ(U^{\text{ev}}_q \fsl_N)$ 
   with
$v^{Q_\pm(k_i)}$ where $Q_\pm$ is a quadratic form.
By Lemma \ref{lem:div} we can rewrite $\sigma_\mu$
as  a linear combination of $\prod^d_{i=1}f_{n_i}(q^{s_i}x_i)$ such that
$\sum_{i} n_i=|\mu|$, $d=N(N+1)/2$ and $s_i\in \Z$.
%Moreover, each $f_n(q^a x_i)$
%that $\sigma_{\mu}(q^{\mu_i+N-i})$ 
%is divisible
%by $(q;q)_k$ for some $k$ depending on $\mu$ and $N$.
Moreover, each $f_n(q^ax_i)$ is divisible by 
$f_n(v^a y_i)$ where $y_i^2=x_i$ and 
hence 
belongs to the ideal $I_n$ of $\Z[v^{\pm 1},y^{\pm 1}_i]$
characterized in Proposition 2.1 of \cite{BCL}.
The result follows now from \cite[Theorem 2.2]{BCL}.
The number $m$ we are looking for is  $\left\lfloor\frac{|\mu|}{N(N+1)}\right\rfloor$.

%that Hopf pairing with $\omega_\pm$ has the same effect
%on the formal expression $V(\mu)$ 
\end{proof}

Combining  previous  results
we obtain an explicit expression for the unified 
invariant for knot surgeries:
%in terms of the coefficients $a_\mu(K):=J_K(P'_\mu,q)$ 
%of the cyclotomic expansion:
\begin{equation}
\label{eq:IM}
I_{M_\pm}= J_K(\omega_\pm)=
\Tr_q^{\omega_{\pm}}J_K(\fgl_N;q)=
\sum_{\lambda} 
 (-1)^{|\lambda|+\binom{N}{2}} 
q^{\mp c(\lambda)}
q^{w_{\pm}(\lambda)} \, J_K(P'_\lambda)
  \end{equation}
%where $ \tilde J_K(V(\mu))=\dim_q (V(\mu)) J_K(V(\mu),q)$.
%{\color{red} Can we express this through $a_\lambda(K)$?}

%$$ I(M):=J_L(\omega_{\varepsilon_1}, \dots, \omega_{\varepsilon_l})$$
% thus $I(M)$ is the Reshetikhin-Turaev
%invariant of $L$  colored by
%$\omega_{\varepsilon}$. 

Assuming that $I(M)=J_L(\omega_{\varepsilon_1}, \dots, \omega_{\varepsilon_l})$ is well defined,  its
topological invariance can be shown directly as follows.
Since $I(M)$ depends only on the isotopy class of $L$,
it remains to check its invariance under Hoste moves (a version of Fenn-Rourke moves between algebraically split links). Without loss of generality, we can assume that the last component is an unknot, then the statement follows from Theorem \ref{thm:topinv}. 

%Now instead
%of showing convergence of $I(M)$ directly, we will
%prove that it is equal to $I^{\text{HL}}(M)$, and hence also belongs to the Habiro ring.

Assuming $I(M)$  belongs
to the Habiro ring, and as such has well defined
 evaluations at  roots of unity
\cite[Thm. 6.3]{Hpoly} we can 
use the same trick as above to show
 that for any root of unity $\zeta$
$$\text{ev}_\zeta I(M)=\text{WRT}(M,\zeta).$$
%To check that second fact let us
Let us recall that the WRT
invariant is obtained from $J_L(\fsl_N;q)$ by
taking  trace along each component  with the Kirby color
$$\Omega_\pm=\frac{\sum_\lambda v^{\pm (\lambda, \lambda+2\rho)}\dim_qV (\lambda) }{\sum_\mu v^{\pm(\mu, 2\rho+\mu)} \dim^2_q V(\mu) } \, V(\lambda) $$
where  the  sums
are taken over all $\lambda, \mu\in \CR^{\text{fin}}=\{\lambda| \dim_\zeta V(\lambda) \neq 0\} $
and  $v^2$ is evaluated to $\zeta$.
% with non-vanishing quantum dimension. 
Hence, we need to show that for any 
$x$ in the ad-invariant part of the completed $\ell$th tensor power of  
$\widehat{\U}_\Z$,
%that is invariant under the adjoint action
%(this is where $J_L$ belongs to 
we have 
%({\color{red} move $\zeta$ over/under $=$})
 $$\Tr_q^{\Omega_\varepsilon} (x) \stackrel{\zeta}{=}
\Tr_q^{\omega_\varepsilon} (x)\quad
\quad \forall x \in 
\text{Inv}(\widehat{\U}_\Z^{\hat{\otimes}\ell}) $$
where 
$\stackrel{\zeta}{=}$ means the equality after evaluation $v^2=\zeta$.
 We will prove this fact in two steps.

{ Step} 1: Assume $\ell=1$, in this case  $\text{Inv}\,\widehat{\U}_\Z=\widehat \CZ$ with basis given by $z_\lambda=
\xi(V(\lambda))$.
% and it is enough to consider $x\in\xi(\CR^{\text{fin}})$.
Since $\Omega_\pm$ is invariant under Hoste moves, we have
$$\Tr_q^{\Omega_\pm}\left(z_\nu\right)=\langle\Omega_\pm, V(\nu)\rangle=\text{ev}_\zeta \left(v^{\mp (\nu, \nu+2\rho)} \dim_q V(\nu)\right)$$
where we interpret the left hand side as a Hopf link with  components colored by $\Omega_\pm$ and  
 $V(\nu)$, and the right hand side
is the result of the sliding. Comparing this computation with
 \eqref{eq:Kirby1}, we deduce that at roots of unity the actions of
 $\Omega_\pm$ and $\omega_\pm$ 
 do coincide on 
%$V(\nu)$ from 
  $\xi(\CR^{\text{fin}})$, and they vanish
  on $x\in \widehat{\CZ}\setminus \xi(\CR^{\text{fin}})$
  after evaluation.
%  In particular, this shows that 
 %This implies the claim since $V(\nu)$ is a basis for
 
{ Step} 2:  Define $a_k$ for $k=0,1,\dots, \ell$ and $b_k=1,\dots, \ell$ as follows:
$$ a_k=\bigotimes^{k}_{j=1} \Tr_q^{\Omega_{\varepsilon_j}}\otimes
\bigotimes^\ell_{j=k+1}\Tr_q^{\omega_{\varepsilon_j}}(x),
\quad
b_k=\bigotimes^{k-1}_{j=1} \Tr_q^{\Omega_{\varepsilon_j}}\otimes
1\otimes\bigotimes^\ell_{j=k+1}\Tr_q^{\omega_{\varepsilon_j}}(x).
$$
 Then
$$a_{k-1}=\Tr_q^{\omega_{\varepsilon_k}}( b_k )
\quad\text{and}\quad
a_{k}=\Tr_q^{\Omega_{\varepsilon_k}} (b_k ) .
$$
Since $b_k\in \CZ_\Q$, we have $a_k\stackrel{\zeta}{=} a_{k-1}$ by Step 1 and Lemma \ref{lem:Kirby}
for $k=1,2, \dots, \ell$. Hence, we have $a_0\stackrel{\zeta}{=}a_\ell$
which is our claim.
%The last equality  follows from \eqref{eq:H}.
%\end{proof}

\section{Interpolation polynomials}\label{sec:inter}
In this section we summarize the theory of interpolation Macdonald polynomials.
\subsection{One variable case}

Consider the space of polynomials in one variable $x$ over $\C(q)$ with the following bilinear form
$$
(x^k,x^m)=q^{-km}.
$$
Let us define polynomials $f_m(x),m=0,1,\ldots$ by the equation $f_0(x)=1$ and
\begin{equation}
\label{eq: def f}
f_m(x)=(x;q)_{m}=(1-x)\cdots (1-xq^{m-1})\quad\text{for}\quad m\ge 1.
\end{equation}
Clearly, $f_m(x)$ is a degree $m$ polynomial with leading term $(-1)^{m}q^{\frac{m(m-1)}{2}}x^m$, so $\{f_m\}_{m\geq 0}$ form a basis in $\Z_{q,q^{-1}}[x]$.
Our next aim is to show that this basis is orthogonal.
Observe that $f_m(q^{-k})=0$ for $k<m$.

\begin{lemma}
\label{lem: orthogonal 1d}
We have $(f_m(x),f_k(x))=\delta_{km}q^{-m}(q;q)_{m}$
\end{lemma}

\begin{proof}
First, observe that $(g(x),x^k)=g(q^{-k})$ for any polynomial $g(x)$. Therefore
for $m>k$ we have $(f_m(x),x^k)=f_m(q^{-k})=0$, so $(f_m(x),g(x))=0$ for any polynomial $g(x)$ of degree strictly less than $m$. In particular, $(f_m(x),f_k(x))=0$ for $m>k$ and 
$$
(f_m(x),f_m(x))=(-1)^{m}q^{\frac{m(m-1)}{2}}(f_m(x),x^m)=(-1)^{m}q^{\frac{m(m-1)}{2}}f_m(q^{-m})=
$$
$$
(-1)^{m}q^{\frac{m(m-1)}{2}}(1-q^{-m})\cdots (1-q^{-1})=q^{-m}(1-q^{m})\cdots (1-q).
$$
\end{proof}

\begin{lemma}
The transition matrix between the monomial basis $x^a$ and the basis $f_b(x)$ has the following form:
\begin{equation}
\label{eq: monomial in interpolation basis}
x^a=\sum_{b\le a}k_{a,b}f_b(x),\quad k_{a,b}=(-1)^{b}q^{-ab+\frac{b(b+1)}{2}}\binom{a}{b}_q.
\end{equation}
\end{lemma}

\begin{proof}
To find the coefficients we compute the pairing $(f_b(x), x^a)$,
then using orthogonality we obtain
$$k_{a,b}=\frac{(f_b(x),x^a)}{(f_b(x),f_b(x))}=\frac{f_b(q^{-a})}{(f_b(x),f_b(x))}.$$
For $a\ge b$ from Lemma \ref{lem: orthogonal 1d} we get
$$
(f_{b}(x),f_{b}(x))=q^{-b}(q;q)_{b},
$$
while
\begin{align*}
f_{b}(q^{-a})&=(1-q^{-a})\cdots (1-q^{-a+b-1})=(-1)^{b}q^{-ab+\frac{b(b-1)}{2}}(1-q^{a})\cdots (1-q^{a-b+1})\\
&=(-1)^{b}q^{-ab+\frac{b(b-1)}{2}}\frac{(q;q)_{a}}{(q;q)_{a-b}}.
\end{align*}
and the equation follows.
\end{proof}

Our next goal is to expand arbitrary polynomial $f(x)$ in the basis $f_m(x)$. This can be done in two different ways. First, we can expand $f(x)$ in the monomial basis and apply \eqref{eq: monomial in interpolation basis}. Alternatively, we can apply {\em Newton interpolation method}: if $f(x)=\sum a_mf_m(x)$ then 
$$
f(q^{-j})=\sum_{m\ge j} a_m f_m(q^{-j}),
$$ 
which is a triangular system of equations for the unknown coefficients $a_m$. Thus knowing $f(q^{-j})$ 
one can at least theoretically reconstruct the coefficients $a_m$. This can be made explicit by the following:

\begin{lemma}
We have 
\begin{equation}
\label{eq: explicit interpolation 1d}
f(x)=\sum_{m=0}^{\infty} a_mf_m(x),\ a_m=\frac{1}{(f_m,f_m)}\sum_{j=0}^{m}(-1)^{j}q^{\frac{j(j-1)}{2}}\binom{m}{j}_{q}f(q^{-j}).
\end{equation}
\end{lemma}

\begin{proof}
By $q$-binomial theorem we have
\begin{equation}
\label{eq: interpolation in monomial 1d}
f_{m}(x)= \sum_{j=0}^{m}(-1)^{j}q^{\frac{j(j-1)}{2}}\binom{m}{j}_{q}x^{j}.
\end{equation}
Now
$$
a_m=\frac{(f,f_m)}{(f_m,f_m)}=\frac{1}{(f_m,f_m)}\sum_{j=0}^{m}(-1)^{j}q^{\frac{j(j-1)}{2}}\binom{m}{j}_{q}(f,x^{j}).
$$
Finally, $(f,x^{j})=f(q^{-j})$.
\end{proof}

\begin{remark}{\rm
Equation \eqref{eq: interpolation in monomial 1d} can be interpreted as an explicit inverse of the matrix in \eqref{eq: monomial in interpolation basis}.}
\end{remark}

One can consider completion $\widehat{\Z_q[x]}$ of the space of polynomials with respect to the basis $f_m(x)$. 
In this completion, infinite sums $\sum_{m=0}^{\infty}a_mf_m(x)$ are allowed. Newton interpolation method and \eqref{eq: explicit interpolation 1d} identify this completion with the space of distributions on the interpolation nodes $1,q^{-1},\ldots$.

We will need the following lemma.

\begin{lemma}
\label{lem: binomial}
We have
\begin{multline*}
(x-q^{s})(x-q^{s+1})\cdots (x-q^{s+m-1})=\\
\sum_{j=0}^{m} (-1)^{j} q^{-jm+\binom{j+1}{2}}\binom{m}{j}_{q}(1-q^{s+j})\cdots (1-q^{s+ m-1})f_j(x).
\end{multline*}
\end{lemma}

\begin{proof}
We prove it by induction in $m$. For $m=1$ we get
$$
x-q^s=-(1-x)+(1-q^s)=-f_1+(1-q^s)f_0.
$$
For the step of induction we observe
\begin{align}
\label{eq:help}
(x-q^{s+m})f_j(x)&=-q^{-j}(1-q^j x)f_{j}(x)+(q^{-j}-q^{s+m})f_{j}(x)\\\nonumber &=
-q^{-j}f_{j+1}(x)+q^{-j}(1-q^{s+m+j})f_{j}(x).
\end{align}
Using \eqref{eq:help}, it is easy to identify the coefficient at $f_j(x)$ in
$$
(x-q^{s+m})\sum (-1)^{j} q^{-jm+\binom{j+1}{2}}\binom{m}{j}_{q}(1-q^{s+j})\cdots (1-q^{s+m-1})f_j(x)
$$
as
\begin{multline*}
-q^{-j+1}(-1)^{j-1} q^{-(j-1)m+\binom{j}{2}}\binom{m}{j-1}_{q}(1-q^{s+j-1})\cdots (1-q^{s+m-1})\\+
q^{-j}q^{-jm+\binom{j+1}{2}}\binom{m}{j}_{q}(1-q^{s+j})\cdots (1-q^{s+m-1})(1-q^{s+m+j})\\=
-q^{-j(m+1)+\binom{j+1}{2}}(1-q^{s+j})\cdots (1-q^{s+m-1})\\\times
\left[q^{m-j+1}\binom{m}{j-1}_{q}(1-q^{s+j-1})+\binom{m}{j}_{q}(1-q^{s+m+j})\right].
\end{multline*}
It remains to notice that
$$
q^{m-j+1}\binom{m}{j-1}_{q}(1-q^{s+j-1})+\binom{m}{j}_{q}(1-q^{s+m+j})
$$
$$
=\left[q^{m-j+1}\binom{m}{j-1}_{q}+\binom{m}{j}_{q}\right]-q^{s+m}\left[\binom{m}{j-1}_{q}+q^j\binom{m}{j}_{q}\right]
$$
$$=
\binom{m+1}{j}_{q}-q^{s+m}\binom{m+1}{j}_{q}=(1-q^{s+m})\binom{m+1}{j}_q.
$$
\end{proof}

\begin{remark}{\rm
If we set a formal variable $y=q^s$ in Lemma \ref{lem: binomial}, then we get the identity
$$
(x-y)(x-qy)\cdots (x-yq^{m-1})=\sum_{j=0}^{m} (-1)^{j} q^{-jm+\binom{j+1}{2}}\binom{m}{j}_{q}f_{m-j}(yq^j)f_j(x).
$$
This is a $q$-analogue of the binomial identity
$$
(x-y)^m=\sum_{j=0}^{m} (-1)^{j}  \binom{m}{j}(1-y)^{m-j}(1-x)^{j}.
$$}
\end{remark}

%\textcolor{blue}{Check normalization in the following lemma}

%\begin{lemma}
%\label{lem:kirby}
%Suppose that $f(q^{-j})=q^{-\frac{j(j-1)}{2}}$ for all $j$. Then
%$$
%f(x)=\sum a_m f_m(x),\quad a_m=\begin{cases}
%\frac{(-1)^{m}q^{m}}{\prod_{i=1}^{m/2}(1-q^{2i})}& m\ \text{even}\\
%0& m\ \text{odd}.\\
%\end{cases}
%$$
%\end{lemma}

%\begin{proof}
%We can plug in $f(q^{-j})=q^{-\frac{j(j-1)}{2}}$ into \eqref{eq: explicit interpolation 1d}:
%$$
%a_m=\frac{1}{(f_m,f_m)}\sum_{j=0}^{m}(-1)^{j}q^{\frac{j(j-1)}{2}}\binom{m}{j}_{q}\times q^{-\frac{j(j-1)}{2}}=
%$$
%$$
%\frac{1}{(f_m,f_m)}\sum_{j=0}^{m}(-1)^{j}\binom{m}{j}_{q}
%$$
%Since
%$$
%\sum_{j=0}^{m}(-1)^{j}\binom{m}{j}_{q}=\begin{cases}
%\prod_{i=1}^{m/2}(1-q^{2i-1})& m\ \text{even}\\
%0& m\ \text{odd},\\
%\end{cases}
%$$
%and 
%$(f_m,f_m)=(-1)^{m}q^{-m}\prod_{i=1}^{m}(1-q^{i})$, we get the desired equation.

%{\bf comment} $q^{-m^2}$ should be $(-1)^mq^{-m}$
%\end{proof}

\subsection{Multi-variable case: polynomials}
\label{sec: def polynomials}

Let us generalize the above results to the case of $N$ variables. 
The pairing has the form
$$
(x_1^{a_1}\cdots x_N^{a_N},x_1^{b_1}\cdots x_N^{b_N})=q^{-\sum a_ib_i}=(x_1^{a_1},x_1^{b_1})\cdots(x_N^{a_N},x_N^{b_N}).
$$
Note that for $\xx=(x_1, \dots, x_N)$
$$
(g(\xx),x_1^{b_1}\cdots x_N^{b_N})=g(q^{-b_1},\ldots,q^{-b_N}).
$$
Consider the products
$$
f_{k_1,\ldots,k_N}(\xx)=f_{k_1}(x_1)\cdots f_{k_N}(x_N).
$$
Since $f_k(x)$ give a basis in $\C(q)[x]$, the polynomials
$f_{k_1,\ldots,k_N}$ give a basis in $\C(q)[x_1,\ldots,x_N]$.
Clearly,
$$
(f_{k_1,\ldots,k_N},x_1^{b_1}\cdots x_N^{b_N})=0\ \textrm{unless}\ b_i\ge k_i\ \textrm{for all}\ i.
$$
\begin{lemma}
We have $(f_{k_1,\ldots,k_N},f_{m_1,\ldots,m_N})=0$ unless $k_i=m_i$ for all $i$.
\end{lemma}

\begin{proof}
Suppose that $k_i>m_i$ for some $i$. Since $f_{m_1,\ldots,m_N}$ contains only monomials of the form $x_1^{b_1}\cdots x_N^{b_N}$ with $b_i\le m_i$, we have $(f_{k_1,\ldots,k_N},x_1^{b_1}\cdots x_N^{b_N})=0$ for all such monomials and hence $(f_{k_1,\ldots,k_N},f_{m_1,\ldots,m_N})=0$.
\end{proof}

Next, we would like to describe the basis in symmetric polynomials. It will be labeled by partitions
$\lambda=(\lambda_1\ge \lambda_2\ge\ldots\ge \lambda_N)$ with at most $N$ parts. We define
\begin{equation}
\label{eq: def F}
F_{\lambda}(\xx)=\frac{\det(f_{\lambda_i+N-i}(x_j))}{\prod_{i<j}(x_i-x_j)}.
\end{equation}
Clearly, the numerator in \eqref{eq: def F} is antisymmetric in $x_i$, so it is divisible by 
$\prod_{i<j}(x_i-x_j)$ and the ratio is a symmetric function. It is easy to see that $F_{\lambda}(\xx)$ is a non-homogeneous polynomial of degree $|\lambda|$, and the top degree component equals $(-1)^{|\lambda|+\binom{N}{2}}q^{D_{N}(\lambda)}s_{\lambda}$ where $s_{\lambda}$ is the Schur function
and $D_{N}(\lambda)$ is defined by \eqref{eq:def D}. The function $F_{\lambda}(\xx)$ is known as a special case of a factorial Schur function \cite{Macd2,Macd,Molev}, it is also a specialization of nonsymmetric Macdonald polynomials described below.

\begin{lemma}
Suppose that $b_1>\ldots > b_N$. Then 
$F_{\lambda}(q^{-b_1},\ldots,q^{-b_N})=0$ unless $b_i\ge \lambda_i+N-i$ for all $i$.
\end{lemma}

\begin{proof}
Suppose that $b_j<\lambda_j+N-j$ for some $j$, then for all $i\le j$ and $\ell>j$ one has
$\lambda_i+N-i\ge \lambda_j+N-j>b_j\ge b_{\ell}$, so $f_{\lambda_i+N-i}(q^{-b_{\ell}})=0$.
This implies $\det[f_{\lambda_i+N-i}(q^{-b_{\ell}})]_{i,\ell=1}^{N}=0$.
On the other hand, since $b_i\neq b_j$ the denominator $\prod_{i<j}(q^{-b_i}-q^{-b_j})$ does not vanish. 
\end{proof}

\begin{corollary}
\label{cor: vanishing}
If $\mu$ is another partition then we can define $b_i=\mu_i+N-i$, and conclude that 
$F_{\lambda}(q^{-\mu_i-N+i})=0$ unless $\mu_i\ge \lambda_i$ for all $i$, that is, partition $\mu$ contains $\lambda$. 
\end{corollary}

\begin{example}{\rm
\label{ex: 1 box}
Suppose that $\lambda=(1)$, then $F_{(1)}$ is a symmetric function of degree~1 with leading term $(-1)^{1+\binom{N}{2}}q^{D_N(1)}s_{(1)}=q^{D_N(1)}\sum x_i$.
We have $D_N(1)=N-1+\binom{N}{3}$, so $F_{(1)}(x_1,\ldots,x_N)=(-1)^{1+\binom{N}{2}}q^{N-1+\binom{N}{3}}\sum x_i+c.$ To find the constant $c$, 
we observe that by Corollary~\ref{cor: vanishing}  we get $F_{(1)}(q^{-N+1},q^{-N+2},\ldots,1)=0$, so 
$$c=(-1)^{\binom{N}{2}}q^{N-1+\binom{N}{3}}(q^{-N+1}+q^{-N+2}+\ldots+1)=(-1)^{\binom{N}{2}}q^{\binom{N}{3}}[N]_q.$$ }
\end{example}

\begin{lemma}
\label{lem:diagonal}
We have 
$$
F_{\lambda}(q^{-\lambda_i-N+i})=(-1)^{\binom{N}{2}}q^{n(\lambda)+\binom{N}{3}}\prod_{\sq\in \lambda}(1-q^{-h(\sq)}),
$$
where $h(\sq)$ is the hook length of a box $\sq$ in the Young diagram corresponding to $\lambda$.
\end{lemma}

\begin{proof}
Since the sequence $\lambda_i+N-i$ is strictly decreasing, we have 
$f_{\lambda_j+N-j}(q^{-\lambda_i-N+i})=0$ for $j>i$ and  
$$
f_{\lambda_i+N-i}(q^{-\lambda_i-N+i})=\{\lambda_i+N-i\}_{q^{-1}}!
$$
and 
$$
\det(f_{\lambda_j+N-j}(q^{-\lambda_i-N+i}))=\prod_i \{\lambda_i+N-i\}_{q^{-1}}!\ .
$$
On the other hand,
$$
\prod_{i<j}(q^{-\lambda_i-N+i}-q^{-\lambda_j-N+j})=(-1)^{\binom{N}{2}}q^{-\sum (\lambda_j+N-j)(j-1)}\prod_{i<j}(1-q^{-\lambda_i+i+\lambda_j-j}).
$$
and the statement follows now from 
formula \eqref{prem:id} and 
%$ h(\sq)=\lambda_i-i+\lambda_j-j+1$ and 
the identity
$$
\sum (\lambda_j+N-j)(j-1)=n(\lambda)+\binom{N}{3}.
$$
\end{proof}

%\begin{example}
%For $N=1$ we get $\lambda=(m)$ and $n(\lambda)=0$. Therefore
%$$
%f_m(q^{-m})=q^{-m(m-1)+0-m}(1-q)\cdots (1-q^m).
%$$
%\end{example}

\begin{example}{\rm  
For arbitrary $N$ and $\lambda=(1)$ we computed in Example \ref{ex: 1 box} that
$$F_{(1)}=(-1)^{1+\binom{N}{2}}q^{\binom{N}{3}}(q^{N-1}(x_1+\ldots+x_N)-[N]_{q}).$$ Hence,
$$
F_{(1)}(q^{-N},q^{-N+2},\ldots,1)=q^{\binom{N}{3}}(q^{N-1}(q^{-N}+q^{-N+2}+\ldots+1)-[N]_{q})=
$$
$$
(-1)^{1+\binom{N}{2}}q^{\binom{N}{3}}(q^{-1}-1)=(-1)^{\binom{N}{2}}q^{\binom{N}{3}}(1-q^{-1}).
$$}
\end{example}

We summarize the above results in the following proposition:

\begin{proposition}\cite{Ok}
\label{prop:interpolation}
There exists a unique collection of nonhomogeneous symmetric polynomials $F_{\lambda}(x_1,\ldots,x_N)$ with the following properties: 
\begin{itemize}
\item $F_{\lambda}(x_1,\ldots,x_N)$ has degree $|\lambda|$.
\item $F_{\lambda}(q^{-\mu_i-N+i})=0$ for all partitions $\mu$ not containing $\lambda$.
\item $F_{\lambda}(q^{-\lambda_i-N+i})=(-1)^{\binom{N}{2}}q^{n(\lambda)+\binom{N}{3}}\prod_{\sq\in \lambda}(1-q^{-h(\sq)}).$
\end{itemize} 
\end{proposition}
%\textcolor{red}{pluses in the power of q?}

We will denote the value $F_{\lambda}(q^{-\lambda_i-N+i})=(-1)^{\binom{N}{2}}q^{n(\lambda)+\binom{N}{3}}\prod_{\sq\in \lambda}(1-q^{-h(\sq)})$ by $c_{\lambda,\lambda}$.

%\textcolor{red}{Move next lemma to another section}

\begin{lemma}
\label{lem: divisibility}
Suppose that $q$ is a root of unity. Then $c_{\lambda,\lambda}$ vanishes for all but finitely many partitions $\lambda$.
\end{lemma}

\begin{proof}
Observe that $\prod_{\sq\in \lambda}(1-q^{-h(\sq)})$ is divisible by $\prod_{i} [\lambda_i-\lambda_{i+1}]_{q}!$ and 
$$\sum_{i=1}^{N}i(\lambda_i-\lambda_{i+1})=|\lambda|.$$ This means that for some $i$ we must have
$$
i(\lambda_i-\lambda_{i+1})\ge \frac{|\lambda|}{N},\ \lambda_i-\lambda_{i+1}\ge \frac{|\lambda|}{iN}\ge \frac{|\lambda|}{N^2},
$$
and $c_{\lambda,\lambda}$ is divisible by $(1-q)\cdots (1-q^{\lfloor \frac{|\lambda|}{N^2}\rfloor})$.
If $q^s=1$ then it vanishes for $|\lambda|\ge sN^2$.
\end{proof}

\begin{remark}
{\rm 
A partition is called an {\em $s$-core} if none of its hook lengths is divisible by $s$. The $s$-core partitions play an important role in representation theory of symmetric groups in finite characteristic, and of Hecke algebras at roots of unity \cite{JK}. If $q^s=1$ then clearly $c_{\lambda,\lambda}(q)\neq 0$ if and only if $\lambda$ is an $s$-core.
Although there are infinitely many $s$-cores, Lemma \ref{lem: divisibility} shows that there are finitely many $s$-cores with at most $N$ rows. 

For example, for $s=2$ the $2$-cores are ``staircase partitions'' $\lambda=(k,k-1,\ldots,1)$, and the maximal $2$-core with at most $N$ rows has size $N+(N-1)+\ldots+1=\binom{N+1}{2}$.
}
\end{remark}

\subsection{Multi-variable case: interpolation}

One can use the polynomials $F_{\lambda}$ to solve the following interpolation problem.
\begin{problem}
Find a symmetric function $f=\sum a_{\lambda}F_{\lambda}$ given its values $f(q^{-\mu_i-N+i})$ for all $\mu$. 
\end{problem}
We have
$$
f(q^{-\mu_i-N+i})=\sum a_{\lambda}F_{\lambda}(q^{-\mu_i-N+i})
$$
This is a linear system on $a_{\lambda}$ with the triangular matrix 
%\textcolor{red}{Notation $C$ conflicts with the Casimir, perhaps use different font}
\begin{equation}
\label{eq: def c}
{\sf C}=\left[c_{\lambda,\mu}\right]_{\lambda,\mu},\ c_{\lambda,\mu}(q):=F_{\lambda}(q^{-\mu_i-N+i})
\end{equation}
It is clear from Proposition \ref{prop:interpolation} that to find $a_{\lambda}$ for a given $\lambda$ it is sufficient to know all coefficients $c_{\mu,\nu}$ for $\mu\subset \nu\subset \lambda$. 

In \cite{Ok} Okounkov  computed the inverse matrix $D={\sf C}^{-1}$ which allows one to 
explicitly compute the coefficients $a_{\lambda}$. 

%\textcolor{blue}{Check normalization}

\begin{theorem}
\label{thm: dlm}
\cite{Ok}
Define $c^*_{\lambda,\mu}(q)=c_{\lambda,\mu}(q^{-1})$ and $\cont(\lambda)=n(\lambda)-n(\lambda')$. Then
$$
D=\left[d_{\lambda,\mu}\right]_{\lambda,\mu},\ d_{\lambda,\mu}=(-1)^{|\mu|-|\lambda|}q^{\cont(\lambda)-\cont(\mu)}\frac{c^*_{\lambda,\mu}}{c_{\mu,\mu}c^*_{\lambda,\lambda}}
$$
and
$$
a_{\mu}=\sum_{\lambda\subset \mu}d_{\lambda,\mu}f(q^{-\lambda_i-N+i})=\frac{1}{c_{\mu,\mu}}\sum_{\lambda\subset \mu}(-1)^{|\mu|-|\lambda|}q^{\cont(\lambda)-\cont(\mu)}\frac{c^*_{\lambda,\mu}}{c^*_{\lambda,\lambda}}f(q^{-\lambda_i-N+i}).
$$
\end{theorem} 

\begin{example}{\rm
If $\lambda=\mu$ then clearly $d_{\lambda,\mu}=\frac{1}{c_{\lambda,\lambda}}$.}
\end{example}

\begin{example}
{\rm 
We have $F_{(\emptyset)}=(-1)^{\binom{N}{2}}q^{\binom{N}{3}}$, so 
$$
c_{(\emptyset),(\emptyset)}=c_{(\emptyset),(1)}=(-1)^{\binom{N}{2}}q^{\binom{N}{3}},\quad c^*_{(\emptyset),(\emptyset)}=c^*_{(\emptyset),(1)}=(-1)^{\binom{N}{2}}q^{-\binom{N}{3}}.
$$
Since 
$$
c_{(1),(1)}=(-1)^{\binom{N}{2}}q^{\binom{N}{3}}(1-q^{-1})=(-1)^{\binom{N}{2}+1}q^{\binom{N}{3}-1}(1-q),
$$
 we get 
$$d_{(\emptyset),(1)}=\frac{(-1)^{\binom{N}{2}}q^{-\binom{N}{3}+1}}{(1-q)}.$$
So the first two terms of interpolation series have the following form:
\begin{multline*}
f(x_1,\ldots,x_N)=(-1)^{\binom{N}{2}}q^{-\binom{N}{3}}f(q^{1-N},q^{2-N},\ldots,1)F_{(\emptyset)}(\xx)+\\
\frac{(-1)^{\binom{N}{2}+1}q^{-\binom{N}{3}+1}}{1-q}\left[-f(q^{1-N},q^{2-N},\ldots,1)+f(q^{-N},q^{2-N},\ldots,1)\right]F_{(1)}(\xx)+\ldots
\end{multline*}}
\end{example}

\begin{example}{\rm
\label{ex:one row}
For $N=1$ and $a\ge b$ we have 
$$
c_{(b),(a)}=f_{b}(q^{-a})=(1-q^{-a})\cdots (1-q^{-a+b-1})
$$
%\textcolor{red}{contradicts with the fomula given in the proof of Lemma 5.2}
hence 
%$$
%c^{*}_{(b),(a)}=q^{ab}(1-q^{-a})\cdots (1-q^{-a+b-1})=(-1)^{b}q^{\frac{b(b-1)}{2}}(1-q^{a})\cdots (1-q^{a-b+1}).
%$$
$$
c^{*}_{(b),(a)}=(1-q^{a})\cdots (1-q^{a-b+1})
$$
Now 
$$
\frac{c^{*}_{(b),(a)}}{c^{*}_{(b),(b)}}=\frac{(1-q^{a})\cdots (1-q^{a-b+1})}{(1-q^{b})\cdots (1-q)}=\binom{a}{b}_{q},
$$
and
$$
d_{(b),(a)}=(-1)^{a-b}q^{\frac{b(b-1)}{2}-\frac{a(a-1)}{2}}\frac{c^*_{(b),(a)}}{c_{(a),(a)}c^*_{(b),(b)}}=\frac{(-1)^{a-b}}{c_{(a),(a)}}q^{\frac{b(b-1)}{2}-\frac{a(a-1)}{2}}\binom{a}{b}_{q},
$$
which matches \eqref{eq: explicit interpolation 1d}. }
\end{example}

\begin{example}
\label{ex: 1and32}
Let $N=2$, $\lambda=(1)$ and $\mu=(3,2)$. We have $F_{\lambda}=q(x_1+x_2)-(1+q)$, so
$$
c_{\lambda,\mu}=F_{\lambda}(q^{-4},q^{-2})=(-q - 1 + q^{-1} + q^{-3}),\ c^*_{\lambda,\mu}=q^3 + q - 1 - q^{-1},\\
$$
and using Lemma \ref{lem:diagonal}
$$
c_{\lambda,\lambda}=-(1-q^{-1}),\ c^*_{\lambda,\lambda}=-(1-q),
$$
$$
c_{\mu,\mu}=-q^2(1-q^{-1})^2(1-q^{-2})(1-q^{-3})(1-q^{-4})=q^{-9}(1-q)^2(1-q^2)(1-q^3)(1-q^4). 
$$
Now
$$
d_{\lambda,\mu}=q^{-2}\frac{c^*_{\lambda,\mu}}{c_{\mu,\mu}c^*_{\lambda,\lambda}}=-q^{6}\frac{q^4 + q^2 - q - 1}{(1-q)^3(1-q^2)(1-q^3)(1-q^4)}.
$$
\end{example}

\subsection{Hopf pairing}

We have a symmetric bilinear form $(\cdot,\cdot)$ on 
$\Z[x_1,\ldots,x_N]^{S_N}$ defined by its values on Schur polynomials
$$
(s_{\lambda},s_{\mu})=s_{\lambda}(q^{-\mu_1-N+1},\ldots,q^{-\mu_N})s_{\mu}(q^{-N+1},\ldots,1).
$$ 
It is closely related to the Hopf pairing $\langle \cdot,\cdot \rangle$ for $\CR=\Rep(\U)$ defined in Section \ref{sec: Hopf}.
Note that 
$$
(f,s_{\mu})=f(q^{-\mu_1-N+1},\ldots,q^{-\mu_N})s_{\mu}(q^{-N+1},\ldots,1).
$$
for any symmetric function $f$.

\begin{proposition}
We have 
\begin{equation}\label{orthogonality}
(F_{\lambda},F_{\nu})=\delta_{\lambda,\nu}q^{-|\lambda|+2\binom{N}{3}}\prod_{\sq \in \lambda}(1-q^{N+c(\sq)}),
\end{equation}
 so the Hopf pairing is diagonal in the basis $\{F_{\lambda}\}_\lambda$.
\end{proposition}

\begin{proof}
We have 
$$
(F_{\lambda},s_{\mu})=F_{\lambda}(q^{-\mu_1-N+1},\ldots,q^{-\mu_N})s_{\mu}(q^{-N+1},\ldots,1)=0
$$
unless $\lambda\subset \mu$. On the other hand, $F_{\nu}$ can be expanded in $s_{\mu}$ for $\mu\preceq \nu$, so
$(F_{\lambda},F_{\nu})$ vanishes unless there exists $\mu\preceq \nu$ such that $\lambda\subset \mu$, in particular, $\lambda\preceq \nu$. 

Since the Hopf pairing is symmetric, $(F_{\lambda},F_{\nu})$ vanishes unless $\lambda\preceq \nu$ and $\nu\preceq \lambda$, so $\lambda=\nu$. Finally,
$$
(F_{\lambda},F_{\lambda})=(-1)^{|\lambda|+\binom{N}{2}}q^{D_{N}(\lambda)}(F_{\lambda},s_{\lambda})=(-1)^{|\lambda|+\binom{N}{2}}q^{D_N(\lambda)}F_{\lambda}(q^{-\lambda_1-N+1},\ldots,q^{-\lambda_N})s_{\lambda}(q^{-N+1},\ldots,1).
$$
Now 
$$
F_{\lambda}(q^{-\lambda_1-N+1},\ldots,q^{-\lambda_N})=(-1)^{\binom{N}{2}}q^{n(
\lambda)+\binom{N}{3}}\prod_{\sq\in \lambda}(1-q^{-h(\sq)})
$$
while
$$
s_{\lambda}(q^{-N+1},\ldots,1)=q^{-n(\lambda)}\prod_{\sq\in \lambda}\frac{(1-q^{-N-c(\sq)})}{(1-q^{-h(\sq)})}.
$$
hence
$$
F_{\lambda}(q^{-\lambda_1-N+1},\ldots,q^{-\lambda_N})s_{\lambda}(q^{-N+1},\ldots,1)=(-1)^{\binom{N}{2}}q^{\binom{N}{3}}\prod_{\sq\in \lambda}(1-q^{-N-c(\sq)})=
$$
$$
(-1)^{|\lambda|+\binom{N}{2}}q^{-N|\lambda|-c(\lambda)+\binom{N}{3}}(1-q^{N+c(\sq)}).
$$
On the other hand, $D_N(\lambda)=c(\lambda)+(N-1)|\lambda|+\binom{N}{3}$.
\end{proof}

This provides us with a different perspective for the interpolation problem. Suppose that we have a Schur expansion for $F_{\lambda}$:
$$
F_{\lambda}=\sum_{\mu\preceq \lambda}b_{\lambda,\mu}s_{\mu}.
$$
Then for an arbitrary symmetric function $f(x_1,\ldots,x_N)$ we can write 
$$
f=\sum_{\lambda}\frac{(f,F_{\lambda})}{(F_{\lambda},F_{\lambda})}F_{\lambda}=\sum_{\lambda}\sum_{\mu\preceq \lambda}b_{\lambda,\mu}\frac{(f,s_{\mu})}{(F_{\lambda},F_{\lambda})}F_{\lambda}=\sum_{\lambda}\sum_{\mu\preceq \lambda}\frac{b_{\lambda,\mu}s_{\mu}(q^{-N-i})}{(F_{\lambda},F_{\lambda})}f(q^{-\mu_i-N+i})F_{\lambda},
$$
and the interpolation coefficient is equal to 
\begin{equation}
\label{eq: dlm from hopf}
d_{\lambda,\mu}=\frac{b_{\lambda,\mu}s_{\mu}(q^{-N+i})}{(F_{\lambda},F_{\lambda})}.
\end{equation}

\begin{example}\label{d's}{\rm
For $N=2$ and $\lambda=(3,2)$ we have
\begin{multline*}
F_{(3,2)}=q^2(1-x_1)(1-qx_1)(1-x_2)(1-qx_2)(q^3(x_1+x_2)-(1+q))=\\
q^7s_{3,2}-q^6(1+q)s_{3,1}-q^4(1+q+q^2+q^3)s_{2,2}+q^6s_{3,0}+q^3(1+q+q^2+q^3)(1+q)s_{2,1}-\\ q^3(1+q+q^2+q^3)s_{2,0}-q^2(1+q+q^2+q^3)(1+q)s_{1,1}+(q^5 + q^4 + 2q^3 + q^2)s_{1,0}-(q^3+q^2).
\end{multline*}
Also
$$
(F_{3,2},F_{3,2})=-q^{-5}(1-q^4)(1-q^3)(1-q^2)^2(1-q)
$$
Therefore the interpolation coefficient for $\lambda=(3,2)$ and $\mu=(1,0)$ equals
$$
d_{(3,2),(1,0)}=(q^5 + q^4 + 2q^3 + q^2)\frac{s_{1,0}(q^{-1},1)}{(F_{3,2},F_{3,2})}=
$$
$$
-\frac{(q^5 + q^4 + 2q^3 + q^2)(1+q^{-1})}{q^{-5}(1-q^4)(1-q^3)(1-q^2)^2(1-q)}=-\frac{q^6(q^4 + q^2 - q - 1)}{(1-q^4)(1-q^3)(1-q^2)(1-q)^3}.
$$
This agrees with Example \ref{ex: 1and32}.
}
\end{example}

\subsection{Divisibility}
\label{sec: divisibility}

Given a polynomial $f(x)$, define 
$$
\partial_{xy}(f):=\frac{f(x)-f(y)}{x-y}.
$$
Observe that 
$$
\partial_{xy}(fg)=\frac{f(x)-f(y)}{x-y}g(x)+f(y)\frac{g(x)-g(y)}{x-y}=\partial_{xy} f\cdot g(x)+f(y)\cdot \partial_{xy}(g).
$$
More generally, we have 
\begin{align}
\label{eq: partial}
\partial_{xy}(f_1\cdots f_k)&=\partial_{xy}(f_1)f_2(x)\cdots f_k(x)+f_1(y)\partial_{xy}(f_2)f_3(x)\cdots f_k(x)+\ldots\\ \nonumber & + f_1(y)f_2(y)\cdots \partial_{xy}(f_k).
\end{align}

\begin{example} {\rm
\label{ex: partial f}
For $f_n(x)=(1-x)\cdots (1-q^{n-1}x)$, note that $\partial_{xy}(1-q^{i}x)=-q^i$, so we get 
\begin{align*}
F_{n,0}(x,y)&=\partial_{x,y}f_{n+1}(x)=\sum_{i=0}^{n} (1-y)\cdots (1-q^{i-1}y)[\partial_{x,y}(1-q^{i}x)](1-q^{i+1}x)\cdots (1-q^{n}x)\\
&= \sum_{i=0}^{n}f_{i}(y)\cdot (-q^{i})f_{n-i}(q^{i+1}x).
\end{align*}
}
\end{example}

\begin{example}{\rm
For example, 
$$
F_{1,0}(x,y)=q(x+y)-(1+q)=q(y-1)+(qx-1)=-[qf_1(y)+f_1(qx)].
$$
Similarly, 
\begin{align*}
F_{2,0}(x,y)&=-q^3(x^2+xy+y^2)+(q+q^2+q^3)(x+y)-(1+q+q^2)\\ &=
-[(1-qx)(1-q^2x)+q(1-q^2x)(1-y)+q^2(1-y)(1-qy)]\\
&=
-[f_2(qx)+qf_1(x)f_2(y)+q^2f_2(y)].
\end{align*}}
\end{example}

\begin{corollary}
For all  integers $a$ and $b$ the value $F_{n,0}(q^a,q^b)$ is divisible by $\left(\left\lfloor\frac{n}{2}\right\rfloor\right)_{q}!$
\end{corollary}

\begin{proof}
Let $k=\left\lfloor\frac{n}{2}\right\rfloor$. In the above equation either $i\ge k$ or $n-i\ge k$, so each term in the sum 
is either divisible by $f_k(q^{i+1+a})$ or by $f_k(q^b)$, so by $q$-binomial theorem it is divisible by $(k)_{q}!$
\end{proof}

More generally, let $\partial_{i}=\partial_{x_i,x_{i+1}}$ then it is well known that $\partial_i$ satisfy braid relations, so one can define $\partial_{w}$ for any permutation $w$. Furthermore, 
$$
F_{\lambda}(x_1,\ldots,x_N)=\partial_{w_0}[f_{\lambda_1+N-1}(x_1)\cdots f_{\lambda_N}(x_N)],
$$
where $w_0=(N\ N-1\ \ldots 1)$ is the longest element in $S_N$. 

\begin{lemma}\label{lem:div}
For all $\lambda$ one can write $F_{\lambda}(x_1,\ldots,x_N)$ as the sum where each term has the form
\begin{equation}
\label{eq: product}
f_{j_1}(q^{s_1}x_{m_1})\cdots f_{j_{d}}(q^{s_d}x_{m_d}),\ \text{where}\ j_1+\ldots+j_{d}=|\lambda|\ \text{and}\ d=\binom{N+1}{2}.
\end{equation}
Here the indices $m_i$ might repeat arbitrarily.
\end{lemma}

\begin{proof}
From \eqref{eq: partial} and Example \ref{ex: partial f} it is clear that $\partial_{i}$ applied to a product \eqref{eq: product} with $\ell$ factors produces a sum of similar products with $\ell+1$ factors. We start from a product of $N$ factors, and $\partial_{w}$ is a composition of $\binom{N}{2}$ operators $\partial_i$, so the terms in the resulting sum have $N+\binom{N}{2}=\binom{N+1}{2}$ factors. 
Also, each $\partial_i$ decreases the degree by 1, so 
$$
j_1+\ldots+j_{d}=\sum(\lambda_i+N-i)-\binom{N}{2}=|\lambda|.
$$
\end{proof}

\begin{remark}{\rm
A more careful analysis of this proof leads to a combinatorial formula for $F_{\lambda}$ where the terms are labeled by semistandard tableaux, but we do not need it here. This is a $q$-analogue of the expansion of a Schur function in the monomial basis. }
\end{remark}

\begin{lemma}
\label{lem: divisible}
For any sequence of integers $a_1,\ldots,a_N$ the value $F_{\lambda}(q^{a_1},\ldots,q^{a_N})$ is divisible by $(k)_{q}!$ where $k=\left\lfloor \frac{|\lambda|}{\binom{N+1}{2}}\right\rfloor$. 
\end{lemma}

\begin{proof}
In each term \eqref{eq: product} there are $d=\binom{N+1}{2}$ indices $j_1,\cdots,j_d$ which add up to $|\lambda|$, so at least one of these indices is greater than $|\lambda|/d$. It remains to  notice that $f_j(q^a)$ is divisible by $(q)_{j}!$ for all integers $a$. 
\end{proof}

The following lemma gives a rough description of the expansion
\begin{equation}
\label{eq: nonsymmetric expansion}
F_{\lambda}(x_1,\ldots,x_N)=\sum_{m_1,\ldots,m_N}b_{m_1,\ldots,m_k}f_{m_1}(x_1)\cdots f_{m_N}(x_N).
\end{equation}
of the symmetric interpolation polynomial $F_{\lambda}$ in terms of nonsymmetric ones. 

\begin{lemma}
Given $k$, for sufficiently large $|\lambda|$ for all terms of the expansion \eqref{eq: nonsymmetric expansion}  either the coefficient $b_{m_1,\ldots,m_k}$ is divisible by $(k)_{q}!$ or there exists $m_i \geq k$  for some  $1\leq i\leq N$.
\end{lemma}

\begin{proof}
We follow the same logic as in Lemma \ref{lem: divisible}. For $|\lambda|>2k\binom{N+1}{2}$ every term \eqref{eq: product} is divisible by $f_{2k}(q^sx_i)$ for some $s$ and $i$.  By Lemma \ref{lem: binomial} this can be further decomposed into terms which are divisible by $(j)_{q}!f_{2k-j}(x_i)$, and either $j$ or $2k-j$ is greater than or equal to $k$. Overall, we presented 
$$
F_{\lambda}(x_1,\ldots,x_N)=A(k)_{q}!+\sum B_i f_{k}(x_i)
$$
for some polynomials $A$ and $B_i$. It remains to notice that the polynomial $B_i f_{k}(x_i)$ can be presented as the sum of 
$f_{m_1}(x_1)\cdots f_{m_N}(x_N)$ where $m_i\ge k$. 
\end{proof}

\section{Stability of interpolation and the case $N=2$}
\label{sec: stability}

\subsection{Stability of interpolation matrices}
In this section 
%we  compare interpolation polynomials for different values of $N$. 
study  the dependence of the interpolation polynomials 
on $N$.
%Since it turns out to be very weak, all our previous results can actually be extended to the
%symmetric functions in infinitely many variables.

As above, if partition $\lambda$ has less than $N$ parts we can complete it with zeroes. We denote by $F_{\lambda;N}(x_1,\ldots,x_N)$ the corresponding polynomial in $N$ variables. 

\begin{lemma}
\label{lem: stability}
Let $\lambda$ be a partition with at most $N$ parts. Then
$$
F_{\lambda;N}(x_1,\ldots,x_{N-1},1)=
\begin{cases}
(-1)^{N-1}q^{\binom{N-1}{2}}F_{\lambda;N-1}(qx_1,\ldots,qx_{N-1})& \text{if}\ \lambda_N=0\\
0 & \text{otherwise}.
\end{cases}
$$
\end{lemma}
\begin{proof}
Let $\mu$ be a partition with at most $N-1$ parts. Then by Proposition \ref{prop:interpolation}
$$
F_{\lambda;N}(q^{-\mu_1-N+1},\ldots,q^{-\mu_{N-1}-1},1)=0
$$
unless $\mu$ contains $\lambda$. If $\lambda_{N}>0$ then this never happens and 
$F_{\lambda;N}(x_1,\ldots,x_{N-1},1)=0$. If $\lambda_{N}=0$ we write 
$L(x_1,\ldots,x_{N-1})=F_{\lambda;N-1}(qx_1,\ldots,qx_{N-1})$.
We have 
$$
L(q^{-\mu_1-N+1},\ldots,q^{-\mu_{N-1}-1})=F_{\lambda;N-1}(q^{-\mu_1-(N-1)+1},\ldots,q^{-\mu_{N-1}})
$$
which vanishes unless $\mu$ contains $\lambda$, so by Proposition \ref{prop:interpolation} $F_{\lambda;N}(x_1,\ldots,x_{N-1},1)$
is proportional to $L(x_1,\ldots,x_{N-1})$. Finally, at $\mu=\lambda$ we can use Lemma \ref{lem:diagonal} to determine the coefficient.
\end{proof}

\begin{remark}{\rm
We can also prove the lemma using the explicit determinantal formula. Indeed, $f_{\lambda_i+N-i}(1)=0$ unless $f_{\lambda_i+N-i}=0$
which is equivalent to $i=N$ and $\lambda_N=0$. Therefore for $\lambda_N\neq 0$ the last row in the matrix $f_{\lambda_i+N-i}(x_j)$
vanishes (where $x_N=1$), and $F_{\lambda;N}(x_1,\ldots,x_{N-1},1)=0$. For $\lambda_N=0$ we have
$$
F_{\lambda;N}(x_1,\ldots,x_{N-1},1)=\frac{\det\left[f_{\lambda_i+N-i}(x_j)\right]_{i,j=1}^{N-1}}{\prod_{i<j\le N-1}(x_i-x_j)\prod_{i\le N-1}(x_i-1)}.
$$
Note that $f_{k+1}(x)=(1-x)f_{k}(qx)$, so 
$$
f_{\lambda_i+N-i}(x_j)=(1-x_j)f_{\lambda_i+(N-1)-i}(qx_j)
$$
Therefore
$$
F_{\lambda;N}(x_1,\ldots,x_{N-1},1)=\frac{\prod_{i=1}^{n} (1-x_i)\det\left[f_{\lambda_i+(N-1)-i}(qx_j)\right]_{i,j=1}^{N-1}}{\prod_{i<j\le N-1}(x_i-x_j)\prod_{i\le N-1}(x_i-1)}=
$$
$$
(-1)^{N-1}q^{\binom{N-1}{2}}F_{\lambda;N-1}(qx_1,\ldots,qx_{N-1}).
$$}
\end{remark}

\begin{corollary}
\label{cor: stability}
{\rm
Let $c_{\lambda,\mu}^{(N)}$ be the coefficient defined in previous section for symmetric functions in $N$ variables.
Then the expressions
$$
(-1)^{\binom{N}{2}}q^{-\binom{N}{3}}c^{(N)}_{\lambda,\mu},(-1)^{\binom{N}{2}}q^{\binom{N}{3}}c^{(N)*}_{\lambda,\mu}, (-1)^{\binom{N}{2}}q^{\binom{N}{3}}d^{(N)}_{\lambda,\mu}
$$
are independent of $N$ (provided that $\lambda$ and $\mu$ have at most $N$ parts).
}
\end{corollary}

\begin{example}{\rm
For one-row partitions $\lambda=(b)$ and $\mu=(a)$ the interpolation coefficients are given by
the formulas in Example \ref{ex:one row} up to a monomial factor.} 
\end{example}

The above results allow us to describe Schur expansion of interpolation polynomials:

\begin{proposition}
\label{prop: HOMFLY interpolation}
{\rm
We have
\begin{equation}
F_{\lambda}^{(N)}=(-1)^{\binom{N}{2}}q^{\binom{N}{3}}\sum_{\mu\subset \lambda} \overline{b_{\lambda,\mu}} A^{|\mu|} \prod_{\sq\in \lambda\setminus \mu}(1-Aq^{c(\sq)}) s_{\lambda}^{(N)}
\end{equation}
where $A=q^N$ and the coefficients 
$$
\overline{b_{\lambda,\mu}}=(-1)^{\binom{N}{2}}q^{\binom{N}{3}}d^{(N)}_{\lambda,\mu}q^{-|\lambda|-|\mu|-n(\mu)}\prod_{\sq\in \mu}(1-q^{h(\sq)})
$$ 
do not depend on $N$.
}
\end{proposition}

\begin{proof}
It follows from \eqref{eq: dlm from hopf} that
$$
F_{\lambda}=\sum b_{\lambda,\mu}s_{\mu},\ b_{\lambda,\mu}=\frac{d_{\lambda,\mu}(F_{\lambda},F_{\lambda})}{s_{\mu}(q^{-N+i})}.
$$
Since $d_{\lambda,\mu}$ vanishes unless $\mu\subset \lambda$, the same is true for $b_{\lambda,\mu}$. 
By Corollary \ref{cor: stability} the product $\overline{d_{\lambda,\mu}}=(-1)^{\binom{N}{2}}q^{\binom{N}{3}}d_{\lambda,\mu}$ does not depend on $N$, and  we can use the formulas
$$
(F_{\lambda},F_{\lambda})=q^{-|\lambda|+2\binom{N}{3}}\prod_{\sq\in \lambda}(1-Aq^{c(\sq)}),
$$
$$
s_{\mu}(q^{-N+i})=q^{n(\mu)-(N-1)|\mu|}\prod_{\sq\in \mu}\frac{(1-Aq^{c(\sq)})}{(1-q^{h(\sq)})}.
$$
to write
$$
b_{\lambda,\mu}=(-1)^{\binom{N}{2}}q^{-\binom{N}{3}}\overline{d_{\lambda,\mu}}\cdot q^{-|\lambda|+2\binom{N}{3}-n(\mu)+N|\mu|-|\mu|}\prod_{\sq\in \lambda}(1-Aq^{c(\sq)})\prod_{\sq\in \mu}\frac{(1-q^{h(\sq)})}{(1-Aq^{c(\sq)})}.
$$
The result follows.
\end{proof}

\begin{corollary}
\label{cor: one row}
{\rm
The one-row interpolation polynomials have the following Schur expansion:
$$
F_{(m)}^{(N)}=(-1)^{\binom{N}{2}}q^{\binom{N}{3}}\sum_{\mu\subset \lambda} (-1)^{j}q^{\frac{j(j-3)}{2}} A^{j} \frac{(1-Aq^{j})\cdots (1-Aq^{m-1})}{(1-q)\cdots (1-q^{m-j})} h_j^{(N)}=
$$
$$
(-1)^{\binom{N}{2}}q^{\binom{N}{3}}\sum_{\mu\subset \lambda} (-1)^{j}q^{\frac{j(j-3+2N)}{2}} \binom{N+m-1}{m-j}_{q}h_j^{(N)}.
$$
Here $h_j^{(N)}=s_{(j)}^{(N)}$ are complete symmetric functions in $N$ variables.
}
\end{corollary}

\begin{proof}
For $\lambda=(m)$ and $N=1$ we have
$$
f_m(x)=\sum_{j=0}^{m}(-1)^{j}q^{\frac{j(j-1)}{2}}\binom{m}{j}_{q} x^j.
$$
By writing $A=q$ and $\mu=(j)$ we get
$$
A^{|\mu|} \prod_{\sq\in \lambda\setminus \mu}(1-Aq^{c(\sq)})=q^{j}(1-q^{j+1})\cdots (1-q^{m}),
$$
so 
$$
\overline{b_{(m),(j)}}=(-1)^{j}\frac{q^{\frac{j(j-3)}{2}}}{(1-q)\cdots (1-q^{m-j})}.
$$
\end{proof}

\begin{remark}
{\rm
The HOMFLY-PT limit of interpolation polynomials in Proposition \ref{prop: HOMFLY interpolation} appears to be related to the results and conjectures in \cite{KNTZ}, it would be interesting to find a precise connection.
}
\end{remark}

\subsection{Adding a column}

It is well known that in symmetric functions in $N$ variables one has the identity
$$
s_{\lambda+1^{N}}=x_1\cdots x_N\cdot s_{\lambda}.
$$
Here $\lambda+1^{N}=(\lambda_1+1,\ldots,\lambda_N+1)$ and the corresponding Young diagram is obtained from the Young diagram for $\lambda$ by adding a vertical column. 

For interpolation polynomials we have two different generalizations of this identity: the first relates $F_{\lambda+1^N}$ to $F_{\lambda}$ and the second describe the action of the multiplication by $x_1\cdots x_N$.

\begin{proposition}
\label{prop: adding a column}
We have $F_{\lambda+1^N}(x_1,\ldots,x_N)=q^{\binom{N}{2}}\prod_{i=1}^{N}(1-x_i)F_{\lambda}(qx_1,\ldots,qx_N)$. More generally,
\begin{equation}
\label{eq: adding column}
F_{\lambda+k^N}(x_1,\ldots,x_N)=q^{k\binom{N}{2}}\prod_{i=1}^{N}f_k(x_i)F_{\lambda}(q^kx_1,\ldots,q^kx_N).
\end{equation}
\end{proposition}

\begin{proof}
We have $f_{m+1}(x)=(1-x)f_m(qx)$, therefore 
$$
\det\left[f_{\lambda_i+1+N-i}(x_j)\right]=\det\left[(1-x_j)f_{\lambda_i+N-i}(qx_j)\right]=\prod_{j=1}^{N}(1-x_j)\det\left[f_{\lambda_i+N-i}(x_j)\right].
$$
Since each factor $(x_i-x_j)$ in the denominator gets multiplied by $q$ after changing $x_i\to qx_i$, this implies the first equation.
Now \eqref{eq: adding column} can be obtained by applying it $k$ times.
\end{proof}

Let $e_i$ denote the $i$-th basic vector in $\Z^N$ with $1$ at $i$-th position and $0$ at other positions. Given $I\subset \{1,\ldots,n\}$, we define $e_I=\sum_{i\in I}e_i$.

\begin{proposition}
\label{prop: eN}
We have
$$
x_1\cdots x_N F_{\lambda}(x_1,\ldots,x_N)=q^{-|\lambda|-\binom{N}{2}} \sum_{I\subset \{1,\ldots,n\}}(-1)^{|I|} F_{\lambda+e_I}(x_1,\ldots,x_N).
$$
Here we use the convention that $F_{\lambda+e_I}=0$ unless the entries of $\lambda+e_I$ are non-increasing (that is, $\lambda+e_I$ is a partition).
\end{proposition}
 
\begin{proof}
We have $f_{m+1}(x)=f_m(x)(1-q^m x)$, so
$$
xf_m(x)=q^{-m}(f_{m}(x)-f_{m+1}(x)).
$$
Therefore 
$$
x_1\cdots x_N\det\left[f_{\lambda_i+N-i}(x_j)\right]=\det\left[x_jf_{\lambda_i+N-i}(x_j)\right]=
$$
$$
\det\left[q^{-\lambda_i-N+i}(f_{\lambda_i+N-i}(x_j)-f_{\lambda_i+1+N-i}(x_j))\right].
$$
\end{proof}

\begin{corollary}
\label{cor: inverse product}
Consider the completion of the space of symmetric functions with coefficients in $\Z[q,q^{-1}]$ with respect to the basis $F_{\lambda}$. Then
the operator of multiplication by $x_1\cdots x_N$ is invertible in this completion and its inverse is given by the equation
$$
(x_1\cdots x_N)^{-1}F_{\lambda}(x_1,\ldots,x_N)=q^{\binom{N}{2}}\sum_{v\in \Z_{\ge 0}^N}q^{|\lambda|+v}F_{\lambda+v}(x_1,\ldots,x_N).
$$
\end{corollary}

\begin{proof}
Define the operators $A_i$ by $A_i(F_{\lambda})=F_{\lambda+e_i}$, and $p_i(F_{\lambda})=q^{\lambda_i}F_{\lambda}$ for $i=1,\ldots,N$. Clearly, $[A_i,A_j]=[p_i,p_j]=[A_i,p_j]$ for $i\neq j$ and 
by Proposition \ref{prop: eN} we have
$$
x_1\cdots x_N = q^{-\binom{N}{2}} \prod_{i}  (1-A_i)p^{-1}_i,
$$
hence
$$
(x_1\cdots x_N)^{-1}=q^{\binom{N}{2}}\prod_{i}  p_i(1+A_i+A_i^2+\ldots).
$$
\end{proof}

\begin{example}
\label{ex: x inverse}
{\rm
For $N=1$ and $\lambda=(0)$ we get a curious identity
$$
x^{-1}=\sum_{m=0}^{\infty}f_m(x)q^m
$$
We can check this identity directly, by computing the values of both sides at $q^{-j}$ for all $j$. Denote
$$
u_j=\sum_{m=0}^{\infty}f_m(q^{-j})q^m=\sum_{m=0}^{j}f_m(q^{-j})q^m.
$$
Then $u_{j+1}=1+q(1-q^{-j-1})u_{j}$ and $u_0=1$, so it is easy to see that $u_j=q^j$.}
\end{example}

%\section{Interpolation for $\fsl_2$ and $\fgl_2$}

\subsection{Interpolation polynomials for $\fgl_2$}

In this subsection we describe the interpolation polynomials for $\fgl_2$ explicitly. 
By definition, we have polynomials $F_{\lambda}(x_1,x_2)$ where $\lambda_1\ge \lambda_2$:
$$
F_{\lambda_1,\lambda_2}(x_1,x_2)=
\frac{1}{x_1-x_2} \left|
\begin{matrix}
f_{\lambda_1+1}(x_1) & f_{\lambda_1+1}(x_2)\\
f_{\lambda_2}(x_1) & f_{\lambda_2}(x_2)
\end{matrix}
\right|
$$

Let us consider the case $\lambda_2=0$ first, and write $\lambda_1=k$.
Then 
$$
F_{k,0}(x_1,x_2)=
\frac{1}{x_1-x_2}\det \left|
\begin{matrix}
f_{k+1}(x_1) & f_{k+1}(x_2)\\
1 & 1
\end{matrix}
\right|=\frac{f_{k+1}(x_1)-f_{k+1}(x_2)}{x_1-x_2}.
$$
Let
$$
h_i(x_1,x_2)=\frac{x_1^{i+1}-x_2^{i+1}}{x_1-x_2}.
$$
Recall that $f_{k+1}(x)=\sum_{j=0}^{k+1}(-1)^{j}q^{\frac{j(j-1)}{2}}\binom{k+1}{j}_{q}x^{j}$, so
$$
F_{k,0}(x_1,x_2)=\sum_{j=1}^{k+1}(-1)^{j}q^{\frac{j(j-1)}{2}}\binom{k+1}{j}_{q}h_{j-1}(x_1,x_2),
$$
compare with Corollary \ref{cor: one row}.
We just replace each $x^j$ in the expression for $f_{k+1}(x)$ by $h_{j-1}(x_1,x_2)$.

\begin{example}{\rm
We have 
$$
f_1(x)=1-x,\ f_2(x)=(1-x)(1-qx)=1-(1+q)x+qx^2,
$$
$$
 f_3(x)=(1-x)(1-qx)(1-q^2x)=1-(1+q+q^2)x+(q+q^2+q^3)x^2-q^3x^3
$$
so
$$
F_{0,0}(x_1,x_2)=-1,\ F_{1,0}(x_1,x_2)=q(x_1+x_2)-(1+q),\ 
$$
$$
F_{2,0}=-q^3(x_1^2+x_1x_2+x_2^2)+(q+q^2+q^3)(x_1+x_2)-(1+q+q^2).
$$}
\end{example}

By Proposition \ref{prop: adding a column} we have
$$
F_{\lambda_1,\lambda_2}(x_1,x_2)=q^{\lambda_2}f_{\lambda_2}(x_1)f_{\lambda_2}(x_2)F_{\lambda_1-\lambda_2,0}(q^{\lambda_2}x_1,q^{\lambda_2}x_2).
$$
In particular, for $(\lambda_1,\lambda_2)=(k,k)$ we have 
$$
F_{k,k}(x_1,x_2)=q^kf_{k}(x_1)f_{k}(x_2).
$$
Also, by Lemma \ref{lem: stability} we get
\begin{equation}
\label{eq: stability N=2}
F_{\lambda_1,\lambda_2}(x_1,1)=\begin{cases}
-f_{\lambda_1}(qx_1) & \text{if}\ \lambda_2=0\\
0 & \text{otherwise}.
\end{cases}
\end{equation}

\subsection{Interpolation tables for $\fgl_2$}
\label{sec: interpolation tables}

For the reader's convenience, we have computed the polynomials $F_{\lambda}(x_1,x_2)$ and the corresponding interpolation matrices using {\tt Sage} \cite{sage}. First, we present $F_{\lambda}$ in Schur basis:
\begin{multline*}
F_{0}=-1,\quad F_{1}=qs_1-(q+1),\quad F_{2}=-q^3s_2+(q^3+q^2+q)s_1-(q^2+q+1),\\
F_{1,1}=-qs_{1,1}+qs_{1}-q=-q(1-x_1)(1-x_2)\\
F_{3}=q^6s_{3}-(q^6+q^5+q^4+q^3)s_{2}+(q^5+q^4+2q^3+q^2+q)s_{1}-(q^3+q^2+q+1)\\
F_{2,1}=q^3s_{2,1}-q^3s_{2}-(q^3+q^2+q)s_{1,1}+(q^3+q^2+q)s_{1}-(q^2+q)\\
F_{3,1}=-q^6s_{3,1}+q^6s_{3}+(q^6+q^5+q^4+q^3)s_{2,1}-(q^6+q^5+q^4+q^3)s_{2}-\\ (q^5+q^4+2q^3+q^2+q)s_{1,1}+(q^5+q^4+2q^3+q^2+q)s_{1}-(q^3+q^2+q)\\
F_{2,2}=-q^4s_{2,2}+(q^4+q^3)s_{2,1}-q^3s_{2}-(q^4+q^3+q^2)s_{1,1}+(q^3+q^2)s_1-q^2\\
F_{3,2}=q^7s_{3,2}-(q^7+q^6)s_{3,1}-(q^7+q^6+q^5+q^4)s_{2,2}+q^6s_{3}+(q^7+2q^6+2q^5+2q^4+q^3)s_{2,1}-\\(q^6+q^5+q^4+q^3)s_{2}-(q^6+2q^5+2q^4+2q^3+q^2)s_{1,1}+(q^5 + q^4 + 2q^3 + q^2)s_1-(q^3+q^2)\\
F_{3,3}=-q^9s_{3,3}+(q^9 + q^8 + q^7)s_{3,2}-(q^8+q^7+q^6)s_{3,1}-(q^9+q^8+2q^7+q^6+q^5)s_{2,2}+q^6s_{3}+\\
(q^8 + 2q^7 + 2q^6 + 2q^5 + q^4)s_{2,1}-(q^6+q^5+q^4)s_{2}-(q^7+q^6+2q^5+q^4+q^3)s_{1,1}+(q^5+q^4+q^3)s_{1}-q^3
\end{multline*}

Next, we list the values of the evaluations $c_{\lambda,\mu}=F_{\lambda}(q^{-\mu_1-1},q^{-\mu_2})$ for various $\lambda$ and $\mu$ in Tables \ref{tab 1}, \ref{tab 2}, \ref{tab 3} below. The resulting matrix ${\sf C}=(c_{\lambda,\mu})$ is upper-triangular, with diagonal entries prescribed by Lemma \ref{lem:diagonal}. Zero entries correspond to pairs $(\lambda,\mu)$ where $\mu$ does not contain $\lambda$. The entry corresponding to $(\lambda,\mu)=((1),(3,2))$ is marked in bold, it is divisible by $1-q$ but does not factor any further.
%{\color{red} Why it is not cyclotomic? it is divisible by $\Phi_1=1-q$ i.e. cyclotomic for the minimal partition}

Using either Theorem \ref{thm: dlm} or equation \eqref{eq: dlm from hopf}, one can easily reconstruct the inverse matrix $D={\sf C}^{-1}$, and we list part of it in Table \ref{tab 4} (see Examples \ref{ex: 1and32} and \ref{d's} for more computations). 

Note that by Corollary \ref{cor: stability} this determines the coefficients $c_{\lambda,\mu}$ and $d_{\lambda,\mu}$ for $\lambda\subset\mu\subset (3,3)$ and arbitrary $N$.

\subsection{Link invariants for $\fgl_2$}
\label{sec: figure 8}

We can use the interpolation tables to expand the invariants of simple knots in the basis $F_{\lambda}$. Indeed, the colored $\fgl_2$ invariants are determined by the colored $\fsl_2$ invariants (that is, colored Jones polynomial) by the formula
$$
J_{K}(V(\lambda_1,\lambda_2),q)=J_{K}(V_{\lambda_1-\lambda_2},q).
$$
The coefficients $a_{\lambda}(K)$ are then determined by Theorem \ref{thm:main}
$$
a_{\lambda}(K)=\sum_{\mu\subset \lambda}d_{\lambda,\mu}(q^{-1})J_K(V(\mu),q).
$$
For example, for the figure eight knot we have the following values of the colored Jones polynomial:
$$
J_{K}(V_{0},q)=1=J_{K}(V({1,1}),q),\ J_{K}(V_{1},q)=J_{K}(V({2,1}),q)=1+q^2+q^{-2}-q-q^{-1},
$$
$$
J_K(V_{2},q)=1+q^3+q^{-3}-q-q^{-1}+(q^3+q^{-3}-q-q^{-1})(q^3+q^{-3}-q^2-q^{-2}).
$$
Using the values of $d_{\lambda,\mu}$ from Table \ref{tab 4} (and changing $q$ to $q^{-1}$) we obtain
$$
a_{0}(K)=-J_K(V_{0},q)=-1,\ a_{1}(K)=-\frac{q^{-1}}{1-q^{-1}}J_K(V_{0},q)+\frac{q^{-1}}{1-q^{-1}}J_K(V_{1},q)=q^{-2}(q^3-1),
$$
\begin{multline*}
a_{2}(K)=-\frac{q^{-2}}{(1-q^{-1})(1-q^{-2})}J_K(V_{0},q)+\frac{q^{-2}}{(1-q^{-1})^2}J_K(V_{1},q)-\\
\frac{q^{-3}}{(1-q^{-1})(1-q^{-2})}J_K(V_{2},q)=
q^{-6}(-q^9 + q^5 + q^4 - q^3 - 1),
\end{multline*}
\begin{multline*}
a_{1,1}(K)=-\frac{q^{-3}}{(1-q^{-1})(1-q^{-2})}J_K(V_{0},q)+\frac{q^{-2}}{(1-q^{-1})^2}J_K(V_{1},q)-\\
\frac{q^{-2}}{(1-q^{-1})(1-q^{-2})}J_K(V({1,1}),q)=q^{-2}(q^2 + q + 1),
\end{multline*}
\begin{multline*}
a_{2,1}(K)=-\frac{q^{-4}}{(1-q^{-1})^2(1-q^{-3})}J_K(V_{0},q)+\frac{q^{-3}}{(1-q^{-1})^3}J_K(V_{1},q)-\\
\frac{q^{-4}}{(1-q^{-1})^2(1-q^{-2})}J_K(V_{2},q)-\frac{q^{-3}}{(1-q^{-1})^2(1-q^{-2})}J_K(V({1,1}),q)+\\
\frac{q^{-4}}{(1-q^{-1})^2(1-q^{-3})}J_K(V({2,1}),q)=q^{-6}(-q^8 - q^7 - q^6 - q^5 + q^4 + 2q^3 + q^2 + q + 1).
\end{multline*}
Using Tables \ref{tab 1}, \ref{tab 2}, \ref{tab 3}, one can similarly compute the values of $a_{\lambda}(K)$ for all $\lambda\subset (3,3)$ and verify that these are indeed Laurent polynomials in $q$.

%\newpage
%\begin{landscape} 
\begin{table}[ht!]
\begin{tabular} {c ||c|c|c|c|}
$\lambda\backslash\mu$ & (0) & (1) & (2) & (1,1)  \\
\hline
(0) & $-1$ & $-1$ & $-1$ & $-1$ \\
\hline
(1) & $0$ &   $q^{-1}(1-q)$    & $q^{-2}(1-q^2)$ &    $q^{-1}(1-q^2)$     \\
\hline
(2) & $0$  &     $0$            &   $-q^{-3}(1-q)(1-q^2)$ &     $0$  \\    
\hline
(1,1) &  $0$    &     $0$    &   $0$    &  $-q^{-2}(1-q)(1-q^2)$ \\                                                                                    
\hline
(3) & $0$   &  $0$    & $0$  & $0$    \\                   
\hline
(2,1) & $0$ &   $0$    & $0$    &    $0$   \\
\hline                                                    
(3,1) & $0$ & $0$    & $0$   & $0$     \\
\hline
(2,2) & $0$ & $0$    & $0$   & $0$   \\
\hline
(3,2) & $0$ & $0$    & $0$   & $0$     \\
\hline
(3,3) & $0$ & $0$    & $0$   & $0$    \\
\hline
\end{tabular}
\caption{Evaluations of interpolation polynomials: matrix $C=(c_{\lambda,\mu})$}
\label{tab 1}
\end{table}

\begin{table}[ht!]

\resizebox{\textwidth}{!}{
\begin{tabular}{c ||c|c|c|c|}
$\lambda\backslash \mu$ &  (3) & (2,1) & (3,1)\\
\hline
(0)   & $-1$ & $-1$ & $-1$\\
\hline
(1)     & $q^{-3}(1-q^3)$    & $ q^{-2}(1-q^3) $ &  $q^{-3}(1-q^4)$ \\
\hline
(2)  &   $-q^{-5}(1-q^2)(1-q^3)$ &  $-q^{-3}(1-q)(1-q^3)$ & $ -q^{-5}(1-q^2)(1-q^4)$ \\    
\hline
(1,1)   & $0$ &   $-q^{-3}(1-q)(1-q^3)$ & $-q^{-4}(1-q)(1-q^4)$  \\                                                                                    
\hline
(3)    & $q^{-6}(1-q)(1-q^2)(1-q^3)$ &     $0$ & $q^{-6}(1-q)(1-q^2)(1-q^4)$\\                   
\hline
(2,1)   &   $0$  &  $q^{-4}(1-q)^2(1-q^3)$ & $q^{-6}(1-q)(1-q^2)(1-q^4)$\\
\hline                                                    
(3,1)   & $0$   & $0$ & $-q^{-7}(1-q)^2(1-q^2)(1-q^4)$ \\
\hline
(2,2)   & $0$   & $0$ & $0$\\
\hline
(3,2)  & $0$   & $0$ & $0$\\
\hline
(3,3)    & $0$   & $0$ & $0$ \\
\hline
\end{tabular}
}
\caption{ Matrix $C=(c_{\lambda,\mu})$, continued}
\label{tab 2}
\end{table}

\begin{table}[ht!]
\centering
\resizebox{\textwidth}{!}{
\begin{tabular}{c ||c|c|c|}
$\lambda\backslash \mu$ &  (2,2) & (3,2) & (3,3)\\
\hline
(0) & $-1$ & $-1$ & $-1$\\
\hline
(1) & $q^{-2}(1+q)(1-q^2)$ & $\mathbf{q^{-3}(-q^4 - q^3 + q^2 + 1)}$ & $q^{-3}(1+q)(1-q^3)$\\
\hline
(2) & $-q^3(1-q^2)(1-q^3)$ &  $-q^{-5}(1-q^3)(1-q^4)$ &  $-q^{-5}(1+q)(1-q^3)^2$\\
\hline
(1,1) & $-q^{-4}(1-q^2)(1-q^3)$  &   $-q^{-5}(1-q^2)(1-q^4)$ &  $-q^{-6}(1-q^3)(1-q^4)$\\
\hline
(3) & $0$    &  $q^{-6}(1-q)(1-q^3)(1-q^4)$              &       $q^{-6}(1-q^2)(1-q^3)(1-q^4)$\\
\hline
(2,1) & $q^{-5}(1-q^2)^2(1-q^3)$                    &    $q^{-7}(1-q^2)(1-q^3)(1-q^4)$  &            $q^{-8}(1+q)(1-q^2)(1-q^3)(1-q^4)$\\
\hline
(3,1) & $0$     & $-q^{-8}(1-q)(1-q^2)(1-q^3)(1-q^4)$    &   $-q^{-9}(1-q^2)(1-q^3)^2(1-q^4)$\\
\hline
(2,2) & $-q^{-6}(1-q)(1-q^2)^2(1-q^3)$        & $-q^{-8}(1-q)(1-q^2)(1-q^3)(1-q^4)$    & $-q^{-10}(1-q^2)(1-q^3)^2(1-q^4)$ \\
\hline
(3,2) & $0$ &    $q^{-9}(1-q)^2(1-q^2)(1-q^3)(1-q^4)$  &  $q^{-11}(1-q^2)^2(1-q^3)^2(1-q^4)$\\
\hline
(3,3) & $0$ & $0$ & $-q^{-12}(1-q)(1-q^2)^2(1-q^3)^2(1-q^4)$\\
\hline
\end{tabular}
 }
\caption{Matrix $C=(c_{\lambda,\mu})$, continued}
\label{tab 3}
\end{table}

\begin{table}
\centering
%\resizebox{\textwidth}{!}{
\begin{tabular}{c ||c|c|c|c|c|}
$\lambda\backslash\mu$ & (0) & (1) & (2) & (1,1) & (2,1) \\
\hline
(0) &   $-1$ & $-\frac{q}{1-q}$ & $-\frac{q^2}{(1-q)(1-q^2)}$ & $-\frac{q^3}{(1-q)(1-q^2)}$ & $-\frac{q^4}{(1-q)^2(1-q^3)}$ \\
\hline
(1) & $0$ & $\frac{q}{1-q}$ & $\frac{q^2}{(1-q)^2}$ & $\frac{q^2}{(1-q)^2}$ & $\frac{q^3}{(1-q)^3}$\\
\hline
(2) & $0$ & $0$ & $-\frac{q^3}{(1-q)(1-q^2)}$ & $0$ & $-\frac{q^4}{(1-q)^2(1-q^2)}$\\
\hline
(1,1) & $0$ & $0$ & $0$ &   $-\frac{q^2}{(1-q)(1-q^2)}$ & $-\frac{q^3}{(1-q)^2(1-q^2)}$ \\ 
\hline
(2,1) & $0$ & $0$ & $0$ & $0$ & $\frac{q^4}{(1-q)^2(1-q^3)}$  \\
\hline
\end{tabular}
%}
\caption{Interpolation matrix $D=(d_{\lambda,\mu})=C^{-1}$}
\label{tab 4}
\end{table}

\section{Appendix}
Here we collect some useful definitions and  facts about Habiro's ring
and interpolation Macdonald polynomials.

\subsection{Habiro's ring}

%In this section we collect some  facts about the Habiro ring $\Ah$. 
%Most of them can be found in.
The Habiro ring  \cite{Hpoly} is defined as 
$$ \widehat{\Z[q]}:=
{\lim\limits_{\overleftarrow{\hspace{2mm}n\hspace{2mm}}}}\;   
\frac{\Z[q]}{((q;q)_n)}$$
 Any element of $\widehat{\Z[q]}$ can be presented (not uniquely)
as infinite series 
%{(\color{blue} the more standard notation for Habiro's ring is $\widehat{\Z[q]}$, I didn't replace all occurences, we can also use your macro and decide at the end)}
$$
f(q)=\sum_{n=0}^{\infty} f_n \,(q;q)_n, \quad f_n\in \Z[q].
$$ 
Evaluations of such $f(q)$ at all roots of unity are well defined, since if
$q^s=1$ one has
$f(q)=\sum_{n=0}^{s-1} f_n (q)_n$.
It is easy to expand every $f(q) \in \widehat{\Z[q]}$ into formal power series in $(q-1)$, denoted by $T(f)$ and called the Taylor series of $f(q)$ at $q = 1$. One
important property of the Habiro ring
 is that any $f \in \widehat{\Z[q]}$ is uniquely determined by its Taylor 
series. In other words, the map $T :  \widehat{\Z[q]} \to \Z[[q - 1]] $ is injective \cite[Thm 5.4]{Hpoly}. 
In particular, $\widehat{\Z[q]}$
is an integral domain. Moreover, every $f \in \widehat{\Z[q]}$ is determined by the values of $f$ at any infinite set of roots of unity of prime power order. 
Because of these properties, Habiro ring is also known as  a ring of analytic functions at roots of unity.
\vspace{1mm}

%It is known that $\Ah$ is an integral domain \cite[Corollary 5.5]{Hpoly} and 

Since $\cap_{n\geq 0} I_n=0$ with $I_n=(q;q)_n\Z[q]$,
  the natural map $\Z[q]\to \widehat{\Z[q]}$ is injective. The image of $q$ under this map is invertible, and the inverse is given by
$$
q^{-1}=\sum_{n=1}^{\infty} q^n (q;q)_n,
$$
compare with Example \ref{ex: x inverse}. This implies that there is an injective map $\Z[q,q^{-1}]\to \widehat{\Z[q]}$. The following result is proved in \cite[Proposition 7.5]{Hpoly}, but we give a slightly different proof here for the reader's convenience. We will denote by $\Phi_n(q)$ the $n$th cyclotomic polynomial $\Phi_n(q)=\prod_{(a,n)=1}\left(q-\zeta^a_{n}\right)$ where
$\zeta_n$ is any primitive $n$th root of unity.

\begin{proposition}
\label{prop: appendix}
Suppose that $f(q)\in \widehat{\Z[q]}$ and $f(q)h(q)\in \Z[q,q^{-1}]$ for some product of cyclotomic polynomials $h(q)=\Phi_{n_1}(q)\cdots \Phi_{n_r}(q)$. Then $f(q)\in \Z[q,q^{-1}]$. 
\end{proposition}

\begin{proof}
Let us denote $g(q)=f(q)h(q)\in \Z[q,q^{-1}]$, we prove the statement by induction in $r$. For $r=1$ we get $h(q)=\Phi_n(q)$ and 
$g(q)=f(q)\Phi_n(q)$, so for any primitive $n$-th root of unity $\zeta_n$ we have $g(\zeta_n)=f(\zeta_n)\Phi_n(\zeta_n)=0$, so 
$g(q)=\alpha(q)\Phi_n(q)$ for some $\alpha\in \Z[q,q^{-1}]$. This implies $(f(q)-\alpha(q))\Phi_n(q)=0$, and since $\widehat{\Z[q]}$ is an integral domain we get $f(q)=\alpha(q)$.

For $r>1$ we get 
$$f(q)\Phi_{n_1}(q)\cdots \Phi_{n_r}(q)\in \Z[q,q^{-1}],$$ so by the above 
$$f(q)\Phi_{n_1}(q)\cdots \Phi_{n_{r-1}}(q)\in \Z[q,q^{-1}],$$ and by the assumption of induction $f(q)\in \Z[q,q^{-1}]$.
\end{proof}

\subsection{Interpolation Macdonald polynomials}

We consider partitions with at most $N$ parts. 

\begin{theorem}\cite{Knop,KS,Ok,Ok2,Ok3,Ols,Sahi}
There exists unique up to scalar factors family of symmetric polynomials $I_{\lambda}(x_1,\ldots,x_N;q,t)$ with the following properties:
\begin{itemize}
\item[(a)] $I_{\lambda}(q^{-\mu_i}t^{N-i})=0$ unless $\mu$ contains $\lambda$
\item[(b)] $I_{\lambda}(q^{-\lambda_i}t^{N-i})\neq 0$
\item[(c)] $I_{\lambda}$ is a nonhomogeneous polynomial of degree $|\lambda|$, and its degree $|\lambda|$ part is proportional to the Macdonald polynomial $P_{\lambda}(x_1,\ldots,x_N;q,t)$.
\end{itemize}
\end{theorem}

The polynomials $I_{\lambda}$ are called interpolation Macdonald polynomials. 
In fact, the properties (a) and (b) already uniquely determine $I_{\lambda}$ (up to a scalar), and their existence follows from the fact that $q^{-\lambda_i}t^{N-i}$ for a nondegenerate grid in the sense of \cite{Ok3}. Part (c) is then a deep property of these polynomials.

It is easy to see that at $q=t$ interpolation Macdonald polynomials $I_{\lambda}$ specialize to $F_{\lambda}$. Unlike $F_{\lambda}$, there is no determinant formula for $I_{\lambda}$ but there is a different combinatorial formula \cite{Ok2}.

%\end{appendix}

\end{document}